\documentclass[12pt]{amsart}
\usepackage{amsmath,amsthm,amsfonts,amssymb}
\usepackage{graphicx}
\usepackage{mathrsfs}
\usepackage{enumitem}
\usepackage{etoolbox}
\usepackage{xcolor}
\apptocmd{\sloppy}{\hbadness 10000\relax}{}{}
\apptocmd{\sloppy}{\vbadness 10000\relax}{}{}
\usepackage[pdfpagelabels]{hyperref}
\usepackage[letterpaper,margin=1.1in, margin=1in]{geometry}

\usepackage[margin=1in]{geometry}  
\usepackage{esint}
\usepackage{hyperref}

\newtheorem{thm}{Theorem}[section]
\newtheorem{lem}[thm]{Lemma}
\newtheorem{prop}[thm]{Proposition}
\newtheorem{cor}[thm]{Corollary}

\newtheorem{rmk}[thm]{Remark}
\newtheorem{definition}[thm]{Definition}

\renewenvironment{abstract}{%
  \noindent\bfseries\abstractname:\normalfont}{}

\numberwithin{equation}{section}

\newcommand{\RR}{\mathbb{R}}      

\begin{document}

\title{The Singular Strata of a Free-Boundary problem for harmonic measure}

\author{Sean McCurdy}
\maketitle

\begin{abstract}
  In this paper, we obtain \textit{quantitative} estimates on the fine structure of the singular set of the mutual boundary $\partial \Omega^{\pm}$ for pairs of complementary domains, $\Omega^+, \Omega^- \subset \RR^n$ which arise in a class of two-sided free boundary problems for harmonic measure.  These estimates give new insight into the structure of the mutual boundary, $\partial \Omega^{\pm}.$ 
\end{abstract}

\tableofcontents

\section{Introduction}\label{S:intro}

The focus of this paper\footnote{This research was partially supported by NSF grant DMS-1361823.}
 is the study of a class of two-phase free boundary problem for harmonic measure.  Let $n \ge 3$, $\Omega^+ \subset \RR^n$ and $\Omega^- = \overline{\Omega^+}^c$ be unbounded NTA domains (see Definition \ref{NTA def}), $\omega^{\pm}$ their associated harmonic measures, and $u^{\pm}$ associated Green's functions with poles at infinity.  Let $\omega^- \ll \omega^+ \ll \omega^-$ and $h= \frac{d\omega^-}{d\omega^+}$ satisfy $\ln(h) \in C^{0, \alpha}$ for  some $0< \alpha < 1$.  We obtain new results on the structure of the geometric singular set of the boundary, $\partial \Omega^{\pm}$.

This problem was introduced without the regularity assumption on $\omega^{\pm}$ by Kenig, Preiss, and Toro \cite{KenigPreissToro09}, with other work under the assumption that $\ln(h) \in VMO(\partial \Omega^{\pm})$ by Kenig and Toro \cite{KenigToro06}, Badger \cite{Badger11} \cite{Badger13}, and Badger, Engelstein, and Toro \cite{BadgerEngelsteinToro17}.    Questions about the structure of the free boundary and the singular set when $\ln(h) \in C^{0, \alpha}$ for $0< \alpha < 1$ have been addressed by Engelstein \cite{Engelstein16} and Badger, Engelstein, Toro \cite{BadgerEngelsteinToro20}, respectively. In \cite{Engelstein16}, Engelstein shows that under the additional assumption that the boundary is sufficiently flat in the sense of Reifenberg, the boundary is locally $C^{1, \alpha}$.  In \cite{BadgerEngelsteinToro20}, the authors remove the assumption of flatness and prove that the geometric singular set is contained in countably many $C^{1, \beta}$ submanifolds of the appropriate dimension.  See \cite{KenigPreissToro09} for an overview of this problem in lower dimensions, and \cite{BadgerEngelsteinToro17,BadgerEngelsteinToro20} for further background.

Until recently, almost all work on the two-sided free boundary problem for harmonic measure in higher dimensions has operated under the assumption that $\Omega^{\pm}$ are NTA domains because the NTA conditions allow for scale-invariant estimates of harmonic measure.  However, Azzam, Mourgoglou, Tolsa, and Volberg \cite{amtv-twophase16} proved, among other things, that if we relax the assumption that the domains are NTA, then $\omega^- \ll \omega^+ \ll \omega^-$ on $G \subset \partial \Omega^{\pm}$ implies that $G$ can be decomposed into $G = R \cup B$, where $R$ is $(n-1)$-rectifiable and $\omega^{\pm}(B) = 0$.  However, we shall work under the assumption that $\Omega^{\pm}$ are NTA domains.

Based upon \cite{BadgerEngelsteinToro20}, we know that when $\ln(h) \in C^{0, \alpha}$, the singular set of $\partial \Omega^{\pm}$ is countably $C^{1, \beta}$-rectifiable.  This leaves open the question of whether or not the singular set is dense, or more generally how it sits in space.  In this paper, we answer the question of how the singular set ``sits in space."  In particular, we provide upper Minkowski content bounds upon the quantitative strata of the singular set (see Theorem \ref{main theorem}). The main approach will be to follow \cite{Engelstein16} and consider jump functions $v = u^+ - u^-$ which are \textit{almost} harmonic and employ the Almgren frequency function and geometric techniques as in \cite{HanLin_nodalsets, CheegerNaberValtorta15} in conjunction with the powerful Federer dimension-reduction techniques of \cite{NaberValtorta17-1,deLellisMarcheseSpadaroValtorta16}.  While these tools are common for problems in Calculus of Variations, it is important to note that the jump functions $v$ are not minimizers of any energy, nor do they satisfy any global PDE.  

The author would like to thank T. Toro, whose generous insight, suggestions, and support were essential to bringing this project to fruition. The author is also grateful to the reviewers for their many excellent comments which dramatically improved this paper.

\section{Definitions and Statement of Main Results}\label{S:defs}

\subsection{Domains and their Green's functions.}

Non-tangentially accessible (NTA) domains were formally introduced by Jerison and Kenig in \cite{JerisonKenig82} to study the boundary behavior of PDEs on non-smooth domains.  The definition is given, below.

\begin{definition}\label{NTA def}
A domain $\Omega \subset \RR^n$ is a non-tangentially accessible (NTA) domain if there exist constants $M > 1, R_0 > 0$ such that the following holds:

1. $\Omega$ satisfies the \emph{corkscrew condition}.  That is, for any $Q \in \partial \Omega$ and $0< r< R_0$, there exists a point $A_r(Q) \in \Omega$ with the following two properties:
$$|A_r(Q) - Q| < r \quad \text{and} \quad B_{\frac{r}{M}}(A_r(Q)) \subset \Omega.$$ 

2. $\overline{\Omega}^c $ also satisfies the \emph{corkscrew condition}.

3. $\Omega$ satisfies the \emph{Harnack Chain condition}.  That is, for any $\epsilon > 0$ and $Q \in \partial \Omega$, if $x_1, x_2 \in \Omega \cap B_{\frac{R_0}{4}}(Q) \setminus B_{\epsilon}(\partial \Omega)$ and $|x_1-x_2| \le 2^k \epsilon$, then there exists a ``Harnack chain" of balls $\{B_{r_i}(y_i)\}_{i=1}^N$ satisfying
\begin{enumerate}
    \item $x_1 \in B_{r_i}(y_1)$ and $x_1 \in B_{r_N}(y_N).$
    \item For all $i =1 , ..., N$, $B_{r_i}(y_i) \subset \Omega$.
    \item For all $i= 1, ..., N-1$, $B_{r_i}(y_i) \cap B_{r_{i+1}}(y_{i+1}) \not = \emptyset$.
    \item $N \le Mk$.
    \item For all $i = 1, ..., N$ $$\frac{1}{2M}min_{1,2}\{\text{dist}(x_i, \partial \Omega) \} \le r_i \le \text{dist}(y_i, \partial \Omega).$$
\end{enumerate}
Note that by increasing the radii if necessary, we may assume that $r_i \sim_M \text{dist}(y_i, \partial \Omega).$ 

We say that $\Omega^+$ is a \emph{two-sided NTA domain} if both $\Omega^+$ and $\Omega^- := \overline{\Omega}^c$ are NTA domains.  We shall refer to the complementary pair $\Omega^{\pm}$ of domains as complementary two-sided NTA domains and denote their mutual boundary by $\partial \Omega^{\pm}.$
\end{definition}

In this paper, we shall only deal with unbounded, two-sided NTA domains. That is, we shall assume that $R_0=\infty$.  However, the results are essentially local.

\begin{definition}
For a $\Omega^{\pm} \subset \mathbb{R}^n$ two sided NTA domains, we shall use 
$$u^{\pm}$$
to denote the \emph{Green's function with pole at infinity} corresponding to $\Omega^{\pm}$, respectively.  

Recall that $u^{\pm}$ are unique up to scalar multiplication and that to each $u^{\pm}$ is associated the \emph{harmonic measure} 
$$\omega^{\pm},$$ defined by the property that for all $\phi \in C^{\infty}_c(\RR^n)$
$$
\int  \Delta \phi u^{\pm} dV = \int \phi d\omega^{\pm}.
$$
See \cite{GarnettMarshall} for more details about harmonic measures.
\end{definition}

Observe that if $\omega^+$ is the harmonic measure associated to $u^+$, then $c\omega^{+}$ is the harmonic measure associated to $cu^+$ for any $c >0.$ 

If $f \in C^{0, \alpha}(\RR^n)$, we shall use $||f||_{\alpha}$ to denote the local norm,
\begin{align*}
||f||_{\alpha} := \sup_{B_2(0)}|f| + \sup_{x \not= y \in B_2(0)}\frac{|f(x) - f(y)|}{|x - y|^{\alpha}}.
\end{align*}

\begin{definition}\label{domain class def}  We define the class $\mathcal{D}(n, \alpha, M_0)$ to be the collection of domains $\Omega^{\pm} \subset \RR^n$ such that $\Omega^{\pm}$ are complementary, unbounded two-sided NTA domains for which $M < M_0$,  $\omega^- << \omega^+ << \omega^-$, the Radon-Nikodym derivative $h = \frac{d\omega^-}{d\omega^+}$ satisfies $\ln(h) \in C^{0, \alpha}(\partial \Omega)$, and $0 \in \partial \Omega^{\pm}$.
\end{definition}

Note that if $\Omega^{\pm} \in \mathcal{D}(n, \alpha, M_0)$ and $Q \in \partial \Omega^{\pm}$, then $\Omega^{\pm} - Q \in \mathcal{D}(n, \alpha, M_0)$.

\subsection{A class of functions and their rescalings}

\begin{definition}\label{align rmk}
Let $\Omega^{\pm} \subset \mathbb{R}^n$ be a pair of complementary two-sided NTA domains with mutual boundary $\partial \Omega^{\pm}$.  For any $Q \in \partial \Omega^{\pm}$ and any Green's functions, $u^{\pm}$, we define the \emph{jump function}
\begin{align}
    v^{Q}(x) := h(Q)u^+(x) - u^-(x).
\end{align}
\end{definition}

The scaling, $h(Q)u^+$, normalizes the Radon-Nikodym derivative of the harmonic measure associated to $h(Q)u^+$ and $u^-$ at $Q \in \partial \Omega^{\pm}$.

\begin{definition} \label{A def}
Let $\mathcal{A}(\Lambda, \alpha, M_0)$ be the set of functions, $v: \RR^n \rightarrow \RR$ which have the following properties:
\begin{enumerate}
\item $v := v^0 = h(0)u^+ - u^-$ where $u^{\pm}$ are Green's functions with poles at infinity associated to a two-sided NTA domain, $\Omega^{\pm} \in \mathcal{D}(n, \alpha, M_0)$ and $h= \frac{d\omega^-}{d\omega^+}$, where $\omega^{\pm}$ are the harmonic measures associated to $u^{\pm}$. 
\item We make the specific choice of $u^{\pm}$ such that $h(0) = 1$.
\item The Almgren frequency function satisfies  $N(0, 1, v) \le \Lambda$.
\end{enumerate}
See Definition \ref{N W def} for the definition of the Almgren frequency function.
\end{definition} 

\begin{rmk}
Observe that for any fixed, $\Omega^{\pm} \in \mathcal{D}(n, \alpha, M_0)$, there is a one-parameter family of associated functions, $v \in \mathcal{A}(\Lambda, \alpha, M_0)$ with $\{v = 0\} = \partial \Omega^{\pm}$.  Indeed, if $v \in \mathcal{A}(\Lambda, \alpha, M_0)$, then $cv \in \mathcal{A}(\Lambda, \alpha, M_0)$ for any $c>0$.  This degree of freedom comes from the non-uniqueness of the Green's function with pole at infinity (see Remark \ref{align rmk}). To avoid degeneracy because of this degree of freedom, in the arguments that follow we must normalize our functions.
\end{rmk}

\begin{definition}\label{harmonic rescaling def}  Let $\Omega^{\pm} \in \mathcal{D}(n, \alpha, M_0)$ and $Q \in \partial \Omega^{\pm}$.  For scales $0< r$, we define the rescaling of the function $v^Q$ to scale $r$ at the point $Q' \in \partial \Omega^{\pm}$ by
\begin{align*}
v^Q_{Q', r}(x) := v^Q(rx +Q')\frac{r^{n-2}}{\omega^-(B_r(Q'))}
\end{align*}
and the corresponding rescaled measure as
\begin{align}
\omega^{\pm}_{Q', r}(E) := \frac{\omega^{\pm}(rE + Q')}{\omega^{\pm}(B_r(Q'))}.
\end{align}
\end{definition}

These rescalings $v^Q_{Q, r}$ were first introduced by Kenig and Toro in \cite{KenigToro06}.  In this paper, we shall employ the following results by Kenig, Toro, Badger, and Engelstein.

\begin{thm}\label{TBE combo}(\cite{KenigToro06}, \cite{Badger11}, \cite{Engelstein16})
For $v^Q_{Q, r}$ and $\omega^{\pm}_{Q, r}$ as in Definition \ref{harmonic rescaling def}, 
\begin{enumerate}
\item Subsequential limits as $r \rightarrow 0$ of the functions $v^Q_{Q, r}$ converge to harmonic polynomials.  Furthermore, the degree of these polynomials is bounded, depending only upon the NTA constant, $M_0$. \cite{KenigToro06}
\item Subsequential limits as $r \rightarrow 0$ of the functions $v^Q_{Q, r}$ converge to \textit{homogeneous} harmonic polynomials.  Furthermore, the degree of homogeneity is unique along blow-ups.  \cite{Badger11}
\item The $v^Q_{Q, r}$ are uniformly locally Lipschitz with Lipschitz constant that only depends upon $M_0.$ \cite{Engelstein16}
\item The $\omega^{\pm}_{Q, r}$ are locally uniformly bounded. \cite{Engelstein16}. 
\end{enumerate}
\end{thm}

In addition to the $v_{Q',r}^Q$ rescalings, we shall also use a different kind of rescaling.
  
\begin{definition} (\cite{CheegerNaberValtorta15})\label{L2 rescaling def}
Let $f:B_1(0) \rightarrow \mathbb{R}$ be a function in  $C(\RR^n)$.  We define the rescaled function, $T_{x, r}f$ of $f$ at a point $x \in B_{1-r}(0)$ at scale $0<r<1$ by
\begin{align}
    T_{x, r}f(y) := \frac{f(x +ry) - f(x)}{\left(\fint_{\partial B_1(0)} (f(x + ry) -f(x))^2d\sigma(y)\right)^{1/2}}.
\end{align} 
In the case that the denominator is zero, we define $T_{x, r}f = \infty$. We denote the limit as $r \rightarrow 0$ by $$T_xf(y) := \lim_{r \rightarrow 0}T_{x, r}f(y).$$ 
 \end{definition}

\subsection{Quantitative symmetry.}

The geometry we wish to capture with the blow-ups $T_xf$ is encoded in their translational symmetries.  
 \begin{definition}\label{symmetric def}
Let $f: \RR^n \rightarrow \mathbb{R}$ be a continuous function. We say $f$ is $0$-\emph{symmetric} if \begin{align}
f(x):= cP^+(x) - P^-(x)
\end{align}
for some $c >0$, where $P^{\pm}$ are the positive and negative parts of a homogeneous harmonic polynomial $P$.  We will say that $f$ is $k$-\emph{symmetric} if $f$ is $0$-symmetric and there exists a $k$-dimensional subspace, V, such that $f(x+y) =f(x)$ for all $x \in \RR^n$ and all $y \in V$.
\end{definition}
 
The constant, $c>0$, is there to allow for the function to ``hinge" along its zero set.  We must allow this kind of ``hinging" to accommodate for the ``non-alignment" issue in the blow-ups at $Q \in \partial \Omega^{\pm} \setminus \{0\}$.  See Remark \ref{hinged blow-ups}. 

We now define a quantitative version of symmetry.

\begin{definition} \label{quant symmetric def}
For any $f \in C(\RR^n)$, $f$ will be called $(k, \epsilon, r, p)$-\emph{symmetric} if there exists a $k$-symmetric function, $P,$ such that, 
 \begin{itemize}
 \item[1.] $\fint_{\partial B_1(0)} |P|^2 dV= 1$
 \item[2.] $\fint_{B_1(0)}|T_{p, r}f - P|^2dV < \epsilon.$
 \end{itemize}
Sometimes, we shall refer to a function $f$ as being $(k, \epsilon)$-symmetric in the ball $B_r(p)$ to mean $f$ is $(k, \epsilon, r, p)$-symmetric.
 \end{definition}
 
This quantitative control allows us to define a quantitative stratification following \cite{CheegerNaber13}.
 
 \begin{definition}(Quantitative Singular Strata)\label{quant strat def}
 Let $v \in \mathcal{A}(\Lambda, \alpha, M_0)$.  We denote the $(k, \epsilon, r)$-\emph{singular stratum} of $v$ by $\mathcal{S}^k_{\epsilon, r}(v)$, and we define it by 
\begin{equation}
\mathcal{S}^k_{\epsilon, r}(v) := \{x \in \partial \Omega^{\pm} : v \text{  is not  } (k+1, \epsilon, s, x) \text{-symmetric for all  } s \ge r\}.
\end{equation}
We shall also use the notation $\mathcal{S}^k_{\epsilon}(v)$ for $\mathcal{S}^k_{\epsilon, 0}(v).$ It is immediate from the definitions that $\mathcal{S}^k_{\epsilon, r}(v) \subset \mathcal{S}^{k'}_{\epsilon', r'}(v)$ if $k \le k', \epsilon' \le \epsilon, r \le r'$.  

We can recover the \emph{qualitative} stratification
$$\mathcal{S}^k(v) := \{x \in \partial \Omega^{\pm} : T_xv \text{  is not  } (k+1)\text{-symmetric}\} = \cup_{\eta} \cap_{r} \mathcal{S}^k_{\eta, r}(v). $$
The set $\mathcal{S}^k(v)$ is the $k^{th}$ stratum of $\mathcal{S}^{n-2}(v)$.  Furthermore, if $x \in \mathcal{S}^k(v),$ then there exists an $0 < \epsilon$ such that $x \in \mathcal{S}^k_{\epsilon}(v).$  
\end{definition}

\begin{rmk}
Note that the singular set and its strata are all stable under scalar multiplication.  That is
\begin{align*}
    \mathcal{S}^k(v) & = \mathcal{S}^k(cv)\\
    \mathcal{S}^k_{\epsilon, r}(v) & = \mathcal{S}^k_{\epsilon, r}(cv)
\end{align*}
for all $c \not = 0.$ In particular, for all $v \in \mathcal{A}(\Lambda, \alpha, M_0)$, $\mathcal{S}^k_{\epsilon, r}(v) = \mathcal{S}^k_{\epsilon, r}(v^Q)$ for all $Q \in \partial \Omega^{\pm}$.
\end{rmk}

Previous results on the singular set are summed up in the following theorem.

\begin{thm}\label{flat implies smooth}(\cite{BadgerEngelsteinToro17} \cite{BadgerEngelsteinToro20}, \cite{Engelstein16})
For  $v \in \mathcal{A}(\Lambda, \alpha, M_0)$, the following hold: 
\begin{enumerate}
\item $\left(\mathcal{S}^{n-2}(v)\setminus \mathcal{S}^{n-3}(v) \right) \cap \partial \Omega^{\pm} = \emptyset$, (\cite{BadgerEngelsteinToro20}, Remark 7.2).
\item There exists an $\epsilon >0$ such that 
$\emph{sing}(\partial \Omega^{\pm}) = \mathcal{S}^{n-3}(v) \cap \partial \Omega^{\pm} \subset \mathcal{S}^{n-2}_{\epsilon}(v)$, (\cite{Engelstein16}, Theorem 1.1).
\item  $\overline{\dim_{\mathcal{M}}}(\emph{sing}(\partial \Omega^{\pm})) \le n-3$ (\cite{BadgerEngelsteinToro17}, Theorem 7.5). 
\end{enumerate}
\end{thm}

\subsection{Main Results and Outline of the Proof}\label{S:main theorem}

In this paper, we prove volume bounds on tubular neighborhoods around the $S^k_{\epsilon, r}(v)$.  We are able to show the following estimates.

\begin{thm}\label{main theorem}
Let $v \in \mathcal{A}(\Lambda, \alpha, M_0)$ with $||\ln(h)||_{\alpha} \le \Gamma.$  For every $0 < \epsilon$ and $0 \le k \le n-2$ there is an $0<r_0(n, \Lambda, \alpha, M_0, \Gamma, \epsilon)$ such that for all $0< r <r_0$, and any $r \le R \le 1$ 
\begin{equation}
\text{Vol}\left(B_{R}(B_{1/4}(0) \cap \mathcal{S}^k_{\epsilon, r}(v))\right) \le C(n, \Lambda, \alpha, M_0, \Gamma, \epsilon) R^{n-k}.
\end{equation}
\end{thm}

We have the following immediate corollary.

\begin{cor}\label{cor 3.2}
Let $v \in \mathcal{A}(\Lambda, \alpha, M_0)$ and $0 \le k \le n-2$.  For every $0 < \epsilon$
\begin{align}
\overline{\dim}_{\mathcal{M}}(\mathcal{S}^{k}_{\epsilon}(v)) \le k,
\end{align}
and there exists a constant such that 
\begin{equation}
\mathcal{M}^{*, k}( \mathcal{S}^{k}_{\epsilon}(v) \cap B_{1/4}(0)) \le C(n, \Lambda, \alpha, M_0, \Gamma, \epsilon).
\end{equation}
\end{cor}

Due to an $\epsilon$-regularity result due to \cite{Engelstein16}, we are able to strengthen the conclusion of Theorem \ref{main theorem} when we consider the full singular set.  

\begin{cor}\label{cor 3.3}
Let $v \in \mathcal{A}(\Lambda, \alpha, M_0)$ with $||\ln(h)||_{\alpha} \le \Gamma$.  Recall that $\emph{sing}(\partial \Omega^{\pm}) = \mathcal{S}^{n-3} \cap \partial \Omega^{\pm}$.  There exists an $0< \epsilon= \epsilon (M_0, \Gamma, \alpha)$ such that $\emph{sing}(\partial \Omega^{\pm}) \subset \mathcal{S}^{n-3}_{\epsilon}$ (see Lemma \ref{epsilon regularity}).  Thus, there is a constant, $C = C(n, \Lambda, \alpha, M_0, \Gamma) < \infty$ such that
\begin{equation}
\mathcal{M}^{*, n-3}( \emph{sing}(\partial \Omega^{\pm}) \cap B_{1/4}(0)) \le C(n, \Lambda, \alpha, M_0, \Gamma).
\end{equation}
\end{cor}

\begin{proof}
This follows immediately from Lemma \ref{epsilon regularity} and Theorem \ref{main theorem}.
\end{proof}

\subsection{Outline of the Proof of Theorem 3.1}

In order to prove a theorem of this kind, we must build a cover of $\mathcal{S}^k_{\epsilon, r}(v)$, and we must count how many balls we use.  Therefore, two things are critical, getting geometric information about $\mathcal{S}^k_{\epsilon, r}(v)$ and keeping track of the how the balls pack.  

The overall strategy of proof is very similar to that of \cite{EdelenEngelstein17, deLellisMarcheseSpadaroValtorta16}.  However, there are several major differences.  First, the functions we consider, $v \in \mathcal{A}(\Lambda, \alpha, M_0)$, are not harmonic functions or minimizers of an energy.   Sections \ref{S:compactness} -- \ref{S: quant rigidity} are devoted to showing that the relevant analogs of harmonic results (e.g. compactness, almost-monotonicity of the Almgren frequency, local uniform boundedness of the Almgren frequency, quantitative rigidity for the Almgren frequency, cone-splitting, etc.) hold for $v \in \mathcal{A}(\Lambda, \alpha, M_0)$. In particular, we prove an estimate on the non-degeneracy of the almost-monotonicity for Almgren frequency in Lemma \ref{N non-degenerate almost monotonicity}. Local geometric control on $\mathcal{S}^{k}_{\epsilon}(v)$ is obtained in Section \ref{S:dichotomy}. 

However, geometric control is not enough to obtain Theorem \ref{main theorem}. To obtain finite upper Minkowski content bounds we need the Discrete Reifenberg Theorem from \cite{NaberValtorta17-1}(Theorem \ref{discrete reif}). This requires that we prove a ``frequency pinching" result (Lemma \ref{beta bound lem}) in which we connect the drop in the Almgren frequency over small scales with the $\beta$-numbers.  The main challenge is to connect the lower bound on the derivative of the Almgren frequency (Lemma \ref{N non-degenerate almost monotonicity}) and employ the techniques of \cite{deLellisMarcheseSpadaroValtorta16} to obtain the necessary estimates on $N(Q, r, v) - N(Q', r, v)$ (see Section \ref{S:spatial derivatives}).

In Section \ref{S:Packing}, obtain the necessary packing estimates, following the framework of \cite{NaberValtorta17-1} to accommodate the estimates of Section \ref{S: beta numbers}.  Sections \ref{S:trees} and \ref{S:covering}, construct the covering which proves the theorem according to the program laid out by \cite{NaberValtorta17-1}.  They are included for completeness.


\section{Compactness}\label{S:compactness}

The main goal of this section is to show that $\mathcal{A}(\Lambda, \alpha, M_0)$ enjoys sufficient compactness to allow for limit-compactness arguments.  Namely, we wish to establish that for any sequence $v_i \in \mathcal{A}(\Lambda, \alpha, M_0)$, we can extract a subsequence which converges to a function $v_{\infty}$ and that $N(p, r, v_i) \rightarrow N(p, r, v_{\infty})$ (Corollary \ref{N W convergence cor}).  This requires strong convergence in $W^{1, 2}_{loc}(\RR^n)$ (see Lemma \ref{strong compactness I} and Lemma \ref{compactness-ish I}).
 
On a technical level, we must extend the compactness implied by Theorem \ref{TBE combo} for $v^Q_{Q, r}$ to $v^Q_{Q', r}$ and $T_{Q', r}v$. Throughout, we shall make essential use of ``standard NTA results" such as the doubling of harmonic measure and various comparability results, all of which may be found in \cite{JerisonKenig82}.

\begin{rmk}\label{r: omega comparable}
Recall that for $E \subset \partial \Omega^{\pm}$
$$\omega^+(E) = \int \chi_{E}d\omega^+$$ and $$\omega^-(E) = \int \chi_{E}hd\omega^+.$$  Furthermore, if  $\ln(h) \in C^{0, \alpha}$ with $||\ln(h)||_{\alpha} \le \Gamma$, then for all $Q, Q' \in \partial \Omega^{\pm}$
\begin{equation}\label{h comparability}
 e^{-\Gamma|Q-Q'|^{\alpha}}h(Q') \le h(Q) \le e^{\Gamma|Q - Q'|^{\alpha}}h(Q').
 \end{equation}

Using Equation (\ref{h comparability}) in the above integral equations implies that in any compact set $K$, if $v \in \mathcal{A}(\Lambda, \alpha, M_0)$ with $||\ln(h)||_{\alpha} \le \Gamma$, there is a constant, $C(K,\Gamma, \alpha)>1$ such that for any $E \subset K \cap \partial \Omega^{\pm}$
 \begin{equation*}
C^{-1} \le \frac{\omega^-(E)}{\omega^{+}(E)} \le C.
\end{equation*}
\end{rmk}

\begin{rmk}\label{hinged blow-ups}
By Theorem \ref{TBE combo}, we know that subsequential limits as $r \rightarrow 0$ of the functions $v^Q_{Q, r}$ converge to homogeneous harmonic polynomials.  However, for $Q, Q' \in \partial \Omega^{\pm}$ and $Q \not=Q',$ it is not true in general that $v^Q_{Q', r}$ converges to a homogeneous harmonic polynomial.  As $r \rightarrow 0$, the function $v^Q_{Q', r}$ will converge to a $0$-symmetric function (see Definition \ref{symmetric def}) where $c = h(Q)$. 
\end{rmk}

\begin{definition}
We shall abuse the notation of Definition \ref{L2 rescaling def} $T_{Q, r }$ to denote translated and scaled versions of various objects.  For example, for sets this is the usual push-forward  
\begin{align*}
T_{Q, r}\Omega^{\pm} & := \frac{\Omega^{\pm} - Q}{r}\\
T_{Q, r} \partial \Omega^{\pm} & := \frac{\partial \Omega^{\pm} - Q}{r}.
\end{align*}
However, for the measures, $\omega^{\pm}$, $T_{Q, r}\omega^{\pm}$ will denote the harmonic measures associated to the positive and negative parts of $T_{Q, r}v$.  The corkscrew points, $A^{\pm}_R(Q)$ will always denote the corkscrew point associated to $Q$ at scale $R$ in the domain $\Omega^{\pm}$.  We shall use $T_{Q, r}A^{\pm}_{r'}(Q')$ will denote the corkscrew point associated to associated to $T_{Q, r}Q' = \frac{Q' - Q}{r} \in T_{Q, r} \partial \Omega^{\pm}$ at the scale $\frac{r'}{r}.$
\end{definition} 

\begin{lem}\label{Lipschitz}
Let $v \in \mathcal{A}(\Lambda, \alpha, M_0)$ with $||\ln(h) ||_{\alpha} \le \Gamma$.  Then for all $Q \in \partial \Omega^{\pm} \cap B_2(0),$ and all radii $0<r\le 2$, the function $T_{Q, r}v$ is locally Lipschitz with uniform constants depending only upon $M_0, \Gamma, \alpha$.
\end{lem}

\begin{proof}
Recall that by Definition \ref{harmonic rescaling def}
$$
v_{Q, r} = v_{Q, r}^0 = \frac{r^{n-2}}{\omega^{-}(B_r(Q))}v(rx + Q).
$$
By NTA estimates, for all $0< r$, $|v(A^-_r(Q))| \sim \frac{\omega^{-}(B_r(Q))}{r^{n-2}}$ by constants which only depend upon $M_0.$  Thus, $v_{Q, r}(T_{Q, r}A^-_r(Q))$ is bounded above and below be constants which only depend upon $M_0$.  By constructing Harnack chains from $T_{Q, r}A^-_r(Q)$ to $T_{Q, r}A^-_{M_0r}(Q)$ we can find a point $y \in \partial B_{1}(0)$ such that $y \in B_{r_i}(y_i) \subset T_{Q, r}\Omega^-$ and $\text{dist}(y_i, T_{Q, r}\partial \Omega^{\pm}) \ge \frac{1}{2M_0^2}.$ Applying Harnack's Inequality to the function $-v$ in chain of balls which connect $T_{Q, r}A^-_{r}(Q)$ and $y$ in $\Omega^-$, we have $|v_{Q, r}(y)| \sim_{M_0} |v_{Q, r}(T_{Q, r}A^-_{r}(Q))|$. That is, $|v_{Q, r}(y)|$ is bounded above and below by constants that only depend upon $M_0$. Thus, by the uniform Lipschitz property of $v_{Q, r}$ guaranteed by Theorem \ref{TBE combo}, we can find a ball of radius $0<c$ such that $|v_{Q, r}| \ge c(M_0)$ on $\partial B_1(0) \cap B_{c}(y).$  Thus, $H(0, 1, v_{Q, r}) \ge c(M_0)$.
Now, recalling Definition \ref{L2 rescaling def} and the fact that $T_{0, 1}v = T_{0,1} (cv)$ for any constant, $c>0$,  we have that $T_{Q, r}v = T_{0, 1}v_{Q, r}$.  Since we have assumed that $||\ln(h) ||_{\alpha} \le \Gamma$, $Q \in B_2(0)$, and $0< r\le 2$, the $v_{Q, r}$ are locally uniformly Lipschitz by Theorem \ref{TBE combo}.  Thus, $H(0, 1, v_{Q, r}) \ge c(M_0)$ implies that $T_{0, 1}v_{Q, r} =T_{Q, r}v$ is also uniformly locally Lipschitz.
\end{proof}

\begin{lem}(Local Growth Control)\label{growth control}
Let $Q \in \partial \Omega^{\pm}$ and $0< r<\infty$.  Let $v \in\mathcal{A}(\Lambda, \alpha, M_0)$ be such that $||\ln(h) ||_{\alpha} \le \Gamma$.  The rescaling $T_{Q, r}v$ satisfies the following minimum growth conditions.  For all $0<\epsilon$, there is a constant $C=C(M_0, \alpha, \Gamma, \epsilon, R)$ such that if $p \in B_R(0)$ with $\text{dist}(p,  \{T_{Q, r}\partial \Omega^{\pm}\} \cap B_R(0)) > \epsilon$
\begin{align*}
|T_{Q, r}v(p)| > C.
\end{align*}
\end{lem}

\begin{proof}
As in Lemma \ref{Lipschitz}, $T_{Q, r}v(T_{Q, r}A^-_{r}(Q))$ is bounded above and below by constants that only depend upon the NTA constant $M_0, \Gamma,$ and $R$. Thus, by Harnack chains between $T_{Q, r}A^-_{r}(Q)$ and $p \in T_{Q, r}\Omega^- \cap B_R(0)$ such that $\text{dist}(p,  T_{Q, r}\partial \Omega^{\pm} \cap B_R(0)) > \epsilon$, Harnack's inequality applied to $-T_{Q, r}v$ implies that $|T_{Q, r}v(p)| > C$.  Note that $C$ only depends upon $R, M_0, \epsilon$.

To get the same inequality for $p \in T_{Q, r}\Omega^+ \cap B_R(0)$, we recall that standard NTA results compare $T_{Q, r}v(T_{Q, r}A^+_r(Q))$ to $T_{Q, r}\omega^+(B_1(0))$.  By Remark \ref{r: omega comparable}, $T_{Q, r}\omega^+(B_1(0)) \sim T_{Q, r}\omega^-(B_1(0))$ by constants which only depend upon $R$, , $\Gamma$, $\alpha$, and the NTA constants in the definition of the class $\mathcal{A}(\Lambda, \alpha, M_0)$.  Applying the same Harnack chain and Harnack inequality argument as above gives the lemma.
\end{proof}

\begin{lem}(Compactness I)\label{compactness-ish I}
Let $\{v_i\}$ be a sequence of functions in $\mathcal{A}(\Lambda, \alpha, M_0)$ such that $||\ln(h) ||_{\alpha} \le \Gamma$.  Let $\{Q_i \} \subset \partial \Omega_i^{\pm} \cap B_1(0)$ and $0<r_i<1$.  There is a subsequence, $\{v_j\}$, and a Lipschitz function, $v_{\infty} \in W^{1, 2}_{loc}$, such that $T_{Q_j, r_j}v_{j} \rightarrow v_{\infty}$ in the following senses:
\begin{enumerate}
\item $T_{Q_j, r_j}v_ j \rightarrow v_{\infty}$ in $C_{loc}(\RR^n)$.
\item $T_{Q_j, r_j}v_j \rightarrow v_{\infty}$ in $L^{2}_{loc}(\RR^n)$.
\item $\nabla T_{Q_j, r_j}v_j \rightharpoonup \nabla v_{\infty}$ in $L^{2}_{loc}(\RR^n; \RR^n).$
\end{enumerate}
\end{lem}

\begin{proof} 
To see $(1)$, we recall Lemma \ref{Lipschitz} and that $T_{Q_i, r_i}v_i(0) =0$.  By Arzela-Ascoli there exists a subsequence that $T_{Q_j, r_j}v_j \rightarrow v_{\infty}$ in $C_{loc}(\mathbb{R}^n)$.  Being uniformly locally Lipschitz and uniformly bounded also implies that the functions $\{T_{Q_j, r_j}v_i\}$ are bounded in $W^{1, 2}_{loc}(\RR^n).$  By Rellich Compactness, there exists a subsequence $T_{Q_j, r_j}v_j \rightarrow v_{\infty}$ in $L^2_{loc}(\RR^n)$ and $\nabla T_{Q_j, r_j} v_j \rightharpoonup \nabla v_{\infty}$ in $L^2_{loc}(\mathbb{R}^n)$. 
\end{proof}

Before we can prove strong convergence $T_{Q_j, r_j}v_j \rightarrow v_{\infty}$ in $W^{1,2}_{loc}(\RR^n)$, we need to control the upper Minkowski dimension of $\{v_{\infty} = 0\}.$

\begin{lem}\label{unif approx of limit set}
 Under the assumptions of Lemma \ref{compactness-ish I}, if $T_{Q_i, r_i}v_i \rightarrow v_{\infty}$ in $C_{loc}(\RR^n)$, then $T_{Q_i, r_i}\partial \Omega^{\pm} \rightarrow \{v_{\infty}=0\}$ locally in the Hausdorff metric on compact subsets.
\end{lem}

\begin{proof}
We argue by contradiction.  Suppose that there exists an $0< \epsilon$ and radius, $0< R$, such that we can find a sequence, $T_{Q_i, r_i}v_i$ such that there exists a point, $x_i \in  B_R(0) \cap \{T_{Q_i, r_i}v_i = 0 \}$, such that $\text{dist}(x_i, \{v_{\infty}=0\}) > \epsilon$.  Taking a subsequence which converges in $C_{loc}(\RR^n)$, we may assume that $x_i \rightarrow x_{\infty} \in \overline B_R(0) \setminus B_{\epsilon}(\{v_{\infty}=0\}).$  Now, convergence in $C_{loc}(\RR^n)$ implies that $T_{Q_i, r_i}v_i(x_{\infty}) \rightarrow v_{\infty}(x_{\infty})$.  Furthermore, since $T_{Q_i, r_i}v_i$ are uniformly locally Lipschitz, $x_i \rightarrow x_{\infty}$, and $x_i \in \{T_{Q_i, r_i}v_i = 0 \}$, we have that
\begin{align*}
T_{Q_i, r_i}v_i(x_{\infty}) \rightarrow 0.
\end{align*}
This implies $x_{\infty} \in \{v_{\infty} = 0\}$, which contradictions our previous assertion, that $x_{\infty} \in \overline{B_R(0)} \setminus B_{\epsilon}(\{v_{\infty}=0\})$.  

The other direction goes the same way.  Suppose that we could find a sequence of $T_{Q_i, r_i}v_i \rightarrow v_{\infty}$ such that there was a point, $x \in \{v_{\infty}=0\} \cap B_R(0)$ for which $$\text{dist}(x, \{T_{Q_i, r_i}v_i = 0 \} \cap B_R(0)) > \epsilon$$ for all $i= 1, 2, ...$.  By Lemma \ref{growth control}, we know that $T_{Q_i, r_i}v_i(x) > C$.  This contradicts convergence in $C_{loc}(\RR^n)$, however, since $v_{\infty}(x) = 0$.
\end{proof}

\begin{thm}\label{NTA limit set}(\cite{KenigToro06}, Theorem 4.1)
 In general, if $\partial \Omega^{\pm}_i \in \mathcal{D}(n, \alpha, M_0)$ converge to a closed set, $A$, locally in the Hausdorff metric on compact subsets, then $A$ divides $\RR^n$  into two unbounded, $2$-sided NTA domains with NTA constant bounded by $2M_0$.
\end{thm} 

We must now bound the upper Minkoski dimension of $A = \{v_{\infty} = 0\}.$  We do so crudely, using only that $A$ is the mutual boundary of a pair of two-sided NTA domains.  That is, using the machinery of porous sets we are able to prove the following lemma.  

\begin{lem}\label{Minkowski bounds on limit set}
Let $\Sigma \subset \RR^n$ be the mutual boundary of a pair of unbounded two-sided NTA domains with NTA constant $1< M_0$. Then, there is an $0< \epsilon = \epsilon(M_0, n)$ such that $\overline{dim_{\mathcal{M}}}(E) \le n-\epsilon$.
\end{lem}

This is an elementary fact which seems to be omitted in the literature.  We defer the proof to Appendix A.  We now prove strong convergence.

\begin{lem}(Strong Compactness)\label{strong compactness I}
Let $\{v_i\}$ be a sequence of functions in $\mathcal{A}(\Lambda, \alpha, M_0)$ such that $||\ln(h) ||_{\alpha} \le \Gamma$.  Let $\{Q_i \} \subset \partial \Omega_i^{\pm} \cap B_1(0)$ and $0<r_i<1$.  There is a subsequence, $\{v_j\}$, and a Lipschitz function, $v_{\infty} \in W^{1, 2}_{loc}$, such that $T_{Q_j, r_j}v_{j} \rightarrow v_{\infty}$ in the following senses:
\begin{enumerate}
\item $T_{Q_j, r_j}v_ j \rightarrow v_{\infty}$ in $C_{loc}(\RR^n)$
\item $T_{Q_j, r_j}v_j \rightarrow v_{\infty}$ in $W^{1, 2}_{loc}(\RR^n)$.
\end{enumerate}
\end{lem}

\begin{proof}
The only new claim is that $\nabla T_{Q_j, r_j}v_j \rightarrow \nabla v_{\infty}$ in $L^2_{loc}(\RR^n; \RR^n).$  By Lemma \ref{unif approx of limit set}, Theorem \ref{NTA limit set}, and Lemma \ref{Minkowski bounds on limit set}, we have that $\overline{\dim}_{\mathcal{M}}(\{ v_{\infty}=0 \}) \le n-\epsilon.$  In particular, then, $\mathcal{H}^n(B_{r}(\{ v_{\infty} = 0\} \cap B_R(0))) \rightarrow 0$ as $r \rightarrow 0$ (see \cite{Mattila95} for fundamental facts about Minkowski content, dimension and Hausdorff measure).  Thus, for any $\theta>0$ we can find an $r(\theta) > 0$ such that  $\mathcal{H}^n(B_{r}( \{ v_{\infty} = 0\}  \cap B_R(0))) \le \theta$.  This allows us to estimate
\begin{align*}
\limsup_{j \rightarrow \infty} ||T_{Q_j, r_j}v_j||_{L^2(B_R(0))}^2 = & \limsup_{j \rightarrow \infty} ( \int_{B_R(0) \cap B_r( \{v_{\infty} =0\} )}|\nabla T_{Q_j, r_j}v_j|^2dV \\
 & \qquad + \int_{B_R(0) \setminus B_r(\{v_{\infty} =0\})} |\nabla T_{Q_j, r_j}v_j|^2dV ) \\
 \le & \lim_{j \rightarrow \infty} \int_{B_R(0) \setminus B_r(\{v_{\infty} =0\})}|\nabla T_{Q_j, r_j}v_j|^2dV + C \theta \\
 \le &  ||v_{\infty}||^2_{L^2(B_R(0))} + C \theta.
\end{align*}
where the penultimate inequality uses that $v_j$ are uniformly Lipschitz, and the last equality follows from $W^{1, 2}$ convergence of harmonic functions in the region $B_R(0) \setminus B_r(\{v_{\infty} =0\}).$  Since $\theta > 0$ was arbitrary, we have that $\lim_{j \rightarrow \infty} ||T_{Q_j, r_j}v_j||^2_{B_R(0)} \le  ||v_{\infty}||^2_{L^2(B_R(0))}$.  The other inequality follows from the same trick or from lower semi-continuity. Therefore, we have the equality
\begin{align*}
\lim_{j \rightarrow \infty}||T_{Q_j, r_j}v_j||^2_{L^2(B_R(0))} = ||v_{\infty}||^2_{L^2(B_R(0))}.
\end{align*}

Thus, we have by weak convergence and norm convergence
\begin{align*}
\lim_j || \nabla T_{Q_j, r_j}v_j - \nabla v_{\infty}||^2_{L^2(B_R(0))}  = & \lim_j \int_{B_R(0)}|\nabla T_{Q_j, r_j}v_j - \nabla v_{\infty}|^2dV \\
 = & \lim_j ||\nabla T_{Q_j, r_j}v_j||_{L^2(B_R(0))}^2 + || \nabla v_{\infty}||_{L^2(B_R(0))}^2 \\
 & - 2 \lim_j\langle \nabla T_{Q_j, r_j}v_j , \nabla v_{\infty}\rangle_{L^2(B_R(0))}\\
  = & 2|| \nabla v_{\infty}||_{L^2(B_R(0))}^2 - 2|| \nabla v_{\infty}||_{L^2(B_R(0))}^2 =  0.
 \end{align*}
\end{proof}

Because the functions $v^Q_{p, r}$ are merely Lipschitz, we will often need to work with a mollified version of them.  We will use the convention that $v_{\epsilon} = v \star \phi_{\epsilon}$ for $\phi \in C^{\infty}$ a mollifying function ($\text{spt} (\phi) \subset B_1$ and $\int \phi dV= 1$).

\begin{cor}\label{v_e to v}
Let $v \in \mathcal{A}(\Lambda, \alpha, M_0)$ and $v_{\epsilon} = v \star \phi_{\epsilon}$, be a mollification of $v$.  By standard mollification results, 
\begin{equation*}
v_{\epsilon} \rightarrow v \text{  in  } W^{1, 2}_{loc}(\RR^n), ~ C_{loc}(\RR^n) \text{  as  } \epsilon \rightarrow 0. 
\end{equation*}
\end{cor}

\section{Almost monotonicity of the Almgren frequency function}\label{S:monotonicity}
 
One of the key tools of this paper will be the Almgren frequency function (introduced in \cite{Almgren79}). Since we want to capture the behavior of $v$ on more than just a single level set, we make the following definitions.

\begin{definition} \label{N W def} 
For any Lipschitz function $v: \RR^n \rightarrow \RR$, radius $r >0$, and point $Q \in \partial \Omega^{\pm}$, we define the following quantities:
Let 
\begin{align}
    H(p, r, v) & = \int_{\partial B_r(p)} |v|^2 d\sigma\\
D(p, r, v) & = \int_{B_r(p)} |\nabla v|^2dV\\
N(p, r, v) & = r \frac{D(p, r, v)}{H(p, r, v)}.
\end{align}
\end{definition}

This normalized version of the Almgren frequency function is invariant in the following senses.

\begin{lem}\label{N rescaling invariance}
Let $a, b, c \in \RR$ with $a, b \not = 0$.  If $w(x) = av(bx)$, then 
$$
N(0,r, v) = N(0, b^{-1}r, w).
$$
\end{lem}

If $u$ is harmonic and $u(p)=0$, then $N(p, r, u) $ is monotonically non-decreasing, and $\lim_{r \rightarrow 0} N(p, r, u) = N(p, 0, u)$ is the degree of the leading homogeneous harmonic polynomial in the Taylor expansion of $u$ at the point $p$.   

\begin{cor}\label{N W convergence cor}
Under the hypotheses of Lemma \ref{compactness-ish I}, there exists a subsequence such that 
$$
N(0, r, T_{Q_j, r_j}v_j) \rightarrow N(0, r, v_{\infty})
$$
for all $r \in (0, 2].$
\end{cor}

\begin{proof}
This follows from the convergence of the numerator and the denominator.  The former is Lemma \ref{compactness-ish I} (2).  The later follows from  Lemma \ref{compactness-ish I} (1).
\end{proof}

This section is dedicated to extending the following result of Engelstein, \cite{Engelstein16}.  

\begin{lem}(\cite{Engelstein16}) \label{E almost monotonicity lem}
Let $v \in \mathcal{A}(\Lambda, \alpha, M_0)$ and $Q \in K \subset \subset \partial \Omega^{\pm}$.  There exists a constant, $C < \infty$, (which can be taken uniformly over $K$ and $r \in (0,1]$) such that 
\begin{equation}\nonumber
\liminf_{\epsilon \rightarrow 0}N(Q, r, v^{Q}_{\epsilon}) - N(Q, 0, v^{Q}_{\epsilon}) > -Cr^{\alpha}.
\end{equation}
\end{lem}

In order to connect the Almgren frequency to the Jones $\beta$-numbers in the ``frequency pinching" result in Lemma \ref{beta bound lem}, we will need to estimate the non-degeneracy of $\frac{d}{dr}N(r, p, v_{\epsilon})$, see Lemma \ref{N non-degenerate almost monotonicity}.

The broad strategy is to use standard NTA estimates to extend estimates on $\frac{d}{dr}N(Q, r, v)$ obtained by \cite{Engelstein16} to $p \not \in \partial \Omega^{\pm}$ and scales $s \ge \text{dist}(p, \partial \Omega^{\pm}).$  All ``standard" NTA estimates may be found in \cite{JerisonKenig82}.

Throughout this section, we shall use the notation $(v_{\epsilon})_{\nu}(y) = \nabla v_{\epsilon}(y) \cdot \nu(y)$, where $\nu(y)$ is the unit normal to $\partial B_r(Q)$ at $y.$

By classical results, (see \cite{Engelstein16}, Section 5.1 for details of the derivation)
\begin{align*} \label{N derivative eq}
& H(Q, r, v_{\epsilon})^{2}\frac{d}{dr}N(Q, r, v_{\epsilon}) = \\
& \qquad 2r\left( \int_{\partial B_r(Q)} (v_{\epsilon})^2_{\nu}d\sigma \int_{\partial B_r(Q)}|v_{\epsilon}|^2d\sigma - [\int_{\partial B_r(Q)}v_{\epsilon}(v_{\epsilon})_{\nu}d\sigma]^2\right)\\ 
& \qquad + 2r \left(\int_{B_r(Q)}v_{\epsilon}\Delta v_{\epsilon}dV\right) \left(\int_{\partial B_r(Q)}v_{\epsilon}(v_{\epsilon})_{\nu}d\sigma\right)\\
& \qquad -2H(Q, r, v_{\epsilon})\int_{B_r(Q)}\langle x - Q, \nabla v_{\epsilon}  \rangle \Delta v_{\epsilon}dV.
\end{align*}

We decompose $\frac{d}{dr}N(Q, r, v_{\epsilon})= N_1'(Q, r, v_{\epsilon}) + N'_2(Q, r, v_{\epsilon})$ as follows. 
\begin{align*}
& N'_1(Q, r, v_{\epsilon}) := \\
& \qquad H(Q, r, v_{\epsilon})^{-2}2r\left( \int_{\partial B_r(Q)} (v_{\epsilon})^2_{\nu}d\sigma \int_{\partial B_r(Q)}|v_{\epsilon}|^2d\sigma - [\int_{\partial B_r(Q)}v_{\epsilon}(v_{\epsilon})_{\nu}d\sigma]^2\right).
\end{align*}  
We call what remains $N'_2(Q, r, v_{\epsilon})$
\begin{align}\nonumber
N'_2(Q, r, v_{\epsilon}) := & H(Q, r, v_{\epsilon})^{-2}[2r \left(\int_{B_r(Q)}v_{\epsilon}\Delta v_{\epsilon}dV\right)\left(\int_{\partial B_r(Q)}v_{\epsilon}(v_{\epsilon})_{\nu}d\sigma\right)\\
& -2H(Q, r, v_{\epsilon})\int_{B_r(Q)}\langle x -Q, \nabla v_{\epsilon}  \rangle \Delta v_{\epsilon}dV(x)].
 \end{align}
Note that by the Cauchy-Schwarz inequality, $N'_1(Q, r, v_{\epsilon}) \ge 0$.

\begin{lem} \label{N+ derivative bound} 
Let $v \in \mathcal{A}(\Lambda, \alpha, M_0)$, $Q \in \partial \Omega^{\pm} \cap B_1(0)$ and $0<r\le 1$. Then, if $C = Lip(v|_{B_{2}(0)})$
\begin{equation}
N'_1(Q, r, v_{\epsilon}) \ge \frac{2}{C} \int_{\partial B_r(Q)} \frac{|\nabla v_{\epsilon} \cdot (y-Q) - N(Q, r, v_{\epsilon}) v_{\epsilon}|^2}{|y-Q|^{n+2}}d\sigma(y).
\end{equation} 
\end{lem}

\begin{proof}
Recall that for the Cauchy-Schwarz Inequality, we have that for $\lambda = \frac{\langle u, v \rangle}{||v||^2}$
$$
||v ||^2 ||u-\lambda v ||^2 = ||u||^2 ||v||^2 - |\langle u, v \rangle|^2.
$$
 
Choosing $u = \nabla v \cdot (y-Q)$ and $v = v_{\epsilon}$, we have by the divergence theorem
\begin{equation}\nonumber
N'_1(Q, r, v_{\epsilon}) = H(Q, r, v_{\epsilon})^{-1}2r\left( \int_{\partial B_r(Q)} |(v_{\epsilon})_{\nu} - N(Q, r, v_{\epsilon}) v_{\epsilon}|^2d\sigma\right) .
\end{equation} 

Since $C = Lip(v|_{B_{2}(0)})$, we observe that $H(Q, r, v_{\epsilon}) \le Cr^{n+1}$.  Plugging this into the above equation,  we get the desired inequality
\begin{equation}\nonumber
N'_1(Q, r, v_{\epsilon}) \ge \frac{2}{C} \int_{\partial B_r(Q)} \frac{|\nabla v_{\epsilon}(y) \cdot (y-Q) - N(Q, r, v_{\epsilon}) v_{\epsilon}(y)|^2}{|y-Q|^{n+2}}d\sigma(y).
\end{equation} 
\end{proof}

\begin{rmk}
Lemma \ref{N+ derivative bound} is scale invariant in the sense that  if $C= Lip(T_{Q, r}v|_{B_{2r}(p)})$
\begin{equation}
N'_1(0, 1, T_{Q, r}v_{\epsilon}) \ge \frac{2}{C} \int_{\partial B_1(0)} \frac{|\nabla T_{Q, r}v_{\epsilon} \cdot (y-0) - N(0, 1, T_{Q, r}v_{\epsilon}) T_{Q, r}v_{\epsilon}|^2}{|y|^{n+2}}d\sigma(y).
\end{equation} 

\end{rmk}

In the next lemma, we use recall some estimates to bound the parts of $N'_2(Q, r, v_{\epsilon})$.

\begin{lem}(\cite{Engelstein16} (Lemmata 5.4, 5.5, and 5.6)) \label{N neg 1}
 Let $v \in \mathcal{A}(\Lambda, \alpha, M_0)$ with $||\ln(h) ||_{\alpha} \le \Gamma$, and let $Q \in B_{1}(0) \cap \partial \Omega^{\pm}$. For any $0< s$ and $\epsilon << s$
\begin{align}
\int_{\partial B_s(Q)}|v_{\epsilon}|^2 d\sigma & \ge C(M_0) \frac{\omega^-(B_s(Q))^2}{s^{n-3}}\\
|\int_{B_s(Q)}v_{\epsilon}\Delta v_{\epsilon} dV| & \le C ||\ln(h) ||_{\alpha} s^{\alpha} \frac{\omega^-(B_{s}(Q))^2}{s^{n-2}}\\
|\int_{B_s(Q)}\langle \nabla v_{\epsilon}, x -Q\rangle\Delta v_{\epsilon} dV(x)| & \le C||\ln(h) ||_{\alpha}s^{\alpha}\frac{\omega^-(B_s(Q))^2}{s^{n-2}}\\
|\int_{\partial B_s(Q)}v_{\epsilon}( v_{\epsilon})_{\nu} d\sigma| & \le C
\frac{\omega^-(B_s(Q))^2}{s^{n-1}},
\end{align}
where the constant $C= C(M_0)$.
\end{lem}

\begin{rmk}\label{N- remark}
Recalling our expansion of $\frac{d}{dr}N(r, p, v_{\epsilon})$ in Equation \ref{N derivative eq} and the bounds contained in Lemmata \ref{N neg 1} we have that for $v \in \mathcal{A}(\Lambda, \alpha, M_0)$ with $||\ln(h) ||_{\alpha} \le \Gamma$, $\epsilon<< r$, and $Q \in B_1(0) \cap \partial \Omega^{\pm}$
\begin{equation}
|N'_2(Q, r, v_{\epsilon})| \le C_1 ||\ln(h)||_{\alpha} r^{\alpha-1},
\end{equation}
where $C_1 = C(\alpha, M_0, \Gamma)$.
\end{rmk}

Now we would like to let $\epsilon \rightarrow 0$ and prove a more general version of Lemma \ref{E almost monotonicity lem} in Lemma \ref{N non-degenerate almost monotonicity}.  

\begin{rmk}\label{lambda rmk}
For all $Q \in B_1(0) \cap \partial \Omega^{\pm}$, and $0< r \le 1$
\begin{align*}
\lim_{\epsilon \rightarrow 0} N(Q, r, v_{\epsilon}) = N(Q, r, v).
\end{align*}
\end{rmk}

\begin{lem}\label{N non-degenerate almost monotonicity}
Let $v \in \mathcal{A}(\Lambda, \alpha, M_0)$ with $||\ln(h) ||_{\alpha} \le \Gamma$, and let $Q \in B_1(0) \cap \partial \Omega^{\pm}$.  For any $0 \le s < S \le 1$
\begin{eqnarray}\nonumber \label{N non-degeneracy}
\frac{2}{C} \int_{A_{s, S}(Q)} \frac{|\nabla T_{0, 1}v(y) \cdot (y-Q) - N(Q, |y-Q|, T_{0, 1}v) T_{0, 1}v(y)|^2}{|y-Q|^{n+2}}dV(y) \\ \le N(Q, S, v) - N(Q, s, v) +C_1||\ln(h)||_{\alpha} S^{\alpha},
\end{eqnarray}
where $C_1 = C_1(\alpha, M_0, \Gamma)$ and $C(M_0, \Gamma, \alpha) = Lip(T_{0, 1}v|_{B_2(0)}).$
\end{lem}

\begin{proof}
We begin by normalizing $v$.  Since $N(r, p, v) = N(r, p, cv)$ for any $c \not = 0$, we may work with $T_{0, 1}v$. By Remark \ref{v_e to v}, since $v_{\epsilon} \rightarrow v$ in $W^{1, 2}_{loc}(\mathbb{R}^n)$ as $\epsilon \rightarrow 0$, we can find an $\epsilon<< s$ small enough that $|N(Q, s, v_{\epsilon}) - N(Q, s, v)| <S^{\alpha}$ and $|N(Q, S, v_{\epsilon}) - N(Q, S, v)| <S^{\alpha}$.  Furthermore, since $T_{0,1}v_{\epsilon}$ is locally uniformly Lipschitz, $N(Q, r, v_{\epsilon})$ is continuous in $r$.  Therefore, if $s = 0$, we can find an $0 = s <s_1 < S$ such that $|N(Q, s_1, v_{\epsilon}) - N(Q, s, v_{\epsilon})| <S^{\alpha}$.

Thus, we reduce to estimating $N(Q, s, T_{0, 1}v_{\epsilon}) - N(Q, s_1, T_{0, 1}v_{\epsilon})$.
\begin{align*}
N(Q, s, T_{0, 1}v_{\epsilon}) - N(Q, s_1, T_{0, 1}v_{\epsilon}) &= \int^S_{s_1} \frac{d}{dr}N(Q, r, T_{0, 1}v_{\epsilon})dr\\
& = \int^S_{s_1} N'_1(Q, r, v_{\epsilon})dr + \int^S_{s_1} N'_2(Q, r, T_{0, 1}v_{\epsilon})dr.
\end{align*}

Recalling Remark \ref{N- remark}, for $\epsilon$ small enough, we bound 
\begin{align*}
\int^S_{s_1} N'_2(Q, r, T_{0, 1}v_{\epsilon})dr & \ge \int^S_{s_1} - C_1||\ln(h) ||_{\alpha}r^{\alpha - 1}dr\\
& \ge -C_1||\ln(h) ||_{\alpha}S^{\alpha}.
\end{align*}

By Lemma \ref{N+ derivative bound}, we have 
\begin{align*}
& \int^S_{s_1} N'_1(Q, r, T_{0, 1}v_{\epsilon})dr \\
& \qquad \ge \int^S_{s_1} \frac{2}{C} \int_{\partial B_r(Q)} \frac{|\nabla T_{0, 1}v_{\epsilon} \cdot (y-Q) - N( Q, |y-Q|, T_{0, 1}v_{\epsilon}) T_{0, 1}v_{\epsilon}|^2}{|y-Q|^{n+2}}d\sigma(y)dr\\
& \qquad \ge \frac{2}{C} \int_{A_{s_1, S}(Q)} \frac{|\nabla T_{0, 1}v_{\epsilon} \cdot (y-Q) - N(Q, |y-Q|,  T_{0, 1}v_{\epsilon}) T_{0, 1}v_{\epsilon}|^2}{|y-Q|^{n+2}}dV(y).
\end{align*}

Recalling Corollary \ref{v_e to v} and letting $\epsilon \rightarrow 0$ gives that stated result.
\end{proof}

Using these estimates it is possible to control the drop across scales from the total drop.

\begin{lem}\label{global to local}
If Let $v \in \mathcal{A}(\Lambda, \alpha, M_0)$, with $||\ln(h)||_{\alpha} \le \Gamma$, and $Q \in B_{1}(0) \cap \partial \Omega^{\pm}$, then for any $0 \le r \le s < S \le R$
\begin{align*}
N(Q, S, v) - N(Q, s, v) \le 2C_1||\ln(h) ||_{\alpha} R^{\alpha} + |N(Q, R, v) - N(Q, r, v)|.
\end{align*}
\end{lem}

\begin{proof}
This is essentially a ``rays of the sun" argument.  To wit
\begin{align*}
N(Q, s, v) - N(Q, s, v) = & \int_s^S N'_1(Q, \rho, v) + N'_2(Q, \rho, v) d\rho\\
\le & \int_s^S N'_1(Q, \rho, v) + |N'_2(Q, \rho, v)| d\rho\\
\le & \int_r^R N'_1(Q, \rho, v) + |N'_2(Q, \rho, v)| d\rho\\
\le & 2\int_r^R |N'_2(Q, \rho, v) |d\rho + |N(Q, R, v) - N(Q, r, v)|.
\end{align*}
The bounds in Remark \ref{N- remark} give the desired statement.
\end{proof} 

\section{Uniform Bound on the Almgren frequency}\label{S:N bound}

We now turn our attention to proving some analogs of classical harmonic results for the Almgren frequency function for functions $v \in \mathcal{A}(\Lambda, \alpha, M_0)$.  The main result of this section is Lemma \ref{N bound lem}, which states that we may bound the Almgren frequency uniformly for all points $Q \in B_{1/4}(0) \cap \partial \Omega^{\pm}$ and all scales $0<r\le 1/2$ by a function of $N(0, 1, v) \le \Lambda.$

\begin{lem} \label{H e doubling-ish lem}($H(r, p ,v_{\epsilon})$ is almost doubling.)
Let $v \in \mathcal{A}(\Lambda, \alpha, M_0)$ with $||\ln(h)||_{\alpha} \le \Gamma$, and let $Q \in B_{1/2}(0) \cap \partial \Omega^{\pm}$. For any $0 < s < S \le 1$ if $\epsilon < < s$ is sufficiently small
\begin{equation}
H(Q, s, v_{\epsilon}) \le \left(\frac{S}{s}\right)^{(n-1) + 2(N(Q, S, v_{\epsilon}) + CS^{\alpha})}e^{\frac{2C}{\alpha}[S^{\alpha} - s^{\alpha}]} H(Q, s, v_{\epsilon}),
\end{equation}
where $C=||\ln(h) ||_{\alpha} C_1(M_0, \alpha, \Gamma)$.
\end{lem}

\begin{proof}
First, observe that 
$$H'(Q, r, v_{\epsilon}) = \frac{n-1}{r}\int_{\partial B_r(Q)}|v_{\epsilon}|^2d\sigma + 2 \int_{B_r(Q)} |\nabla v_{\epsilon}|^2dV +2\int_{B_r(Q)}v_{\epsilon}\Delta v_{\epsilon}dV.$$

Next, we consider the following identity
\begin{align*}
\ln(\frac{H(Q, S, v_{\epsilon})}{H(Q, s, v_{\epsilon})}) & =  \ln(H(Q, S, v_{\epsilon})) - \ln(H(Q, s, v_{\epsilon}))\\
& =  \int_s^S \frac{H'(Q, r, v_{\epsilon})}{H(Q, r, v_{\epsilon})}dr\\
 & =  \int_s^S \frac{n-1}{r} + \frac{2}{r}N(Q, r, v_{\epsilon}) + 2\left(\frac{\int_{B_r(Q)}v_{\epsilon}\Delta v_{\epsilon}dV}{\int_{\partial B_r(Q)} (v_{\epsilon})^2d\sigma}\right)dr.
\end{align*}
 
We bound $N(r, p, v_{\epsilon})$ from above using Lemma \ref{N non-degenerate almost monotonicity}.  We bound the last term using Lemma \ref{N neg 1}.  Plugging in these bounds, we have that for $\epsilon << s$
\begin{align*}
\ln\left(\frac{H(Q, S, v_{\epsilon})}{H(Q, s, v_{\epsilon})}\right) & \le \left[(n-1) + 2(N(Q, s, v_{\epsilon}) + CS^{\alpha})\right] \ln(r)|^{S}_s + \frac{2C}{\alpha} r^{\alpha}|^S_s.
\end{align*}
Evaluating and exponentiating gives the desired result.
\end{proof}

\begin{rmk}\label{H doubling-ish remark}
Because $H(Q, r, v_{\epsilon}) \rightarrow H(Q, r, v)$ as $\epsilon \rightarrow 0$ and $N(Q, r, v_{\epsilon}) \rightarrow N(Q, r, v)$ as $\epsilon \rightarrow 0$ (a consequence of Corollary \ref{v_e to v}), we have the following inequality.  For all $v \in \mathcal{A}(\Lambda, \alpha, M_0)$ with $||\ln(h)||_{\alpha}\le \Gamma$, $Q \in B_{\frac{1}{2}}(0) \cap \partial \Omega^{\pm}$, and $0 < s < S \le 1$ 
\begin{align}
H(Q, S, v) \le& \left(\frac{S}{s}\right)^{(n-1) + 2(N(Q, S, v) + CS^{\alpha})}e^{\frac{2C}{\alpha}[S^{\alpha} - s^{\alpha}]} H(Q, s, v).
\end{align}
\end{rmk}
 
Now, we bound the Almgren frequency.

\begin{lem}\label{N bound lem}
Let $v \in \mathcal{A}(\Lambda, \alpha, M_0)$ as above. There is a function, $C(n, \alpha, \Gamma, \Lambda, M_0)$ such that if $||\ln(h)||_{\alpha}\le \Gamma$, then for all $Q \in B_{\frac{1}{4}}(0) \cap \partial \Omega^{\pm}$ and all $r \in (0, \frac{1}{2})$ 
\begin{equation}
 N(Q, r, v)  \le  C(n, \Lambda, \alpha, M_0, \Gamma).
\end{equation}
\end{lem}

\begin{proof}
We recall that the Almgren frequency function is invariant under rescalings of the function $v$.  Therefore, we normalize our function $v$ by the rescaling $v_{0,1}$ and relabel as $v$.  As remarked in Theorem \ref{TBE combo}, $v$ is now uniformly locally Lipschitz with a Lipschitz constant which only depends upon the NTA constant $M_0$.  Therefore, applying Remark \ref{H doubling-ish remark} to $Q = 0$ (hence $v(Q) = 0$), letting $r = cR$, and integrating both sides with respect to $R$ from $0$ to $S$, we have that for any $c \in (0, 1)$
\begin{align*}
\int_0^S\int_{\partial B_{R}(0)}|v|^2 d\sigma dR &\le \int_0^{S} (\frac{1}{c})^{(n-1) + 2(N(0, R, v) + CR^{\alpha})}e^{\frac{2C}{\alpha}[R^{\alpha} - (cR)^{\alpha}]} \int_{\partial B_{cR}(0)}|v|^2d\sigma dR\\
& \le (\frac{1}{c})^{(n-1) + 2(N(0, S, v) + 2CS^{\alpha})}e^{\frac{2C}{\alpha}S^{\alpha}} \int_0^{S}  \int_{\partial B_{cR}(0)}|v|^2d\sigma dR.
\end{align*}

Thus, we have that for any such $c \in (0,1)$ and any $0< S \le 1$
\begin{align*}
\fint_{B_S(0)}|v|^2dV \le c (\frac{1}{c})^{2[N(0, S, v) + 2CS^{\alpha}]}\cdot e^{\frac{2C}{\alpha}S^{\alpha}} \fint_{B_{cS}(0)} |v|^2dV.
\end{align*}

Let $S = 1$ and $c= \frac{1}{16}$.  We have that
\begin{align}\label{H almost doubling balls}
\fint_{B_1(0)}|v|^2dV \le \frac{1}{16} (16)^{2[N(0, 1, v) + C]}\cdot e^{\frac{2C}{\alpha}} \fint_{B_{\frac{1}{16}}(0)} |v|^2dV.
\end{align}

We now underestimate the left-hand side of Equation (\ref{H almost doubling balls}).  For any $Q \in B_{\frac{1}{4}}(0) \cap \partial \Omega^{\pm}$
\begin{align*}
\fint_{B_1(0)}|v|^2dV & \ge c(n) \fint_{B_{\frac{3}{4}}(p)}|v|^2dV \\
& \ge c(n) \fint_{B_{\frac{3}{4}}(Q) \setminus B_{\frac{1}{4}}(Q)}|v|^2dV.
\end{align*}

Now, we overestimate the right-hand side of Equation \ref{H almost doubling balls}. For any $Q \in B_{\frac{1}{4}}(0) \cap \partial \Omega^{\pm}$
\begin{align*}
\fint_{B_{\frac{1}{16}}(0)}|v|^2dV & \le c(n) \fint_{B_{\frac{9}{16}}(Q)}|v|^2dV \\
& \le c(n) \left(\int_{B_{\frac{9}{16}}(Q) \setminus B_{\frac{1}{4}}(Q)}|v|^2dV +  \int_{B_{\frac{1}{4}}(Q)}|v|^2dV\right)\end{align*}

Now, let $x^-_{max}(Q, 9/16) \in \partial B_{\frac{9}{16}}(Q) \cap \overline{\Omega^-}$ be the point which maximizes $|v|$ on $\partial B_{\frac{9}{16}}(Q) \cap \overline{\Omega^{-}}.$ Note that $|v(x^-_{max}(Q, 9/16))| \ge |v(A^-_{\frac{9}{16}}(Q))|$, and by Harnack chains between $A^-_{\frac{9}{16}}(Q)$ and $A^-_{1}(0)$, we have $|v(x^-_{max}(Q, 9/16))| \sim_{M_0} 1.$  By a similar argument applied to $x^+_{max}(Q, 9/16)$, recalling Remark \ref{r: omega comparable} we have a similar statement for $\Omega^+.$  Thus, by the uniform Lipschitz control on $v$, we can find a balls of radius $c(n, \alpha, M_0, \Gamma) \frac{1}{4}$ in $\Omega^{\pm} \cap B_{\frac{9}{16}}(Q)\setminus B_{\frac{1}{4}}(Q)$ respectively in which $|v| \ge C(n, \alpha, M_0, \Gamma)\frac{1}{2}$.  Thus, we have 
\begin{align*}
    \int_{B_{\frac{9}{16}}(Q) \setminus B_{\frac{1}{4}}(Q)}|v|^2dV \ge C(n, \alpha, M_0, \Gamma).
\end{align*}
Similarly, by the Maximum principle, we have that $|v| \le |v(x^-_{max}(Q, 9/16))|$ in $\Omega^- \cap B_{\frac{1}{4}}(Q)$.  By Remark \ref{r: omega comparable}, and an identical argument for $\Omega^+ \cap B_{\frac{1}{4}}(Q),$ then 
\begin{align*}
    \int_{B_{\frac{1}{4}}(Q) \setminus B_{\frac{1}{4}}(Q)}|v|^2dV \le \frac{1}{4^n} C(n, \alpha, M_0, \Gamma).
\end{align*}
Thus, there is a constant $C(n, \alpha, M_0, \Gamma)$ such that 
\begin{align*}
    C(n, \alpha, M_0, \Gamma)\int_{B_{\frac{9}{16}}(Q) \setminus B_{\frac{1}{4}}(Q)}|v|^2dV  \ge  \int_{B_{\frac{1}{4}}(Q)}|v|^2dV.
\end{align*}

Putting together these under- and over- estimates
\begin{equation}\label{N bound base ineq}
\fint_{B_{\frac{3}{4}}(Q)}|v|^2dV \le C(n, \alpha,  M_0, \Gamma) (16)^{2[N(0, 1, v) + C]}\cdot e^{\frac{2C}{\alpha}} \fint_{B_{\frac{9}{16}}(Q) \setminus B_{\frac{1}{4}}(Q)} |v|^2dV.
\end{equation}

Now, we wish to bound $\fint_{B_{\frac{3}{4}}(Q)}|v|^2dV$ from below and $\fint_{B_{\frac{9}{16}}(Q) \setminus B_{\frac{1}{4}}(Q)} |v|^2dV$ from above.  To get the lower bound, we recall that
\begin{align*}
\frac{d}{dr} \int_{\partial B_r(Q)} |v_{\epsilon}|^2 dV & = \frac{n-1}{r}\int_{\partial B_r(Q)} |v_{\epsilon}|^2d\sigma + 2\int_{\partial B_r(Q)} v_{\epsilon}(v_{\epsilon})_{\nu} d\sigma.
\end{align*}

Since the first integrand on the right hand side is positive, we need only bound the second integral.  Recall that $v = v_{0,1}$ is Lipschitz with a Lipschitz constant that only depends upon the NTA constant $M_0$.  Then, for all $0< r$ and $\epsilon << r$
\begin{align*}
\frac{d}{dr} \int_{\partial B_r(Q)} |v_{\epsilon}|^2d\sigma & \ge - 2Cr^{n+1}.
\end{align*}
Thus, for any $Q \in B_{\frac{1}{4}}(0) \cap \partial \Omega^{\pm}$, and $0<s<S<1$ we obtain
\begin{align*}
\int_{\partial B_{S}(Q)} |v_{\epsilon}|^2d\sigma - \int_{\partial B_{s}(Q)} |v_{\epsilon}|^2d\sigma & = \int_{s}^{S} \frac{d}{dr} \int_{\partial B_r(Q)} |v_{\epsilon}|^2d\sigma dr \\
& \ge  \int_{s}^{S} - 2Cr^{n+1} dr\\
& \ge  -C(M_0, n)S^{n+2}.
\end{align*}
Thus, for all $Q \in B_{\frac{1}{4}}(0) \cap \partial \Omega^{\pm}$, we may bound $\int_{B_{\frac{3}{4}}(Q)}(v_{\epsilon})^2dV$ from below as follows.
\begin{align*}
\int_{B_{\frac{3}{4}}(Q)}|v_{\epsilon}|^2 dV & \ge \int_{B_{\frac{3}{4}}(Q) \setminus B_{\frac{5}{8}}(Q)} |v_{\epsilon}|^2 dV\\
& = \int_{\frac{5}{8}}^{\frac{3}{4}} \int_{\partial B_r(Q)} |v_{\epsilon}|^2 d\sigma dr \\
& \ge  \int_{\frac{5}{8}}^{\frac{3}{4}} \left(\int_{\partial B_{\frac{5}{8}}(Q)} |v_{\epsilon}|^2 d\sigma  -C(M_0, n)\right)dr \\
& \ge  c \int_{\partial B_{\frac{5}{8}}(Q)} |v_{\epsilon}|^2 d\sigma  -C(M_0, n).
\end{align*}

To get the upper bound we want, we use the same trick. 
\begin{align*}
\int_{B_{\frac{9}{16}}(Q)\setminus B_{\frac{1}{4}}(Q)}|v_{\epsilon}|^2dV & = \int_{\frac{1}{4}}^{\frac{9}{16}} \int_{\partial B_s(Q)}|v_{\epsilon}|^2 d\sigma dr\\
& \le \int_{\frac{1}{4}}^{\frac{9}{16}} \left(\int_{\partial B_{\frac{9}{16}}(Q)}|v_{\epsilon}|^2 d\sigma + C(M_0, n)\right) dr\\
& \le c\int_{\partial B_{\frac{9}{16}} (Q)} |v_{\epsilon}|^2d\sigma +  C(M_0, n).
\end{align*}

Since these inequalities hold for all $\epsilon << 1/2$, and by Lemmata \ref{Lipschitz} and \ref{v_e to v} $$\lim_{\epsilon \rightarrow 0}H(Q, r, v_{\epsilon}) = H(Q, r, v),$$ they also hold for $v$.  Putting it all together, we plug our above bounds into Equation \ref{N bound base ineq} and consolidating constants, we obtain the following for all $Q \in B_{\frac{1}{4}}(0) \cap \partial \Omega^{\pm}$
\begin{align*}
& \fint_{\partial B_{\frac{5}{8}}(Q)} |v|^2 d\sigma  - C(M_0, n) \le \\
& \qquad C(n, \alpha, M_0, \Gamma) (16)^{2[N(0, 1, v) + C]}\left(\fint_{\partial B_{\frac{9}{16}} (Q)} |v|^2d\sigma
+ C(M_0, n)\right).
\end{align*}
Though somewhat messier, we restate the above in the following form for convenience later.
\begin{align} \nonumber \label{ugly}
\frac{\fint_{\partial B_{\frac{5}{8}}(Q)} |v|^2 d\sigma}{\fint_{\partial B_{\frac{9}{16}}(Q)}|v|^2 d\sigma} \le & C(n, \alpha,  M_0,\Gamma) (16)^{2[N(0, 1, v) + C]} + \\
&  \frac{C(n, \alpha, M_0, \Gamma)(16)^{2[N(0, 1, v) + C]} + C(M_0, n)}{\fint_{\partial B_{\frac{9}{16}}(Q)}|v|^2 d\sigma}.
\end{align}

Now, we change tack.  Observe that
\begin{align*}
& \frac{d}{ds}\ln\left(\frac{1}{s^{n-1}}H(Q, s, v_{\epsilon})\right) \\
& \qquad = \frac{1}{\fint_{\partial B_s (Q)}(v_{\epsilon})^2d\sigma} \left[\frac{(1-n)}{s} \fint_{\partial B_s (Q)}(v_{\epsilon})^2d\sigma + \frac{1}{s^{n-1}}H'(Q, s, v_{\epsilon})\right]\\
& \qquad =  \frac{1}{\fint_{B_{\partial B_s (Q)}}(v_{\epsilon})^2d\sigma} \left[2s \fint_{B_s(Q)} |\nabla v_{\epsilon}|^2dV +2s \fint_{B_s(Q)} v_{\epsilon} \Delta v_{\epsilon}dV\right]\\
& \qquad =  \frac{2}{s}\left[N(Q, s, v_{\epsilon}) + 2 \left(\frac{\int_{B_s(Q)} v_{\epsilon}\Delta v_{\epsilon}dV}{\int_{\partial B_s(Q)}(v_{\epsilon})^2d\sigma}\right)\right].
\end{align*}

Again, we wish to bound the absolute value of the negative part of the derivative.  This amounts to bounding the last term.  By Lemma \ref{N neg 1}, for all $Q \in B_{\frac{1}{4}}(0) \cap \partial \Omega^{\pm}$, $0< s$, and $\epsilon << s$, we have
\begin{align*}
& \ln(\fint_{\partial B_{\frac{5}{8}}(Q)} (v_{\epsilon})^2) - \ln(\fint_{\partial B_{\frac{9}{16}}(Q)} (v_{\epsilon})^2) \\
&\qquad  = \int_{\frac{1}{2}}^{\frac{9}{16}} \frac{d}{ds}\ln\left(\frac{1}{s^{n-1}}H(Q, s, v_{\epsilon})\right) ds\\
& \qquad \ge  \int_{\frac{1}{2}}^{\frac{9}{16}} \frac{2}{s}N(Q, s, v_{\epsilon}) - 2C s^{\alpha -1} ds\\
& \qquad \ge  2\left[N(Q, \frac{1}{2}, v_{\epsilon}) - C(\frac{5}{8})^{\alpha}\right]\ln(s)|^{\frac{5}{8}}_{\frac{9}{16}} - \frac{2C}{\alpha}s^{\alpha}|^{\frac{5}{8}}_{\frac{9}{16}}\\
&\qquad  \ge  2c\left[N(Q, \frac{1}{2}, v_{\epsilon}) - C\right] - \frac{2C}{\alpha}.
\end{align*}

Since by Lemmata \ref{Lipschitz} and \ref{v_e to v} $\lim_{\epsilon \rightarrow 0}H(Q, r, v_{\epsilon}) = H(Q, r, v)$ and $\lim_{\epsilon \rightarrow 0}N(Q, r, v_{\epsilon}) = N(Q, r, v)$, and the above estimates hold for all $\epsilon << 1/2$,  the estimates hold for $v,$ as well.  Thus, by Equation (\ref{ugly}) 
\begin{align*}
 & 2c\left[N(Q, \frac{1}{2}, v) - C\right] - \frac{2C}{\alpha} \\
 & \qquad \le  \ln\left(\frac{\fint_{\partial B_{\frac{5}{8}}(Q)} (v)^2d\sigma}{\fint_{\partial B_{\frac{9}{16}}(Q)} (v)^2d\sigma}\right)\\
 & \qquad \le \ln\left(C(n, \alpha, M_0, \Gamma) (16)^{2[N(0, 1, v) + C]} +  \frac{C(n, \alpha, M_0, \Gamma)(16)^{2[N(0, 1, v) + C]} + C(M_0, n)}{\fint_{\partial B_{\frac{9}{16}}(Q)}(v)^2 d\sigma}\right). 
\end{align*}

Now, we clean up the terms inside the logarithm. By our choice of normalization $v = v_{0, 1}$ and Lemma \ref{N neg 1}, we have that $\fint_{\partial B_{\frac{9}{16}}(Q)}(v)^2 d\sigma > c(M_0).$  Thus, we have that
\begin{align*}
\frac{C(n, \alpha,  M_0, \Gamma)(16)^{2[N(0, 1, v) + C]} + C(M_0, n)}{\fint_{\partial B_{\frac{9}{16}}(Q)}(v)^2 d\sigma} & \le
C(n, \alpha, M_0, \Gamma)\left((16)^{2[N(0, 1, v) + C]} + 1\right).
\end{align*}
Isolating for $N(Q, \frac{1}{2}, v)$, now, we have
\begin{align*}
N(Q, \frac{1}{2}, v) \le & \ln\left(C(n, \alpha, M_0, \Gamma) (16)^{2[N(0, 1, v) + C]} + 1)\right) + \frac{C}{\alpha}\\
\le & C(n, \Lambda, \alpha, M_0, \Gamma).
\end{align*}

Now, Lemma \ref{N non-degenerate almost monotonicity}, gives that for $0< a < 1/2$, if $\epsilon << s$, then 
\begin{equation}
N(Q, \frac{1}{2}, v_{\epsilon}) + C(1/2)^{\alpha} >  N(Q, s, v_{\epsilon}).
\end{equation}
Thus, again taking limits as $\epsilon \rightarrow 0$, $N(Q, s, v)<  C(n, \Lambda, \alpha, M_0, \Gamma) + C$ for all $0< s < 1/2$.
\end{proof}

\section{Quantitative Rigidity}\label{S: quant rigidity}

Throughout the rest of the paper, we shall need to use limit-compactness arguments.  The key will be that as $||\ln(h)||_{\alpha} \rightarrow 0$, $v \rightarrow u$ for some harmonic function $u.$  We make this rigorous in the following lemma.

\begin{lem}\label{better compactness lem}(Convergence to Harmonic Functions)
Let $v_i \in \mathcal{A}(\Lambda, \alpha, M_0)$ with $||\ln(h_i) ||_{\alpha} \rightarrow 0$.  Assume that  $Q_i \in B_\frac{1}{4}(0) \cap \partial \Omega^{\pm}$ and $\{r_i\} \subset (0, 1/2]$.  Then, there exists a function $v_{\infty}$ and a subsequence $v_j$ such that $T_{Q_j, r_j}v_j \rightarrow v_{\infty}$ in the senses of Lemma \ref{compactness-ish I} and $v_{\infty}$ is harmonic.
\end{lem}

\begin{proof}
We observe that Lemma \ref{N bound lem} states that $N(0, 1, T_{Q_j, r_j}v_j) \le C$.  Since $T_{Q_j, r_j}v_j,$ are uniformly Lipschitz (Lemma \ref{Lipschitz}), $T_{Q_j, r_j}v_j(0) = 0$, and $N(0, r, T_{Q_j, r_j}v_j) = \int_{B_r(0)}|\nabla T_{Q_j, r_j}v_j|^2dV$, Lemma \ref{N bound lem} states that $\{T_{Q_j, r_j}v_j\}_j$ is a bounded sequence in $W^{1,2}(B_1(0))$.

Lemma \ref{strong compactness I} gives a subsequence which converges strongly in $W^{1, 2}_{loc}(\mathbb{R}^n)$ to a function $v_{\infty}$. We claim that $v_{\infty}$ is harmonic.  To see this, we will investigate the behavior of its mollifications, $v_{\infty, \epsilon} = v_{\infty} \star \phi_{\epsilon}$.  Observe that by Young's Inequality 
$$||T_{Q_j, r_j}v_{j, \epsilon} - v_{\infty, \epsilon} ||_{L^2(B_2(0))} \le ||\phi_{\epsilon}||_{L^1(B_2(0))} ||T_{Q_j, r_j}v_j - v_{\infty}||_{L^2(B_2(0))}.$$ 

Thus, for any $\epsilon > 0$ we have $T_{Q_j, r_j}v_{j, \epsilon} \rightarrow v_{\infty, \epsilon}$ as $j \rightarrow \infty$ strongly in $L^2(B_2(0)).$  By a similar argument applied to $\nabla T_{Q_j, r_j}v_{j, \epsilon}$ we also have that $\nabla T_{Q_j, r_j}v_{j, \epsilon} \rightarrow \nabla v_{\infty, \epsilon}$ in $L^2(B_2(0); \RR^n)$ as $j \rightarrow \infty$.  Furthermore, by our uniform Lipschitz bounds,  $T_{Q_j, r_j}v_{j, \epsilon} \rightarrow v_{\infty, \epsilon}$ as $j \rightarrow \infty$ in $C(B_2(0)),$ as well.  

We will show that for $\epsilon << 1$ the function $v_{\infty, \epsilon}$ is harmonic.  First, for any test function $\xi \in C^{\infty}_c(B_2(0))$, we have that

\begin{eqnarray*}
|\int_{B_2(0)}\xi (\Delta T_{Q_j, r_j}v_{j, \epsilon} - \Delta v_{\infty, \epsilon})dV| & = & | \int_{B_2(0)} \Delta \xi ( T_{Q_j, r_j}v_{j, \epsilon} - v_{\infty, \epsilon}) dV| \\
& \le & ||\Delta \xi ||_{L^2(B_2(0))} || T_{Q_j, r_j}v_{j, \epsilon} - v_{\infty, \epsilon} ||_{L^2(B_2(0))}.
\end{eqnarray*}

Since, $T_{Q_j, r_j}v_{j, \epsilon} \rightarrow v_{\infty, \epsilon}$ strongly in $L^2(B_2(0))$, $\Delta T_{Q_j, r_j}v_{j, \epsilon} \rightharpoonup \Delta v_{\infty, \epsilon}$ in $L^2(B_2(0)).$  

However, by assumption, we also have that 
\begin{eqnarray*}
|\int_{B_2(0)}\xi \Delta T_{Q_j, r_j}v_{j, \epsilon}dV| & \le & \int_{B_2(0)} |\xi_{\epsilon}|  |\frac{h_j(0)}{h_j(x)} - 1| dT_{Q_j, r_j}\omega^- \\
& \le & C \max_{B_2(0)} |\xi| \cdot ||\ln(h_j)||_{\alpha}  T_{Q_j, r_j}\omega^-(B_{3}(0)),
\end{eqnarray*}
where $T_{Q_j, r_j}\omega^{\pm}$ are the interior and exterior harmonic measures associated to $T_{Q_j, r_j}v_{j}$.  Note that $T_{Q_j, r_j}\omega^- \not = \omega_{Q_j, r_j}^-,$ but, by Remark \ref{align rmk}, Definition \ref{L2 rescaling def}, and Lemma \ref{Lipschitz} there is a constant, $c' = c'(M_0)$, such that $T_{Q_j, r_j}\omega^{-} = c \omega_{Q_j, r_j}^{-}$ and $c \le c'$.  Since $\omega_{r_j, Q_j}^-(B_{3}(0))$ are uniformly bounded by Theorem \ref{TBE combo}, the $T_{Q_j, r_j}\omega^{-}(B_3(0))$ are, too.  Thus, as $j \rightarrow \infty$, we have that $\Delta T_{Q_j, r_j}v_{j, \epsilon} \rightharpoonup 0$ in $L^2(B_2(0))$, as well.  Thus, $\Delta v_{\infty, \epsilon} = 0$ weakly in $L^2(B_2(0))$.  Since $v_{\infty, \epsilon} \in C^{\infty}(B_2(0)),$ by classical results, $v_{\infty, \epsilon}$ is harmonic.

Since $v_{\infty}$ is Lipschitz continuous, $v_{\infty, \epsilon} \rightarrow v_{\infty}$ in $C(B_R(0))$ as $\epsilon \rightarrow 0$.  Thus, for all $x \in B_R(0)$ we have both that $v_{\infty, \epsilon}(x) \rightarrow v_{\infty}(x)$ as $\epsilon \rightarrow 0$ and that 
$$
\fint_{B_r(x)}v_{\infty, \epsilon}(y)dV(y) \rightarrow \fint_{B_r(x)}v_{\infty}(y)dV(y)
$$
 as $\epsilon \rightarrow 0$.  Thus, $v_{\infty}$ must satisfy the Mean Value Property and is therefore harmonic.  
\end{proof}

Now that we have Lemma \ref{better compactness lem}, we can prove a quantitative rigidity result.  Loosely speaking, it says that if a function $v \in \mathcal{A}(n, \Lambda, \alpha, M_0)$ behaves like a homogeneous harmonic polynomial with respect to the Almgren frequency (in the sense that if has small drop across scales), then it must be close to being a homogeneous harmonic polynomial.  This will connect the behavior of the Almgren frequency to our quantitative stratification.

 \begin{lem} (Quantitative rigidity)\label{quant rigidity}
Let $v \in \mathcal{A}(\Lambda, \alpha, M_0)$, as above.  Let $Q_i \in B_\frac{1}{4}(0) \cap \partial \Omega^{\pm}$.  For every $\delta > 0$, there is an $\gamma= \gamma(n, \Lambda, \alpha, M_0, \delta)> 0$ such that if $||\ln(h)||_{\alpha} \le \gamma$ and 
$$
N(Q, 1, v) - N(Q, \gamma, v) \le \gamma
$$
then $v$ is $(0, \delta, 1, p)$-symmetric.
\end{lem}

\begin{proof}
We argue by contradiction.  Assume that there exists a $\delta >0$ such that there is a sequence of functions, $v_i \in \mathcal{A}(n, \Lambda, \alpha, M_0)$ with $||\ln(h_i)||_{\alpha} \le 2^{-i}$ for which there exists a point, $Q_i$ with
$$
N(Q_i, 1, v_i) - N(Q_i, 2^{-i}, v_i) \le 2^{-i},
$$
but that no $v_i$ is $(0, \delta, 1, Q_i)$-symmetric.

By Lemma \ref{better compactness lem} there exists a subsequence $T_{Q_j, 1}v_j$ which converges strongly in $W^{1, 2}_{loc}$ to a harmonic function, $v_{\infty}$.  Therefore   $N(Q, r, v_{\infty})$ is monotone increasing.  Further, by Corollary \ref{N W convergence cor} we know that $\lim_{j \rightarrow \infty} N(0, r, T_{Q_j, 1}v_j) = N(0, r, v_{\infty})$.  By Lemma \ref{global to local}, and the nature of convergence, we have that 
$$
N(0, 1, v_{\infty}) - N(0, 0, v_{\infty}) = 0
$$

By classical results, this implies that $v_{\infty}$ is a homogeneous harmonic polynomial.  Thus, we arrive at our contradiction, since $T_{Q_j, 1}v_j$ were assumed to stay away from all such functions in $L^2(B_1(0))$.
\end{proof}

\begin{rmk}
Since $N(Q, r, v)$ is scale-invariant, Lemma \ref{quant rigidity} is also scale-invariant in the sense that if $N(Q, r, v) - N(Q, \gamma r, v) \le \gamma$ and $||\ln(h)||_{\alpha} \le \gamma$, then $v$ is $(0, \delta, r, Q)$-symmetric.
\end{rmk}

\section{A Dichotomy}\label{S:dichotomy}

The proof technique in the rest of the paper is an adaptation of techniques developed by Naber and Valtorta in \cite{NaberValtorta17-1}.

This section is dedicated to proving a lemma that gives us geometric information on the quantitative strata.  Roughly, it says that if we can find $(k+1)$ points that are well-separated and the Almgren frequency has very small drop at these points, then the quantitative strata is contained in a neighborhood of the affine $k$-plane which contains them and we have control on the Almgren frequency for all points in that neighborhood.  This is a quantitative analog of the following classical result.

\begin{prop}\label{cone-splitting harmonic}
Let $P: \RR^n \rightarrow \RR$ be a homogeneous harmonic polynomial.  Let $0 \le k \le n-2$.  If $P$ is translation invariant with respect to some $k$-dimensional subspace $V$ and $P$ is homogeneous with respect to some point $x \not\in V$, then $P$ is $k + 1$-symmetric with respect to $span\{x, V\}$.  
\end{prop}

See \cite{CheegerNaberValtorta15}, (Proposition 2.11), or \cite{HanLin_nodalsets} in the proof of Theorem 4.1.3.

We shall use the notation $\langle y_0, ... , y_k \rangle$ to denote the $k$-dimensional affine linear subspace which passes through $y_0, ... , y_k$.

\begin{lem}\label{geometric control lem}
Let $v \in \mathcal{A}(\Lambda, \alpha, M_0)$ and $0<\epsilon$ be fixed.  Let $\gamma, \eta', \rho>0$ be fixed, then there exist constants, $0<\eta_0(n, \Lambda, \alpha, E_0, \epsilon, \eta', \gamma, \rho) << \rho$ and a $\beta(n, \Lambda, \alpha, E_0, \epsilon, \eta', \rho) <1$ such that the following holds.  If $\eta \le \eta_0$ and 

\begin{enumerate}
\item $E= \sup_{Q \in B_1(0) \cap \partial \Omega^{\pm}}N( Q, 2, v) \in [0, E_0]$
\item There exist points $\{y_0, y_1, ... , y_k \}\subset B_1(0)\cap \partial \Omega^{\pm}$ satisfying $y_i \not \in B_{\rho}(\langle y_0, ..., y_{i-1}, y_{i+1}, ... , y_k \rangle)$ and 
$$
N(y_i, \gamma \rho, v) \ge E-\eta_0
$$
for all $i = 0, 1, ..., k$.
\item $||\ln(h)||_{\alpha} \le \eta$.
\end{enumerate}

Then, if we denote $\langle y_0, ..., y_{k}\rangle = L$, for all $Q \in B_{\beta}(L) \cap B_1(0)\cap \partial \Omega^{\pm}$
\begin{align*}
N(Q, \gamma \rho, v) & \ge E - \eta',
\end{align*}
and 
\begin{align*}
\mathcal{S}^k_{\epsilon, \eta_0} \cap B_1(0) & \subset B_{\beta}(L).
\end{align*}
\end{lem}

\begin{proof}
There are two conclusions. We argue by contradiction for both.  Suppose that the first claim fails.  That is, assume that there exist constants $\gamma, \rho, \eta' >0$ for which there exists a sequence of $v_i \in \mathcal{A}(n, \Lambda, \alpha, M_0)$ with $\sup_{Q \in B_1(0)} N(Q, 2, v_i) = E_i \in [0, E_0]$ and points $\{y_{i, j} \}_{j}$ satisfying (2), above, with $\eta_0 <2^{-i}$, $||\ln(h_i)||_{\alpha} \le 2^{-i}$, and a sequence $\beta_i \le 2^{-i}$ such that for each $i$, there exists a point $x_i \in B_{\beta_i}(L_i) \cap B_1(0)\cap \partial \Omega^{\pm}$ for which $N(x_i, \gamma \rho, v_i) < E - \eta'$.

By Lemma \ref{better compactness lem}, there exists a subsequence $v_j$ such that $T_{0, 1}v_j$ converges to a harmonic function $v_{\infty}$ in the senses outlined in the lemma. Further, by the compactness of $[0, E_0]$, $\overline{B_1(0)}$, and the Grassmannian we may assume that
$$E_j \rightarrow E \quad y_{ij} \rightarrow y_i \quad L_j \rightarrow L \quad x_j \rightarrow x_{\infty} \in \bar B_1(0)\cap \partial \Omega^{\pm}.$$  Note that the convergence given by Lemma \ref{better compactness lem} implies 
$$
\sup_{Q \in B_1(0)}N( Q, 2, v_{\infty}) \le E, \qquad N(x_{\infty}, \gamma \rho, v_{\infty}) < E - \eta',
$$
and
$$
N(y_i, 0, v_{\infty}) \ge E
$$
for all $j= 0, 1, ..., k$.  Because $v_{\infty}$ is harmonic, $N(Q, r, v_{\infty})$ is non-decreasing in $r.$  Therefore, $N(y_i, r, v_{\infty}) = E$ for all $y_i$ and all $r \in [0, 2]$.  Classical results then imply that $v_{\infty}$ is $0$-symmetric in $B_2(y_j)$ for each $y_j$.  Because the $y_j$ are in general position, by Proposition \ref{cone-splitting harmonic}, $v_{\infty}$ is translation invariant along $L$ in $B_{1 +\delta}(0) \subset \cap_j B_2(y_j)$, where $\delta >0$ depends upon the placement of the $y_j \in \overline{B_1(0)}$.  Since $x_{\infty} \in L$, this implies that $N(x_{\infty}, 0, v_{\infty}) = E$.  But this contradicts $N( x_{\infty}, \gamma \rho, v_{\infty}) < E - \eta'$, since $N( x_{\infty}, r, v_{\infty})$ must be non-decreasing in $r$.  This proves the first claim.

Now assume that the second claim fails.  That is, fix $\beta >0$ and assume that there is a sequence of $v_i \in \mathcal{A}(\Lambda, \alpha, M_0)$ with $\sup_{Q \in B_1(0)} N(Q, 2, v_i) = E_i \in [0, E_0]$ and points $\{y_{i, j} \}_{j}$ satisfying (2), above, with $||\ln(h_i)||_{\alpha} \le 2^{-i}$ and a sequence of $\eta_i \rightarrow 0$ such that for each $i$ there exists a point $x_i \in \mathcal{S}^k_{\epsilon, \eta_i}(v_i) \cap B_1(0) \setminus B_{\beta}(L_i).$

Again, we extract a subsequence as above.  The function $v_{\infty}$ will be harmonic and $k$-symmetric in $B_{1+ \delta}(0)$, as above.  And, $x_i \rightarrow x \in \overline{ B_1(0)} \setminus B_{\beta}(L).$  Note that by our definition of $\mathcal{S}^k_{\epsilon, \eta_i}(v_i)$ and Lemma \ref{better compactness lem}, $x \in \mathcal{S}^k_{\epsilon/2}(v_{\infty}).$  

Since $v_{\infty}$ is $k$-symmetric in $B_{1+ \delta}(0)$, every blow-up at a point in $\overline{B_1(0)}\setminus B_{\beta}(L)$ will be $(k+1)$-symmetric.  Thus, there must exist a radius, $r$ for which $v_{\infty}$ is $(k+1, \epsilon/4, r, x)$-symmetric.  This contradicts the conclusion that $x \in \mathcal{S}^k_{\epsilon/2}(v_{\infty}).$ 
\end{proof}

Consider the following dichotomy: either we can find  $(k+1)$ well-separated points, $y_{ij}$, with very small drop in frequency or we cannot.  In the former case, Lemma \ref{geometric control lem} implies that the Almgren frequency has small drop on all of $\mathcal{S}^k_{\epsilon, \eta}(v)$ (and we also get good geometric control).  In the later case, the set on which the Almgren frequency has small drop is close to a $(k-1)$-plane.  In this latter case, even though we have no geometric control on $\mathcal{S}^k_{\epsilon, \eta}(v)$, we have very good packing control on the part with small drop in frequency.  We make this formal in the following corollary.

\begin{cor}(Key Dichotomy)\label{key dichotomy}
Let $\gamma, \rho, \eta'  \in (0, 1)$ and $0<\epsilon$ be fixed. There is an $\eta_0= \eta_0(n, \Lambda, \alpha, E_0, \epsilon, \eta', \gamma, \rho) << \rho$ so that the following holds.  For all $v \in \mathcal{A}(\Lambda, \alpha, M_0)$, with $\sup_{Q \in B_1(0)}N(Q, 2, v) \le E \in [0, E_0]$, if $\eta \le \eta_0$ and $||\ln(h)||_{\alpha} \le \eta$, then one of the following possibilities must occur:
\begin{itemize}
\item[1.] $N(Q, \gamma \rho, v) \ge E - \eta'$ on $S^k_{\epsilon, \eta_0}(v) \cap B_1(0),$ and
$$
\mathcal{S}^k_{\epsilon, \eta_0} \cap B_1(0) \subset B_{\beta}(L).
$$

\item[2.] There exists a $(k-1)$-dimensional affine plane, $L^{k-1}$, such that $$\{Q: N(Q, 2\eta, v) \ge E - \eta_0 \} \cap B_1(0) \subset B_{\rho}(L^{k-1}).$$
\end{itemize}
\end{cor}

\begin{rmk}
The former case is simply the conclusion of Lemma \ref{geometric control lem}.  In the later case of the dichotomy, we know that all points in $B_1(0) \setminus B_{\rho}(L^{k-1})$ must have $N(2\eta, Q, v) < E-\eta_0.$  Since $N(Q, r, v)$ is almost monotonic and uniformly bounded, this can only happen for each $p$ finitely many times.
\end{rmk}

\section{Spatial Derivatives of the Almgren Frequency}\label{S:spatial derivatives}

The main result of this section is Corollary \ref{frequency pinching triangle inequality error}, in which we estimate the difference between the Almgren frequency at nearby points. 

We need a preliminary estimate.
\begin{lem}\label{l:H comparable}
Let $u \in \mathcal{A}(\Lambda, \alpha, M_0)$ Let $0<r \le \frac{1}{2}$ and $Q \in \partial \Omega^{\pm} \cap B_\frac{1}{4}(0)$ and $p \in B_{\frac{1}{3}r}(Q).$ Then there is a constant $1< c(n, \Lambda, \alpha, M_0, \Gamma) <\infty$ such that
\begin{align*}
c^{-1} < \frac{H(p, r, u)}{H(Q, r, u)}<c.
\end{align*}
\end{lem}

\begin{proof}
Suppose that the lemma is false.  Suppose that there exists a sequence of $ v_i$, radii $0<r_i \le \frac{1}{2}$, with points $Q^i, p^i \in \overline{\Omega} \cap B_{\frac{1}{4}}(0)$ such that  
$$
\frac{H(x^i_2, r_i,  v_i)}{H(Q^i, r_i,  v_i)}<2^{-i}.
$$
We note that $\int_{\partial B_{r_i}(Q^i)} v_i^2d\sigma = r_i^{n-1}\int_{\partial B_1(0)} v_i(r_ix + Q^i)^2d\sigma(x)$ and hence, the assumption is equivalent to
$$
\frac{H(T_{Q^i, r_i}x^i_2, 1, T_{Q^i, r_i} v_i)}{H(0, 1, T_{Q^i, r_i} v_i)} <2^{-i}.
$$
By Lemma \ref{strong compactness I}, we may extract a subsequence $T_{Q^i, r_i} v_i$ which converges in $C^{0}(B_2(0))$ to a function $u_{\infty}$.  Thus, $H(0, 1, u_{\infty}) = H(0, 1, T_{Q^i, r_i} v_i) = 1$.

However, we may find a subsequence such that $T_{Q^i, r_i}x^i_2 \rightarrow x_2 \in \overline{B_{\frac{1}{3}}(0)}$ for which 
$$
H(0, x_2, u_{\infty}) = \lim_{j \rightarrow \infty}H(T_{Q^i, r_i}x^i_2, 1, T_{Q^i, r_i} v_i) = 0.
$$
Thus, by the Maximum Principle, $u_{\infty} \equiv 0$.  This contradicts $H(0, 1, u_{\infty}) = 1$.
\end{proof}

\begin{lem} \label{H comparable rmk}
Let $v \in \mathcal{A}(\Lambda, \alpha, M_0)$, $||\ln(h)||_{\alpha} \le \Gamma$, and let $Q \in B_{\frac{1}{2}}(0) \cap \partial \Omega^{\pm}$ and let $0< s \le \frac{1}{2}$.  For any $p \in B_\frac{s}{3}(Q) \cap \overline{\Omega^{\pm}}$.
$$\int_{\partial B_s(p)}|v|^2 d\sigma \ge C(n, \alpha, M_0, \Gamma) \frac{\omega^-(B_s(Q))^2}{s^{n-3}} $$
 \end{lem}

\begin{proof} 
Let $x_{\text{max}}(p, s)$ to denote the point in $\partial B_s(p) \cap \Omega^-$ which maximizes $|v|$ on $\partial B_s(p) \cap \Omega^-$.  We observe that
\begin{align*}
\int_{\partial B_s(p)}|v|^2 d\sigma& \ge \int_{\partial B_s(p) \cap B_{\delta s}(x_{\text{max}}(p, s))}|v|^2d\sigma\\
& \ge C(n, M_0) |v(x_{\text{max}}(p, s))|^2 s^{n-1}\\
& \ge C(n, M_0) \frac{\omega^-(B_s(Q))^2}{s^{n-3}},
\end{align*}
where $\delta = \delta (M_0, \alpha, \Gamma) >0$ is a proportional constant depending only upon the Lipschitz constant of $v$. The last inequality comes from the NTA estimates $|v(x_{\text{max}}(p, s))| \sim  |v(A_{2s}^-(Q))| \sim \frac{\omega^-(B_s(Q))}{s^{n-2}}$ (see \cite{Engelstein16}Lemma 5.4).

 \end{proof}

\begin{lem}\label{N neg 2}
Let $v \in \mathcal{A}(\Lambda, \alpha, M_0)$ with $||\ln(h) ||_{\alpha} \le \Gamma$, and let $Q \in B_{\frac{1}{4}}(0) \cap \partial \Omega^{\pm}$ and let  $0< s \le \frac{1}{2}$ and $\epsilon << s$.  For all $p \in B_\frac{s}{3}(Q) \cap \overline{\Omega^-}$ and all vectors $|\vec{v}| \le r$. 
\begin{equation}
|\int_{B_s(p)}\langle \nabla v_{\epsilon}, \vec{v}\rangle\Delta v_{\epsilon} dV(x)| \le C||\ln(h) ||_{\alpha}s^{\alpha}\frac{\omega^-(B_s(Q))^2}{s^{n-2}},
 \end{equation}
 where the constant $C= C(M_0)$.
 \end{lem}
\begin{proof}
Let $p, Q,$ and $s$ be as above.  
\begin{align*}
|\int_{B_s(p)}\langle \nabla v_{\epsilon}, \vec{v} \rangle \Delta v_{\epsilon} dV(x)| &= |\int_{B_s(p)}\langle \nabla v_{\epsilon}, \vec{v} \rangle_{\epsilon}\Delta v(x) |\\
& \le  \int_{B_s(p)}|(\langle\nabla v_{\epsilon}, \vec{v} \rangle)_{\epsilon}| |\frac{h(0)}{h(x)} - 1| d\omega^- \\
& \le  \int_{B_{2s}(Q)}|(\langle \nabla v_{\epsilon}, \vec{v} \rangle)_{\epsilon} ||\frac{h(0)}{h(x)} - 1| d\omega^- \\
& \le  ||\ln(h) ||_{\alpha} (2s)^{\alpha}\int_{B_{2s}(Q)}|(\langle \nabla v_{\epsilon}, \vec{v} \rangle)_{\epsilon}| d\omega^- \\
& \le  ||\ln(h) ||_{\alpha} (2s)^{\alpha+1}\int_{B_{2s}(Q)}|\nabla v_{\epsilon}|_{\epsilon} d\omega^-.
\end{align*}

Chasing through the change of variables $x = ry + Q$, we see that $\nabla_x v(x) = \frac{1}{r}\nabla_y v(ry+Q) = \frac{\omega^-(B_r(Q))}{r^{n-1}}\nabla_y v_r(y)$.  Thus, we calculate that for the change of variables $x = 2sy + Q$
\begin{align*}
|\int_{B_s(p)}\langle \nabla v_{\epsilon}, \vec{v} \rangle \Delta v_{\epsilon} dV(x)| & \le ||\ln(h) ||_{\alpha} (2s)^{\alpha+1} \frac{\omega^-(B_{2s}(0))^2}{(2s)^{n-1}}\int_{B_{1}(0)}|\nabla v_{Q, 2s} \star \phi_{\frac{\epsilon}{2s}}| \star \phi_{\frac{\epsilon}{2s}} d\omega_{2s}^- \\
& \le  ||\ln(h) ||_{\alpha} Cs^{\alpha+1} \frac{\omega^-(B_{2s}(0))^2}{(s)^{n-1}}\omega_{Q, 2s}^- (B_{2}(0))\\ 
& \le ||\ln(h) ||_{\alpha} Cs^{\alpha}\frac{\omega^-(B_s(0))^2}{s^{n-2}},
\end{align*}
where the last two inequalities are because $v_{Q,r}$  are uniformly locally Lipschitz, $1 + \epsilon/r < 2$, $\omega_{Q, r}^-(B_2(0))$ are uniformly bounded for $Q \in B_1(0)$ and $r<2$, and the doubling of harmonic measure on NTA domains.
\end{proof}

\begin{lem}
For $v \in \mathcal{A}(\Lambda, \alpha, M_0)$ and $Q \in \partial \Omega^{\pm} \cap B_{\frac{1}{2}}(0)$ and $0< r$.  Then for $p \in \overline{\Omega^{\pm}} \cap B_\frac{r}{3}(Q)$ and $\vec{v} \in \mathbb{R}^n$ such that $|\vec{v}| \le r$, we calculate the spatial directional derivatives as follows. 
\begin{align}
    \frac{\partial}{\partial \vec{v}}H (p, r, v_{\epsilon}) & = 2\int_{\partial B_r(p)}v_\epsilon \nabla v_\epsilon \cdot \vec{v} d\sigma\\
    \frac{\partial}{\partial \vec{v}}D (p, r, v_{\epsilon}) & = 2\int_{\partial B_r(p)}(\nabla v_{\epsilon} \cdot \vec{v})(\nabla v_{\epsilon} \cdot \eta) d\sigma -\int_{B_r(p)}\frac{\partial}{\partial \vec{v}}v_{\epsilon} \Delta v_{\epsilon}dV\\
\label{e:N spatial derivative} \nonumber   \frac{\partial}{\partial \vec{v}}N (p, r, v_{\epsilon}) & = \frac{2}{H (Q, r, v_{\epsilon})}\left(\int_{\partial B_r(p)}\left(r \nabla v_{\epsilon} \cdot \eta - N (Q, r, v_{\epsilon})v_{\epsilon} \right)(\nabla v_{\epsilon} \cdot \vec{v}) d\sigma \right) \\
& - \frac{r\int_{B_r(p)}\frac{\partial}{\partial \vec{v}}v_{\epsilon} \Delta v_{\epsilon}dV}{H (Q, r, v_{\epsilon})}.
\end{align}
\end{lem}

\begin{proof}
Equation (\ref{e:N spatial derivative}) follows immediately from the preceding equations.  The spatial derivative for $H (Q, r u)$ follows from differentiating inside the integral. To obtain the spatial derivative for $D (Q, r, v),$ we recall the Divergence theorem.
\begin{align*}
    \frac{\partial}{\partial \vec{v}}D (Q, r, v) & = \frac{\partial}{\partial \vec{v}}\left( \int_{\partial B_r(p)}v \nabla v \cdot \left(\frac{x-p}{|x-p|}\right)d\sigma(x) - \int_{B_r(p)}v_{\epsilon}\Delta v_{\epsilon}dV\right)\\
    & = \int_{\partial B_r(p)}(\nabla v_{\epsilon} \cdot \vec{v})(\nabla v_{\epsilon} \cdot \left(\frac{x-p}{|x-p|}\right))d\sigma + \int_{\partial B_r(p)} v_{\epsilon}\frac{\partial}{\partial \vec{v}}(\nabla v_{\epsilon} \cdot \left(\frac{x-p}{|x-p|}\right))d\sigma.
\end{align*}
Now, we focus upon the last term. Recalling Green's theorem and the fact that partial derivatives of harmonic functions are themselves harmonic
\begin{align*}
    \int_{\partial B_r(p)} v_{\epsilon}\frac{\partial}{\partial \vec{v}}(\nabla v_{\epsilon} \cdot \left(\frac{x-p}{|x-p|}\right))d\sigma & = \int_{\partial B_r(p)} v_{\epsilon} \nabla \left(\frac{\partial}{\partial \vec{v}}v_{\epsilon}\right) \cdot \left(\frac{x-p}{|x-p|}\right)d\sigma\\
    & = \int_{\partial B_r(p)} \nabla v_{\epsilon} \cdot \left(\frac{x-p}{|x-p|}\right) \frac{\partial}{\partial \vec{v}}v_{\epsilon}d\sigma -\int_{B_r(p)}\frac{\partial}{\partial \vec{v}}v_{\epsilon} \Delta v_{\epsilon}dV\\
    & = \int_{\partial B_r(p)}(\nabla v_{\epsilon} \cdot \vec{v})(\nabla v_{\epsilon} \cdot v) d\sigma -\int_{B_r(p)}\frac{\partial}{\partial \vec{v}}v_{\epsilon} \Delta v_{\epsilon}dV.
\end{align*}
\end{proof}

\begin{definition}
For the sake of concision, we define the following notation for $v \in \mathcal{A}(\Lambda, \alpha, M_0)$, $y \in \overline{\Omega}$, and radii $0< r, R \le 2$.
\begin{align}\label{E def}
E_y(z) & := \nabla v_{\epsilon}(z) \cdot (z-y) - N (y, |z-y|, v_{\epsilon})v_{\epsilon}(z)\\
W_{r, R}(y) & := N (y, R, v_{\epsilon}) - N (y, r, v_{\epsilon}).
\end{align}
\end{definition}

\begin{lem}
Let $v \in \mathcal{A}(\Lambda, \alpha, M_0)$ with $||\ln(h) ||_{\alpha} \le \Gamma$.  Let $Q \in \partial \Omega^{\pm} \cap B_\frac{1}{2}(0)$, and $0<r \le 1$.  Let $Q' \in \partial \Omega^{\pm} \cap B_\frac{r}{3}(Q)$ satisfy and $p \in [Q, Q']$. Then, for $\vec{v} = Q' - Q$ and $0< \epsilon <<r$
\begin{align*}
\label{e:N spatial derivative bound}
\left|\frac{\partial}{\partial \vec{v}}N (p, r, v_{\epsilon}) \right| \lesssim_{n, \Lambda, \alpha, M_0, \Gamma} & 2(W_{\frac{1}{2}r, 2r}(Q)+ W_{\frac{1}{2}r, 2r}(Q') + C \Gamma r^{\alpha})\left(r\left(\frac{\int_{\partial B_r(p)}|\nabla v_{\epsilon}|^2d\sigma}{H (p, r, v_{\epsilon})}\right)^{\frac{1}{2}}+1\right)\\
& + \frac{2}{H (p, r, v_{\epsilon})} \left(\int_{\partial B_r(p)}|E_{Q}(z)|^2 + |E_{Q'}(z)|^2 d\sigma \right)^{\frac{1}{2}} \left(\int_{\partial B_r(p)}(\nabla v_{\epsilon}(z) \cdot (z-p))^2 d\sigma \right)^{\frac{1}{2}}\\
& + C(n, \Lambda)\left(\frac{1}{H (p, r, v_{\epsilon})} \int_{\partial B_r(p)}|E_{Q}(z)|^2 + |E_{Q'}(z)|^2 d\sigma\right)^{\frac{1}{2}}\\
& + 2(N (Q, r, v_{\epsilon}) - N (Q', r, v_{\epsilon})) \left(\frac{\int_{B_r(p)} v_{\epsilon}\Delta v_{\epsilon}dV}{H (p, r, v_{\epsilon})}\right) + r^{\alpha}\Gamma.
\end{align*}
\end{lem}

\begin{proof}
We begin by noting that Lemma \ref{N neg 2} and Lemma \ref{H comparable rmk} give
\begin{align*}
    \left|\frac{r\int_{B_r(p)}\frac{\partial}{\partial \vec{v}}v_{\epsilon} \Delta v_{\epsilon}dV}{H (Q, r, v_{\epsilon})}\right| \le C(n, \alpha, M_0, \Gamma)r^{\alpha}\Gamma.
\end{align*}

Now, we decompose
\begin{align*}
    \nabla v_{\epsilon} \cdot (Q - Q') & = \nabla v_{\epsilon} \cdot (z-Q') - \nabla v_{\epsilon} \cdot (z - Q)\\
    & = (N (Q, |z-Q|, v_{\epsilon}) - N (Q', |z-Q'|, v_{\epsilon}))v_{\epsilon} + (E_{Q}(z) - E_{Q'}(z)).
\end{align*}
Therefore, plugging this into Equation (\ref{e:N spatial derivative}), we obtain for $v = Q' - Q$
\begin{align*}
    &\frac{\partial}{\partial \vec{v}}N (p, r, v_{\epsilon}) \\
    & = \frac{2}{H (p, r, v_{\epsilon})} \left(\int_{\partial B_r(p)}E_{p}(z)((N (Q, |z-Q|, v_{\epsilon}) - N (Q', |z-Q'|, v_{\epsilon}))v_{\epsilon} + (E_{Q}(z) - E_{Q'}(z))) d\sigma \right)\\
    & = A + B - C,
\end{align*}
where we define
\begin{align*}
    & A := \frac{2}{H (p, r, v_{\epsilon})} \int_{\partial B_r(p)}E_{p}(z)(N (Q, |z-Q|, v_{\epsilon}) - N (Q', |z-Q'|, v_{\epsilon}))v_{\epsilon} d\sigma\\
    &B := \frac{2}{H (p, r, v_{\epsilon})} \int_{\partial B_r(p)} \nabla v_{\epsilon}(z) \cdot (z-p) (E_{Q}(z) - E_{Q'}(z)) d\sigma \\
    &C := \frac{2}{H (p, r, v_{\epsilon})} \int_{\partial B_r(p)} N (Q, |z-p|, v_{\epsilon})v_{\epsilon}(z) (E_{Q}(z) - E_{Q'}(z)) d\sigma.
\end{align*}

We begin by estimating $A$.  We rewrite
\begin{align*}
    N (Q, |z-Q|, v_{\epsilon}) - N (Q', |z-Q'|, v_{\epsilon}) & = N (Q, r, v_{\epsilon}) - N (Q', r, v_{\epsilon})\\ 
    & + W_{|z -Q|, r}(Q) + W_{|z -Q'|, r}(Q').
\end{align*}
Note that if $|Q - Q'| \le \frac{r}{3}$ and $p \in [Q, Q']$,  $\frac{r}{2} \le |z-x_i| \le 2r$ for $i = 1, 2$ and all $z \in \partial B_r(p).$ Therefore, by Lemma \ref{global to local}, for all $z \in \partial B_r(p).$
\begin{align*}
    |W_{|z -Q|, r}(Q)| & \le W_{\frac{1}{2}r, 2r}(Q) + 2C_1\Gamma (2r)^{\alpha}\\
    |W_{|z -Q'|, r}(Q')| & \le W_{\frac{1}{2}r, 2r}(Q') + 2C_1\Gamma(2r)^{\alpha}.
\end{align*}
Furthermore, we estimate by the divergence theorem
\begin{align*}
& \int_{\partial B_r(p)}E_{p}(z)(N (Q, r, v_{\epsilon}) - N (Q', r, v_{\epsilon}))v_{\epsilon} d\sigma \\
& = (N (Q, r, v_{\epsilon}) - N (Q', r, v_{\epsilon}))\int_{\partial B_r(p)} v_{\epsilon}\nabla v_{\epsilon}(z) \cdot (z-y) - N (p, |z-p|, v_{\epsilon})v_{\epsilon}^2 d\sigma\\
& = (N (Q, r, v_{\epsilon}) - N (Q', r, v_{\epsilon})) \cdot (0+ \int_{B_r(p)} v_{\epsilon}\Delta v_{\epsilon}dV).
\end{align*}
Thus, we may estimate $A$ using Cauchy-Schwartz
\begin{align*}
    |A| & \le (W_{\frac{1}{2}r, 2r}(Q)+ W_{\frac{1}{2}r, 2r}(Q') + C \Gamma r^{\alpha}) \frac{2}{H (p, r, v_{\epsilon})}\int_{\partial B_r(p)}|E_p(z)||v_{\epsilon}|d\sigma\\
    & + 2(N (Q, r, v_{\epsilon}) - N (Q', r, v_{\epsilon})) \left(\frac{\int_{B_r(p)} v_{\epsilon}\Delta v_{\epsilon}dV}{H (p, r, v_{\epsilon})}\right)\\
    & \le (W_{\frac{1}{2}r, 2r}(Q)+ W_{\frac{1}{2}r, 2r}(Q') + C \Gamma r^{\alpha}) \frac{2}{H (Q, r, v_{\epsilon})}\int_{\partial B_r(p)} r|v_{\epsilon}\nabla v_{\epsilon} \cdot \eta| + |v_{\epsilon}|^2d\sigma\\
    & + 2(N (Q, r, v_{\epsilon}) - N (Q', r, v_{\epsilon})) \left(\frac{\int_{B_r(p)} v_{\epsilon}\Delta v_{\epsilon}dV}{H (p, r, v_{\epsilon})}\right)\\
    & \le (W_{\frac{1}{2}r, 2r}(Q)+ W_{\frac{1}{2}r, 2r}(Q') + C \Gamma r^{\alpha})2\left(r\left(\frac{\int_{\partial B_r(p)}|\nabla v_{\epsilon}|^2d\sigma}{H (Q, r, v_{\epsilon})}\right)^{\frac{1}{2}}+1\right)\\
    & + 2(N (Q, r, v_{\epsilon}) - N (Q', r, v_{\epsilon})) \left(\frac{\int_{B_r(p)} v_{\epsilon}\Delta v_{\epsilon}dV}{H (p, r, v_{\epsilon})}\right).
\end{align*}

Now we estimate $B$ using Cauchy-Schwartz and Lemma \ref{Lipschitz}. 
\begin{align*}
    &\left(\frac{2}{H (p, r, v_{\epsilon})} \int_{\partial B_r(p)} \nabla v_{\epsilon}(z) \cdot (z-p) (E_{Q}(z) - E_{Q'}(z)) d\sigma \right)^2\\
    & \qquad \le \frac{4}{H (p, r, v_{\epsilon})^2}\int_{\partial B_r(p)}(E_{Q}(z) - E_{Q'}(z))^2 d\sigma \int_{\partial B_r(p)}(\nabla v_{\epsilon}(z) \cdot (z-p))^2 d\sigma\\
    & \qquad \le \frac{4}{H (p, r, v_{\epsilon})^2} \int_{\partial B_r(p)}|E_{Q}(z)|^2 + |E_{Q'}(z)|^2 d\sigma \left(\int_{\partial B_r(p)}(\nabla v_{\epsilon}(z) \cdot (z-p))^2 d\sigma \right).
\end{align*}
The same Cauchy-Schwartz argument for $|C|$ shows that
\begin{align*}
    |C| & \lesssim_{n, \Lambda, \alpha, M_0, \Gamma}\left(\frac{1}{H (Q, r, v_{\epsilon})} \int_{\partial B_r(p)}|E_{Q}(z)|^2 + |E_{Q'}(z)|^2 d\sigma\right)^{\frac{1}{2}}.
\end{align*}
This proves the point-wise estimate
\begin{align*}
&\frac{\partial}{\partial v}N (p, r, v_{\epsilon}) \lesssim_{n, \Lambda, \alpha, M_0, \Gamma}
(W_{\frac{1}{2}r, 2r}(Q)+ W_{\frac{1}{2}r, 2r}(Q') + C \Gamma r^{\alpha})2\left(r\left(\frac{\int_{\partial B_r(p)}|\nabla v_{\epsilon}|^2d\sigma}{H (p, r, v_{\epsilon})}\right)^{\frac{1}{2}}+1\right)\\
& \qquad + \frac{2}{H (p, r, v_{\epsilon})} \left(\int_{\partial B_r(p)}|E_{Q}(z)|^2 + |E_{Q'}(z)|^2 d\sigma \right)^{\frac{1}{2}} \left(\int_{\partial B_r(p)}(\nabla v_{\epsilon}(z) \cdot (z-p))^2 d\sigma \right)^{\frac{1}{2}}\\
& \qquad + \left(\frac{1}{H (p, r, v_{\epsilon})} \int_{\partial B_r(p)}|E_{Q}(z)|^2 + |E_{Q'}(z)|^2 d\sigma\right)^{\frac{1}{2}}\\
& \qquad + 2(N (Q, r, v_{\epsilon}) - N (Q', r, v_{\epsilon})) \left(\frac{\int_{B_r(p)} v_{\epsilon}\Delta v_{\epsilon}dV}{H (p, r, v_{\epsilon})}\right) + r^{\alpha}\Gamma.
\end{align*}
We prove the lemma by reversing the roles of $Q, Q'.$
\end{proof}

\begin{cor}\label{frequency pinching triangle inequality error}
Let $v \in \mathcal{A}(\Lambda, \alpha, M_0)$ with $||\ln(h) ||_{\alpha} \le \Gamma$.  Let $Q \in \partial \Omega^{\pm} \cap B_\frac{1}{2}(0)$, and $0<r \le 1$.  Let $Q' \in \partial \Omega^{\pm} \cap B_\frac{r}{3}(Q)$. Then,
\begin{align*}
    |N (Q', r, v) - N (Q, r, v)|& \lesssim_{n, \Lambda, \alpha, M_0, \Gamma} (W_{\frac{1}{2}r, 2r}(Q)+ W_{\frac{1}{2}r, 2r}(Q') + C \Gamma r^{\alpha})(1+ r^{\frac{1}{2}})\\
    & + (W_{\frac{1}{2}r, 2r}(Q)^{\frac{1}{2}} + W_{\frac{1}{2}r, 2r}(Q)^{\frac{1}{2}})(1 + r^{\frac{1}{2}}) + r^{\alpha}\Gamma.
\end{align*}
\end{cor}

\begin{proof}
We shall show that $|N (Q', r, v_{\epsilon}) - N (Q, r, v_{\epsilon})|$ satisfies a corresponding inequality, and let $\epsilon \rightarrow 0$.  Since $N(Q', r, v_{\epsilon}) \rightarrow N(Q', r, v)$ as $\epsilon \rightarrow 0$, this will prove the claim.

Let $v = Q' - Q$ and $p_t := Q + tv$.  Then, we calculate
\begin{align*}
    & |N (Q', r, v_{\epsilon}) - N (Q, r, v_{\epsilon})| \le \int_0^1|\frac{\partial}{\partial t}N (p_t, r, v_{\epsilon})|dt \\
    & \qquad \lesssim_{n, \Lambda, \alpha, M_0, \Gamma} \int_0^1 (W_{\frac{1}{2}r, 2r}(Q)+ W_{\frac{1}{2}r, 2r}(Q') + C \Gamma r^{\alpha})2\left(r\left(\frac{\int_{\partial B_r(p_t)}|\nabla v_{\epsilon}|^2d\sigma}{H (p_t, r, v_{\epsilon})}\right)^{\frac{1}{2}}+1\right)dt \\
    & \qquad \qquad + \int_0^1 \frac{2}{H (p_t, r, v_{\epsilon})} \left(\int_{\partial B_r(p_t)}|E_{Q}(z)|^2 + |E_{Q'}(z)|^2 d\sigma \right)^{\frac{1}{2}} \left(\int_{\partial B_r(p_t)}(\nabla v_{\epsilon}(z) \cdot (z-p_t))^2 d\sigma \right)^{\frac{1}{2}} dt \\
    & \qquad \qquad + \int_0^1\left(\frac{1}{H (p_t, r, v_{\epsilon})} \int_{\partial B_r(p_t)}|E_{Q}(z)|^2 + |E_{Q'}(z)|^2 d\sigma\right)^{\frac{1}{2}}dt\\
    & \qquad \qquad + 2(N (Q, r, v_{\epsilon}) - N (Q', r, v_{\epsilon})) \int_0^1\left(\frac{\int_{B_r(p_t)} v_{\epsilon}\Delta v_{\epsilon}dV}{H (p_t, r, v_{\epsilon})}\right)dt + r^{\alpha}\Gamma.
    \end{align*}
We begin by estimating the first term.    
\begin{align*}
    & \int_0^1 (W_{\frac{1}{2}r, 2r}(Q)+ W_{\frac{1}{2}r, 2r}(Q') + C \Gamma r^{\alpha})2\left(r\left(\frac{\int_{\partial B_r(p_t)}|\nabla v_{\epsilon}|^2d\sigma}{H (p_t, r, v_{\epsilon})}\right)^{\frac{1}{2}}+1\right)dt \\
    & \qquad \le 2\left(W_{\frac{1}{2}r, 2r}(Q) + W_{\frac{1}{2}r, 2r}(Q') + C \Gamma r^{\alpha}\right)\int_0^1r \left(\frac{\int_{\partial B_r(p_t)}|\nabla v_{\epsilon}|^2d\sigma}{H (p_t, r, v_{\epsilon})}\right)^{\frac{1}{2}}d\sigma + 1dt.
\end{align*}
Observe that by Lemma \ref{l:H comparable} and Lemma \ref{H doubling-ish remark}
\begin{align*}
    \int_0^1r \left(\frac{\int_{\partial B_r(p_t)}|\nabla v_{\epsilon}|^2d\sigma}{H (p_t, r, v_{\epsilon})}\right)^{\frac{1}{2}}d\sigma dt & \lesssim_{n, \Lambda, \alpha, M_0, \Gamma}\int_0^1r \left(\frac{\int_{\partial B_r(p_t)}|\nabla v_{\epsilon}|^2d\sigma}{H (Q, r, v_{\epsilon})}\right)^{\frac{1}{2}}d\sigma dt\\
    & \lesssim_{n, \Lambda, \alpha, M_0, \Gamma} r \frac{2}{H (Q, r, v_{\epsilon})^{\frac{1}{2}}}\left(\int_{B_{2r}(Q)}|\nabla v_{\epsilon}|^2 dV\right)^{\frac{1}{2}}\\
    & \lesssim_{n, \Lambda, \alpha, M_0, \Gamma} r^{\frac{1}{2}}.
\end{align*}
By a similar argument, we estimate the terms, $\int_0^1\frac{1}{H (p_t, r, v_{\epsilon})}\int_{\partial B_r(p_t)}|E_{Q}(z)|^2 d\sigma dt$. By Lemma \ref{H comparable rmk}, Lemma \ref{H doubling-ish remark}, and Lemma \ref{N non-degenerate almost monotonicity} 
\begin{align*}
    & \int_0^t\frac{1}{H (p_t, r, v_{\epsilon})}\int_{\partial B_r(p_t)}|E_{Q}(z)|^2 d\sigma dt\\
    & \le \int_{A_{\frac{1}{2}r, 2r}(Q)} \frac{1}{H (p_t, r, v_{\epsilon})} |\nabla v_{\epsilon}(z) \cdot (z-Q) - N (Q, |z-Q|, v_{\epsilon})v_{\epsilon}(z)|^2 dV(z)\\
    & = C(n, \Lambda, \alpha, M_0, \Gamma)W_{\frac{1}{2}r, 2r}(Q)r.
\end{align*}
An identical argument holds for $Q'$ in the place of $Q$.  

We now estimate the second term. Using Cauchy-Schwartz, Lemma \ref{H comparable rmk} and the estimate on $$\int_0^1\frac{1}{H (p_t, r, v_{\epsilon})}\int_{\partial B_r(p_t)}|E_{Q}(z)|^2 d\sigma dt,$$ immediately above, we obtain
    \begin{align*}
    & \int_0^1 \frac{2}{H (p_t, r, v_{\epsilon})} \left(\int_{\partial B_r(p_t)}|E_{Q}(z)|^2 + |E_{Q'}(z)|^2 d\sigma \right)^{\frac{1}{2}} \left(\int_{\partial B_r(p_t)}(\nabla v_{\epsilon}(z) \cdot (z-p_t))^2 d\sigma \right)^{\frac{1}{2}} dt \\
    & \lesssim_{n, \Lambda, \alpha, M_0, \Gamma}  2 \left(W_{\frac{1}{2}r, 2r}(Q) + W_{\frac{1}{2}r, 2r}(Q')\right)^{\frac{1}{2}} \left(\frac{1}{H (Q, r, v_{\epsilon})}\int_{A_{\frac{1}{2}r, 2r}(Q)}(\nabla v_{\epsilon}(z) \cdot (z-p))^2 dV \right)^{\frac{1}{2}}\\
    & \lesssim_{n, \Lambda, \alpha, M_0, \Gamma} 2\left(W_{\frac{1}{2}r, 2r}(Q) + W_{\frac{1}{2}r, 2r}(Q')\right)^{\frac{1}{2}}(N (Q, 2r, v_{\epsilon})r)^{\frac{1}{2}}\\
    & \lesssim_{n, \Lambda, \alpha, M_0, \Gamma} 2\left(W_{\frac{1}{2}r, 2r}(Q) + W_{\frac{1}{2}r, 2r}(Q')\right)^{\frac{1}{2}}r^{\frac{1}{2}},
\end{align*}
where for the second inequality, we have used Lemma \ref{H doubling-ish remark}. 

We now prove bounds on 
$$
2(N (Q, r, v_{\epsilon}) - N (Q', r, v_{\epsilon})) \int_0^1\left(\frac{\int_{B_r(p_t)} v_{\epsilon}\Delta v_{\epsilon}dV}{H (p_t, r, v_{\epsilon})}\right)dt.
$$
Note that
\begin{align*}
|\int_{B_r(p_t)} v_{\epsilon} \Delta v_{\epsilon} dV(x)| &= |\int_{B_r(p)}\langle (v_{\epsilon})_{\epsilon}\Delta v(x) |\\
& \le  \int_{B_r(p)}|(v_{\epsilon})_{\epsilon}| |\frac{h(0)}{h(x)} - 1| d\omega^- \\
& \le  \int_{B_{2r}(Q)}|(v_{\epsilon})_{\epsilon} ||\frac{h(0)}{h(x)} - 1| d\omega^- \\
& \le  ||\ln(h) ||_{\alpha} (2r)^{\alpha}\int_{B_{2r}(Q)}|(v_{\epsilon})_{\epsilon}| d\omega^- \\
& \le  ||\ln(h) ||_{\alpha} (2r)^{\alpha} C\int_{B_{2r}(Q)}|(v_{\epsilon})_{\epsilon}| d\omega^-.
\end{align*}

Chasing through the change of variables $x = ry + Q$, we see that $\nabla_x v(x) = \frac{1}{r}\nabla_y v(ry+Q) = \frac{\omega^-(B_r(Q))}{r^{n-1}}\nabla_y v_r(y)$.  Thus, we calculate that for the change of variables $x = 2ry + Q$
\begin{align*}
|\int_{B_r(p)} v_{\epsilon} \Delta v_{\epsilon} dV(x)| & \le ||\ln(h) ||_{\alpha} (2r)^{\alpha} \frac{\omega^-(B_{2r}(0))^2}{(2r)^{n-1}}\int_{B_{1}(0)}|v_{Q, 2r} \star \phi_{\frac{\epsilon}{2r}}| \star \phi_{\frac{\epsilon}{2r}} d\omega_{2r}^- \\
& \le  ||\ln(h)||_{\alpha} Cr^{\alpha} \frac{\omega^-(B_{2r}(0))^2}{r^{n-1}} C(\frac{\epsilon}{r})\omega_{Q, 2r}^-(B_{2}(0))\\ 
& \le ||\ln(h) ||_{\alpha} Cr^{\alpha}C(\frac{\epsilon}{r})\frac{\omega^-(B_r(0))^2}{r^{n-2}},
\end{align*}
where the last two inequalities are because $v_{Q,r}$  are uniformly locally Lipschitz, $1 + \epsilon/r < 2$, $\omega_{Q, r}^-(B_2(0))$ are uniformly bounded for $Q \in B_1(0)$ and $r<2$, and the doubling of harmonic measure on NTA domains.

Thus, by Lemma \ref{H comparable rmk} we have
\begin{align*}
    2(N (Q, r, v_{\epsilon}) - N (Q', r, v_{\epsilon})) \int_0^1\left(\frac{\int_{B_r(p_t)} v_{\epsilon}\Delta v_{\epsilon}dV}{H (p_t, r, v_{\epsilon})}\right)dt & \le C(n, \Lambda, \alpha, M_0, \Gamma)\Gamma r^{\alpha-1}(\frac{\epsilon}{r}).
\end{align*}
Letting $\epsilon \rightarrow 0$ this term vanishes. Thus, we have
\begin{align*}
    |N (Q', r, v) - N (Q, r, v)|& \le \lim_{\epsilon \rightarrow 0}|N (Q', r, v_{\epsilon}) - N (Q, r, v_{\epsilon})|\\
    & \lesssim_{n, \Lambda, \alpha, M_0, \Gamma}(W_{\frac{1}{2}r, 2r}(Q)+ W_{\frac{1}{2}r, 2r}(Q') + C \Gamma r^{\alpha})(1+ r^{\frac{1}{2}})\\
    & + (W_{\frac{1}{2}r, 2r}(Q)^{\frac{1}{2}} + W_{\frac{1}{2}r, 2r}(Q')^{\frac{1}{2}})(1 + r^{\frac{1}{2}}) + r^{\alpha}\Gamma.
\end{align*}
This proves the lemma.
\end{proof}

\section{Frequency Pinching}\label{S: beta numbers}

In this section, we prove a ``frequency pinching" result (Lemma \ref{beta bound lem}) in the style of \cite{deLellisMarcheseSpadaroValtorta16}.  This kind of result relates the Jones $\beta$-numbers to the drop in Almgren frequency.

\begin{definition}(Jones $\beta$-numbers)\label{beta def}
For $\mu$ a Borel measure, we define $\beta_{\mu}^k (Q,r)^2$ as follows.
\begin{equation*}
\beta_{\mu}^k (Q, r)^2 = \inf_{L^k}\frac{1}{r^k}\int_{B_r(p)} \frac{\emph{dist}(x, L)^2}{r^2}d\mu(x)
\end{equation*}
where the infimum is taken over all affine $k$-planes.
\end{definition}

Taking the \textit{infimum}-- as opposed to the \textit{minimum}-- here is a convention.  The space of admissible planes is compact, so a minimizing plane exists.  Let $V_{\mu}^k(Q, r)$ denote a k-plane which minimizes the \textit{infimum} in the definition of $\beta_{\mu}^k(Q,r)^2$.  Note that this k-plane is not \textit{a priori} unique.  

\begin{lem}(Frequency pinching)\label{beta bound lem}
There exists a constant, $\delta_0 = \delta_0(n, \alpha, M_0, \Gamma)>0$ such that for any $0< \delta \le \delta_0$, if $v \in \mathcal{A}(\Lambda, \alpha, M_0)$ with $||\ln(h) ||_{\alpha} \le \Gamma$ then for any $Q \in \partial \Omega \cap B_{\frac{1}{4}}(0)$ and $0< r \le \frac{1}{4}$ if $v$ is $(0, \delta, 8r, Q)$-symmetric, but not $(k+1, \epsilon, 8r, Q)$-symmetric, then for any finite Borel measure $\mu$ supported in $B_r(Q) \cap \partial \Omega$
\begin{align}\nonumber
\beta_{\mu, 2}^k(Q, r)^2  & \lesssim_{n, \Lambda, \alpha, M_0, \Gamma, \epsilon} \frac{1}{r^{k}}\left( \int_{B_r(Q)} W_{\frac{1}{2}r, 16r}(y)^2 + W_{\frac{1}{2}r, 16r}(y)d\mu(y) \right) \\
& \qquad + (W_{\frac{1}{2}r, 16r}(Q)^2+ W_{\frac{1}{2}r, 16r}(Q) +\Gamma r^{\alpha})\frac{\mu(B_r(Q))}{r^k}.
\end{align}
\end{lem}
Before proving Lemma \ref{beta bound lem}, we prove a few preliminary lemmata. We begin by noting that for any finite Borel measure, $\mu$, and any $B_r(Q)$ we can define the $\mu$ center of mass, $X = \fint_{B_r(Q)}xd\mu(x)$, and define the covariance matrix of the mass distribution in $B_r(Q)$ by 
\begin{align*}
\Sigma = \int_{B_r(Q)}(y - X)(y-X)^{\perp}d\mu(y).
\end{align*}
With this matrix, we may naturally define a symmetric, non-negative bilinear form
$$
Q(v, w) = v^{\perp} \Sigma w = \fint_{B_r(Q)}(v\cdot (y-X))(w \cdot (y - X))d\mu(y).
$$
Let $\vec v_1, ... ,\vec v_n$ be an orthonormal eigenbasis and $\lambda_1 \ge ... \ge \lambda_n \ge 0$ their associated eigenvalues.  These objects enjoy the following relationships
$$
V_{\mu, 2}^k(Q, r)= X + \text{span}\{\vec v_1, ..., \vec v_k \} , \quad \beta_{\mu, 2}^k(x, r)^2 = \frac{\mu(B_r(Q))}{r^k}(\lambda_{k+1} + ... + \lambda_n).
$$
See \cite{Hochman15} Section 4.2.

\begin{lem}\label{beta bound lem, part 1}
Let $v \in \mathcal{A}(\Lambda, \alpha, M_0)$ and let $Q \in \partial \Omega \cap B_{\frac{1}{4}}(0)$ and $0< r \le \frac{1}{4}$.  Let $\mu, Q, \lambda_i, \vec v_i$ defined as above.  For any $i$ and any scalar $c \in \RR$
\begin{align}\nonumber \label{quadratic form bound}
& \lambda_i\frac{1}{r^{n+2}}\int_{A_{3r, 4r}(Q) }(\vec v_i \cdot \nabla v(z))^2dz   \\
 & \qquad \qquad \qquad \le  5^n \fint_{B_r(Q)} \left( \int_{A_{3r, 4r}(y)  } \frac{|cv(z) - \nabla v(z) \cdot (z - y)|^2}{|z-y|^{n+2}}dz \right) d\mu(y).
\end{align}
\end{lem}

\begin{proof} 
Observe that by the definition of center of mass $$\fint_{B_{r}(Q)}\vec w \cdot (y-X)d\mu(y) = 0$$ for any $\vec w \in \RR^n.$ Therefore, for any $z$ for which $\nabla v(z)$ is defined
\begin{align*}
\lambda_i(\vec v_i \cdot \nabla v(z)) = &~ Q(\vec v_i, \nabla v(z))\\
=& \fint_{B_r(Q)}(\vec v_i \cdot (y-X))(\nabla v(z) \cdot (y - X))d\mu(y) \\
=& \fint_{B_r(Q)}(\vec v_i \cdot (y-X))(\nabla v(z) \cdot (y - X))d\mu(y)\\
 &+ \fint_{B_{r}(Q)} c v(z)(\vec v_i \cdot (y-X))d\mu(y)\\
= &  \fint_{B_r(Q)}(\vec v_i \cdot (y-X))(c v(z) - \nabla  v(z) \cdot (X - z + z -y))d\mu(y) \\
= &  \fint_{B_r(Q)}(\vec v_i \cdot (y-X))(c v(z) - \nabla  v(z) \cdot (z - y))d\mu(y) \\
 \le & \lambda_i^{\frac{1}{2}} \left(\fint_{B_r(Q)} |c v(z) - \nabla  v(z) \cdot (z - y)|^2d\mu(y)\right)^{\frac{1}{2}}.
\end{align*}
Recalling that $A_{r, R}(Q) = B_{R}(Q) \setminus B_{r}(Q),$ we calculate
\begin{align*}
&\lambda_i\frac{1}{r^{n+2}}\int_{A_{3r, 4r}(Q)}(\vec v_i \cdot \nabla  v(z))^2dz \\
& \qquad \le \frac{1}{r^{n+2}} \int_{A_{3r, 4r}(Q)} \left( \fint_{B_r(Q)} |c v(z) - \nabla  v(z) \cdot (z - y)|^2d\mu(y)\right)dz\\
& \qquad \le 5^n \fint_{B_r(Q)} \left( \int_{A_{3r, 4r}(Q)} \frac{|c v(z) - \nabla  v(z) \cdot (z - y)|^2}{|z-y|^{n+2}}dz \right) d\mu(y)\\
& \qquad \le 5^n \fint_{B_r(Q)} \left(\int_{A_{3r, 4r}(y)} \frac{|c v(z) - \nabla  v(z) \cdot (z - y)|^2}{|z-y|^{n+2}}dz \right)d\mu(y).
\end{align*}
\end{proof}

\begin{lem}\label{fudge III}
Let $v \in \mathcal{A}(\Lambda, \alpha, M_0)$ with $||\ln(h) ||_{\alpha} \le \Gamma$ and $0 \le k \le n-2$.  Let $Q \in \partial \Omega \cap B_{\frac{1}{4}}(0)$ and $0< r \le \frac{1}{32}$.  Let $0< \epsilon$ be fixed.  There is a $0< \delta_0(n, \Lambda, \alpha, M_0, \Gamma, \epsilon)$ such that if $v$ is $(0, \delta)$-symmetric in $B_{8r}(Q)$ for any  $0<\delta \le \delta_0$ but not $(k+1, \epsilon, 8r, Q)$-symmetric, then for any $y \in B_r(Q)\cap \partial \Omega$
\begin{align*}
& \int_{A_{3r, 4r}(y) } \frac{|N(Q, 7r, v) v(z) - \nabla  v(z) \cdot (z - y)|^2}{|z-y|^{n+2}}dz\\
& \le 4 \int_{A_{2r, 7r}(y) } \frac{|N (y, |z-y|,  v) v(z) - \nabla  v(z) \cdot (z - y)|^2}{|z-y|^{n+2}}dz\\
& \qquad + C(n, \Lambda, \alpha, M_0, \Gamma)\left(W_{r, 16r}(Q)^2 + W_{r, 16r}(y)^2 + W_{\frac{1}{2}r, 16r}(y) + W_{\frac{1}{2}r, 16r}(Q) + \Gamma r^{2\alpha}\right)(\int_{B_1(0)} v(rx+p)^2d\sigma).
\end{align*}
\end{lem}

\begin{proof}
First, we observe that
\begin{align*}
N (Q, 7r,  v) = &  N (y, |z-y|,  v) + W_{|z-y|, 7r}(Q) \\
& + [N (Q, |z-y|,  v) - N (y, |z-y|,  v)].
\end{align*}
Therefore, by the triangle inequality
\begin{align*}
&|N (Q, 7r,  v) v(z) - \nabla  v(z) \cdot (z - y)|^2\\
&\quad \le (|(W_{|z-y|, 7r}(Q) + N (Q, |z-y|,  v) - N (y, |z-y|,  v)) v(z)| \\
&\qquad +|(N (y, |z-y|,  v)) v(z) - \nabla  v(z) \cdot (z - y)|)^2 \\
&\quad \le 2|(W_{|z-y|, 7r}(Q) + N (Q, |z-y|,  v) - N (y, |z-y|,  v)) v(z)|^2 \\
&\qquad +2|(N (y, |z-y|,  v)) v(z) - \nabla  v(z) \cdot (z - y)|^2. 
\end{align*}

Now, using Lemma \ref{frequency pinching triangle inequality error} at scale $|z-y|$, the monotonicity of the Almgren frequency, and the fact that $0< r \le \frac{1}{32}$ to estimate
\begin{align*}
    |N (Q, |z-y|,  v) - N (y, |z-y|,  v)| & \lesssim_{n, \Lambda, \alpha, M_0, \Gamma} (W_{\frac{1}{2}r, 2r}(Q)+ W_{\frac{1}{2}r, 2r}(y) + C \Gamma r^{\alpha})(1+ r^{\frac{1}{2}})\\
    & \qquad + (W_{\frac{1}{2}r, 2r}(Q)^{\frac{1}{2}} + W_{\frac{1}{2}r, 2r}(y)^{\frac{1}{2}})(1 + r^{\frac{1}{2}}) + r^{\alpha}\Gamma \\
    & \lesssim_{n, \Lambda, \alpha, M_0, \Gamma} (W_{\frac{1}{2}r, 2r}(Q)+ W_{\frac{1}{2}r, 2r}(y) + \Gamma r^{\alpha} \\
    &\qquad + W_{\frac{1}{2}r, 2r}(Q)^{\frac{1}{2}} + W_{\frac{1}{2}r, 2r}(y)^{\frac{1}{2}}.
\end{align*}

\begin{align*}
    & \int_{A_{2r, 7r}(y)} \frac{|(W_{|z-y|, 7r}(Q) + N (Q, |z-y|,  v) - N (y, |z-y|,  v)) v(z)|^2}{|z-y|^{n+2}}dz \\
    & \lesssim_{n, \Lambda, \alpha, M_0, \Gamma}  \left(W_{r, 16r}(Q)^2 + W_{r, 16r}(y)^2 + W_{\frac{1}{2}r, 16r}(y) + W_{\frac{1}{2}r, 16r}(Q) + \Gamma r^{2\alpha}\right) \int_{A_{2r, 7r}(y) } \frac{| v(z)|^2}{|z-y|^{n+2}}dz\\
    & \lesssim_{n, \Lambda, \alpha, M_0, \Gamma}\left(W_{r, 16r}(Q)^2 + W_{r, 16r}(y)^2 + W_{\frac{1}{2}r, 16r}(y) + W_{\frac{1}{2}r, 16r}(Q) + \Gamma r^{2\alpha}\right)\left(\int_{B_1(0)} v(rx+Q)^2d\sigma \right).
\end{align*}
In the last inequality, we have use the fact that $T_{Q,r}u$ is uniformly Lipschitz.
\end{proof}


\begin{lem}\label{vector thing}
Let $v \in \mathcal{A}(\Lambda, \alpha, M_0)$ with $||\ln(h) ||_{\alpha} \le \Gamma$ and $0 \le k \le n-2$.  Let $Q \in B_{\frac{1}{16}}(0) \cap \partial \Omega^{\pm},$ $0< r \le \frac{1}{32}$.  Let $0<\epsilon$ be fixed. There exists a constant, $\delta = \delta_0(n, \Lambda, \alpha, M_0, \Gamma \epsilon)>0$ and a constant, $0< C(n, \Lambda, \alpha, M_0, \Gamma, \epsilon)$, such that if $v$ is $(0, \delta, 8r, Q)$-symmetric, but not $(k+1, \epsilon, 8r, Q)$-symmetric, then for any orthonormal vectors, $\vec v_1, ... , \vec v_{k+1}$
\begin{align*}
\frac{1}{C} \le \frac{1}{r^{n+2}} \int_{A_{3r, 4r}(Q)} \sum_{i = 1}^{k+1}(\vec v_i \cdot \nabla v(z))^2dz \left(\int_{\partial B_1(0)}v(rx + Q)^2 d\sigma \right)^{-1}.
\end{align*}
\end{lem} 
\begin{proof}

We argue by contradiction.  Assume that there is a sequence of functions in $ v_i \in \mathcal{A}(\Lambda, \alpha, M_0)$, $Q_i \in B_{\frac{1}{16}}(0) \cap \partial \Omega_i^{\pm},$ and $0< r_i \le \frac{1}{32}$ such that $ v_i$ is $(0, 2^{-j}, 8r_i, Q_i)$-symmetric, but not $(k+1, \epsilon, 8r_i, Q_i)$-symmetric.  And, for each $i$, there exists an orthonormal collection of vectors, $\{\vec v_{ij} \}$, such that, rescaling
$$
\int_{A_{3, 4}(0)} \sum_{j = 1}^{k+1}(\vec v_{ij} \cdot \nabla T_{Q, r}  v_{i}(z))^2dz \le 2^{-i}
$$
By Lemma \ref{compactness-ish I}, we may extract a subsequence $T_{Q_j, r_j}v_j$ for which $ T_{Q_j, r_j} v_j$ converge to a non-degenerate harmonic function, $v_{\infty}$.  Similarly, $\{\vec v_{ij} \}$ converges to an orthonormal collection $\{\vec v_i\}$.  Given the assumptions above, $v_{\infty}$ is also $0$-symmetric in $B_8(0)$ and $\nabla v_{\infty} \cdot \vec v_i =0$ for all $i = 1, ..., k+1.$  Thus, $v_{\infty}$ is $(k+1, 0)$-symmetric in $B_8(0)$.  But, this is a contradiction, since $T_{Q_j, r_j}v_j$ were supposed to stay away from $(k+1)$-symmetric functions in $L^2(B_1(0))$.
\end{proof}

\subsection{The proof of Lemma \ref{beta bound lem}}
\begin{proof}
By Lemma \ref{vector thing} and properties of the $\beta$-numbers, we have for $\{\vec v_i \}$ the orthonormal basis and $\lambda_i$ the associated eigenvalues of the quadratic form in Lemma \ref{beta bound lem, part 1} 
\begin{align*}
\beta_{\mu, 2}^k(Q, r)^2 & \le \frac{\mu(B_r(Q))}{r^{k}}n\lambda_{k+1}\\
& \le \frac{\mu(B_r(Q))}{r^{k}}n C \sum_{i=1}^{k+1} \frac{\lambda_{k+1}}{r^{n+2}} \int_{A_{3r, 4r}(Q) }(\vec v_i \cdot \nabla v(z))^2dz  \left(\int_{\partial B_1(0)}v(rx + Q)^2 d\sigma \right)^{-1}\\
& \le \frac{\mu(B_r(Q))}{r^{k}}n C \sum_{i=1}^{k+1} \frac{\lambda_i}{r^{n+2}} \int_{A_{3r, 4r}(Q) }(\vec v_i \cdot \nabla v(z))^2dz \left(\int_{\partial B_1(0)}v(rx + Q)^2 d\sigma \right)^{-1}.
\end{align*}
Be Lemma \ref{beta bound lem, part 1} and Lemma \ref{fudge III}, we have 
\begin{align*}
    & \frac{\lambda_i}{r^{n+2}} \int_{A_{3r, 4r}(Q) }(\vec v_i \cdot \nabla v(z))^2dz \left(\int_{\partial B_1(0)}v(rx + Q)^2 d\sigma \right)^{-1} \lesssim_{n, \Lambda, \alpha, M_0, \Gamma}\\
    & \fint_{B_r(Q)}\left(\int_{A_{2r, 7r}(y) } \frac{|N (y, |z-y|,  v) v(z) - \nabla  v(z) \cdot (z - y)|^2}{|z-y|^{n+2}}dz \right) d\mu(y)\left(\int_{\partial B_1(0)}v(rx + Q)^2 d\sigma \right)^{-1}\\
    & \qquad +  \fint_{B_r(Q)}\left(W_{r, 16r}(Q)^2 + W_{r, 16r}(y)^2 + W_{\frac{1}{2}r, 16r}(y) + W_{\frac{1}{2}r, 16r}(Q) + \Gamma r^{2\alpha}\right)d\mu(y)\\
    & \lesssim_{n, \Lambda, \alpha, M_0, \Gamma} \fint_{B_r(Q)}\left(\int_{A_{2r, 7r}(y) } \frac{|N (y, |z-y|,  v) v(z) - \nabla  v(z) \cdot (z - y)|^2}{|z-y|^{n+2}}dz \right) d\mu(y)\left(\int_{\partial B_1(0)}v(rx + Q)^2 d\sigma \right)^{-1}\\
    & \qquad + \left(W_{\frac{1}{2}r, 16r}(Q)^2 +W_{\frac{1}{2}r, 16r}(Q) + \Gamma r^{2\alpha} \right)\frac{\mu(B_r(Q))}{r^k}+ \fint_{B_r(Q)}\left(W_{\frac{1}{2}r, 16r}(y)^2 + W_{\frac{1}{2}r, 16r}(y))\right)d\mu(y)\\
\end{align*}
Therefore, collecting constants and using Lemma \ref{N non-degenerate almost monotonicity}, we have that for $c = (\int_{\partial B_1(0)}v(rx+Q)^2d\sigma)^{\frac{-1}{2}}$,
\begin{align*}
\beta_{\mu, 2}^k(Q, r)^2  & \lesssim_{n, \Lambda, \alpha, M_0, \Gamma} \frac{1}{r^{k}}\left( \int_{B_r(Q)} N (y, 8r, cv) - N (y, r, cv)d\mu(y) \right)\\
& \qquad + \left(W_{\frac{1}{2}r, 16r}(Q)^2 +W_{\frac{1}{2}r, 16r}(Q) + \Gamma r^{2\alpha}\right)\frac{\mu(B_r(Q))}{r^k}\\
& \qquad + \fint_{B_r(Q)}\left(W_{\frac{1}{2}r, 16r}(y)^2 + W_{\frac{1}{2}r, 16r}(y))\right)d\mu(y).
\end{align*}
Since the Almgren frequency is invariant under scalar multiplication, the lemma is proved.
\end{proof}


\section{Packing}\label{S:Packing}

The following theorem of Naber and Valtorta \cite{NaberValtorta17-1} is a powerful tool which links the sum of the $\beta_{\mu}^k (Q, r)^2$ over all points and scales to packing estimates.

 \begin{thm}\label{discrete reif}(Discrete Reifenberg, \cite{NaberValtorta17-1})
Let $\{ B_{\tau_i}(x_i)\}_i$ be a collection of disjoint balls such that for all $i=1, 2, ... $  $\tau_i \le 1$. Let $\epsilon_k>0$ be fixed.  Define a measure $$\mu := \sum_i \tau_i^k \delta_{x_i},$$
and suppose that for any $x \in B_2(0)$ and any scale $l \in \{0, 1, 2, ...\}$, if $B_{r_l}(x) \subset B_2(0)$ and $\mu(B_{r_l}(x)) \ge \epsilon_k r_l^k$ then

\begin{equation*}
\sum_{i \ge l} \int_{B_{2r_{l}(x)}} \beta_{\mu}^k (z, 16r_i)^2 d\mu(z) < r_{l}^k \delta^2.
\end{equation*}

Then, there exists a $\delta_0= \delta_0(n, \epsilon_k) >0$ such that if $\delta \le \delta_0$ $$\mu (B_1(0)) = \sum_{\substack{i \emph{ s.t.}\\ x_i \in B_1(0)}} \tau_i^k \le C(n).$$
\end{thm}

Now, we are ready to prove the crucial packing lemma.  

\begin{lem}\label{packing lem}
Fix $0< \epsilon$, and let $v \in \mathcal{A}(\Lambda, \alpha, M_0)$ satisfy $||\ln(h) ||_{\alpha}\le \eta$ and $\sup_{p \in B_1(0)} N(Q, 2, v) = E$. There is an $\eta_1(n, \Lambda, \alpha, M_0, \epsilon)>0$ such that if $\eta \le \eta_1$, then for any $r>0$ if $\{B_{2r_p}(p)\}$ is a collection of disjoint balls satisfying 
\begin{equation}
N(p, \eta r_p, v) \ge E - \eta_1 , \qquad p \in S^k_{\eta_1, r}, \qquad r \le r_p \le 1,
\end{equation} 
we have the following packing condition
\begin{equation}
\sum_{p}r_p^k \le C_2(n, \Lambda, \alpha, M_0, \epsilon).
\end{equation}
\end{lem}

\begin{proof}
Choose $\delta_0(n, \Lambda, \alpha, M_0, \epsilon)$ as in Lemma \ref{beta bound lem}, and $\gamma(n, \Lambda, \alpha, M_0, \delta_0)$ as in Lemma \ref{quant rigidity}.  Note that we may assume without loss of generality that $\eta_1 \le 1$, and so for $C_1(\alpha, M_0, 1)$ the constant in Lemma \ref{N non-degenerate almost monotonicity}, let
$$\eta_1 \le \frac{\min\{\delta_0, \gamma\}}{2C_1+1}.$$
We will employ the convention that $r_i = 2^{-i}$.

For each $i \in \mathbb{N}$, define the truncated measure 
\begin{equation*}
\mu_i = \sum_{r_p \le r_i}r_p^k\delta_p.
\end{equation*}
We will write $\beta_i(x, r)=\beta_{\mu_i, 2}^k(x, r).$  Observe that $\beta_i$ enjoy the following properties.  First, because the balls are disjoint, for all $j \ge i$
\begin{equation*}
\beta_i(x,r_j) = \begin{cases}
\beta_j(x, r_j) & x\in supp(\mu_j)\\
0 & \text{  otherwise.}
\end{cases}
\end{equation*}

Furthermore, for $r_i \le 2^{-4},$ recalling Lemma \ref{global to local} our assumption of the Almgren frequency gives that $N(16r_i, Q, v) - N(r_Q, Q, v) \le (2C_1 ||\ln(h)||_{\alpha} + \eta \le \max\{\delta_0, \gamma\}$.  By Lemma \ref{quant rigidity} and Lemma \ref{beta bound lem} and our choice of $\eta \le \eta_1$
\begin{align*} 
\beta_{\mu_i, 2}^k(Q, r_i)^2 & \lesssim_{n, \Lambda, \alpha, M_0, \Gamma, \epsilon} \fint_{B_{r_i}(Q)}\left(W_{\frac{1}{2}r_i, 16r_i}(y)^2 + W_{\frac{1}{2}r_i, 16r_i}(y))\right)d\mu_i(y)\\
& \qquad + \left(W_{\frac{1}{2}r_i, 16r_i}(Q)^2 +W_{\frac{1}{2}r_i, 16r_i}(Q) + \Gamma r_i^{2\alpha}\right)\frac{\mu_i(B_{r_i}(Q))}{r_i^k}.
\end{align*}

The claim of the lemma is that $\mu_0(B_1(0)) \le C(n, \Lambda, \alpha, M_0, \epsilon)$.  We prove the claim inductively. That is, we shall argue that there is an fixed scale, $0< R= 2^{-\ell}$,  such that for $r_i \le R$ and all $x \in B_1(0)$
\begin{align*}
\mu_i(B_{r_i}(x)) \le C_{DR}(n)r_i^k.
\end{align*}

Observe that since $r_p \ge r> 0$, for $r_i < r$, the claim is trivially satisfied because $\mu_i = 0$. Assume, then, that the inductive hypothesis holds for all $j \ge i+1$.

Let $x \in B_1(0).$ We consider $\mu_i(B_{4r_i}(x)).$  Observe that we can get a course bound
\begin{equation*}
\mu_j(B_{4r_j}(x)) \le \Gamma(n)r_j^k, \quad \forall j \ge i-2, \quad \forall x \in B_1(0),
\end{equation*}
by writing $\mu_j(B_{4r_j}(x)) = \mu_{j+2}(B_{4r_j}(x)) + \sum r^k_p$, where the sum is taken over all $p \in B_{4r_j}(x)$ with $r_{j+2} < r_p \le r_j$.  Since the balls $B_{r_p}(p)$ are disjoint, there is a dimensional constant, $c(n),$ which bounds the number of such points.  Thus, we may take $\Gamma(n) = c(n)C_{DR}$.  

Now, we calculate
\begin{align*}
& \sum_{r_j < 2r_i} \int_{B_{2r_i}(x)} \beta_i(z, r_j )^2d\mu_i(z) =  \sum_{r_j < 2r_i} \int_{B_{2r_i}(x)} \beta_j(z, r_j )^2d\mu_j(z) \\
 & \qquad  \le  C \sum_{r_j < 2r_i}  \frac{1}{r_j^k}  \int_{B_{2r_i}(x)} \left(\int_{B_{r_j}(z)}\left(W_{\frac{1}{2}r_j, 16r_j}(y)^2 + W_{\frac{1}{2}r_j, 16r_j}(y)\right)d\mu_j(y)\right)d\mu_j(z)\\
& \qquad \qquad + C \sum_{r_j < 2r_i}  \int_{B_{2r_i}(x)}\left(\left(W_{\frac{1}{2}r_j, 16r_j}(z)^2 +W_{\frac{1}{2}r_j, 16r_j}(z) + || \ln(h)||_{\alpha} r_j^{2\alpha}\right)\frac{\mu(B_{r_j}(z))}{r_j^k}\right)d\mu_j(z)\\
&\qquad  \le  C \sum_{r_j < 2r_i} \int_{B_{2r_i + r_j}(x)} \frac{\mu_j(B_{r_j}(y))}{r_j^k} (W_{\frac{1}{2}r_j, 16r_j}(y)^2 + W_{\frac{1}{2}r_j, 16r_j}(y)) d\mu_j(y) \\
 & \qquad \qquad +  C \sum_{r_j < 2r_i} \int_{B_{2r_i}(x)} \left(\left(W_{\frac{1}{2}r_j, 16r_j}(z)^2 +W_{\frac{1}{2}r_j, 16r_j}(z) + || \ln(h)||_{\alpha} r_j^{2\alpha}\right)\frac{\mu(B_{r_j}(z))}{r_j^k}\right)d\mu_j(z)\\
& \qquad \le  2C\Gamma(n) \int_{B_{4r_i}(x)}\left( \sum_{r_j < 2r_i} (W_{\frac{1}{2}r_j, 16r_j}(y)^2 + W_{\frac{1}{2}r_j, 16r_j}(y))\right) d\mu_j(y)\\
 & \qquad \qquad +  C \Gamma(n) \sum_{r_j < 2r_i} || \ln(h)||_{\alpha} r_j^{2\alpha} \mu_i(B_{4r_i}(x)).
 \end{align*}
Now, we note that by assumption, $W_{\eta r_p, 1}(p) \le \eta,$ and therefore, recalling the proof of Lemma \ref{global to local} 
\begin{align*}
    |W_{\frac{1}{2}r_j, 16r_j}(p)| \le \eta + C(\alpha, M_0, \Gamma)||\ln(h) ||_{\alpha}(16r_j)^\alpha.  
\end{align*}
Thus, for $0< \eta$ small enough depending only upon $\alpha$ and $M_0$, $|W_{\frac{1}{2}r_j, 16r_j}(p)| \le 1.$  Therefore, $W_{\frac{1}{2}r_j, 16r_j}(p)^2 \le |W_{\frac{1}{2}r_j, 16r_j}(p)|.$ Therefore, recalling $r_i = 2^{-i}$ we see that
\begin{align*}
    \sum_{j = i-1}^{\infty} W_{\frac{1}{2}r_j, 16r_j}(p)^2 + W_{\frac{1}{2}r_j, 16r_j}(p) & \le \sum_{j = i-1}^{N} |W_{\frac{1}{2}r_j, 16r_j}(p)| + W_{\frac{1}{2}r_j, 16r_j}(p)\\
    & \le 6 \text{var}_{r \in [r_p, r_{i-1}]}\left(N(r, p, v)\right)\\
    & \le 12C(\alpha, M_0)||\ln(h)||_{\alpha}(r_{i-1}^{\alpha} - r_p^{\alpha}) + 6\eta.
\end{align*}
Therefore
\begin{align*}
& \sum_{r_j < 2r_i} \int_{B_{2r_i}(x)} \beta_i(z, r_j )^2d\mu_i(z) \\
& \qquad \le C \Gamma(n) \mu_i(B_{4r_i}(x)) \left(6\eta +  12C_1r_{i-1}^{\alpha} \eta\right)r_i^k +  C \Gamma(n)^2 \eta \left(\sum_{r_j < 2r_i} r_j^{2\alpha}\right) r_i^k \\
& \qquad \le C\Gamma(n)^2(1 + C(\alpha)) \eta r_i^k.
\end{align*}
Thus, for $\eta \le \eta_1(n, \Lambda, \alpha, M_0, \epsilon)$ sufficiently small
\begin{align*}
C \Gamma(n)^2 (1 + C(\alpha)) \eta & \le \frac{1}{2} \delta_{DR}.
\end{align*}
Thus, $\mu_i$ satisfies the hypotheses of Theorem \ref{discrete reif}
\begin{equation*}
 \sum_{r_j < 2r_i} \int_{B_{2r_i}(x)} \beta_i(z, r_j )^2d\mu_i(z) \le \delta_{DR}r_i^k.
 \end{equation*}
The Discreet Reifenberg Theorem therefore implies that $\mu_i(B_{r_i}(x)) \le C_{DR}r_i^k$.  

Thus, by induction, the claim holds for $r_i \le 2^{-4}$. We may use a packing argument to obtain estimates at larger scales.  That is, $\mu_0(B_1(0)) \le C_{DR}C(n, \ell)$.
\end{proof}

\section{Tree Construction}\label{S:trees}

In this section, we detail two procedures for inductively-refined covering schemes.  We will use these covering schemes in the next section to generate the actual cover which proves Theorem \ref{main theorem}. First, we fix our constants.

\subsection{Fixing Constants and a Definition.}

In this section, we fix our constants as follows.  Fix $0 < \epsilon$, and let $v \in \mathcal{A}(\Lambda, \alpha, M_0)$. Let $E = \sup_{p \in B_1(0)}N(Q, 2, v)$ and fix the scale of the covering we wish to construct as $R \in (0, 1].$  

We will let $\rho$ denote the inductive scale at which we will refine our cover.  For convenience, we will use the convention $r_i = \rho^{-i}.$  Let $\rho < \frac{1}{10}$ be sufficiently small so that 
$$
2C_2(n, \Lambda, \alpha, M_0, \epsilon)c_2(n)\rho < 1/2.
$$
where $C_2(n, \Lambda, \alpha, M_0, \epsilon)$ is as in Lemma \ref{packing lem} and $c_2(n)$ is a dimensional constant which will be given in the following lemmata.

Let $\delta(n, \Lambda, \alpha, M_0, \epsilon)$ be as in Lemma \ref{beta bound lem} and $\gamma(n, \Lambda, \alpha, M_0, \delta)$ as in Lemma \ref{quant rigidity}.  Now, we also let $\eta_1(n, \Lambda, \alpha, M_0, \epsilon)$ be as in Lemma \ref{packing lem}, and let
$$
\gamma_0 = \eta' = \eta_1/20.
$$

Note that while $\gamma_0 \le \gamma,$ Lemma \ref{quant rigidity} still holds with $\gamma_0$ in place of $\gamma.$  We then let $\eta = \eta_0(n, \Lambda, \alpha, E +1, \epsilon, \eta', \gamma_0, \rho)$ as in Corollary \ref{key dichotomy}.  We shall assume that $v$ satisfies $$||\ln(h) ||_{\alpha} \le \frac{1}{2C_1+1}\eta.$$
 
The sorting principle for our covering comes from Corollary \ref{key dichotomy}.    To formalize this, we make the following definition.

\begin{definition}
For $p \in B_2(0)$ and $0< R< r< 2$, the ball $B_r(p)$ will be called ``good" if 
$$
N(Q, \gamma \rho r, v) \ge E - \eta' \qquad \text{on} \quad S^k_{\epsilon, \eta R}(v) \cap B_r(p).
$$
We will say that $B_r(p)$ is ``bad" if it is not good.
\end{definition}
\vspace{1cm}

\begin{rmk}
By Corollary \ref{key dichotomy}, with $E+ \eta_0/2$ in place of $E$, which is admissible by monotonicity and our choice of $||ln(h) ||_{\alpha} \le \frac{1}{2C_1+1}\eta$, in any bad ball $B_r(p)$ there exists a $(k-1)$-dimensional affine plane $L^{k-1}$ such that
$$
\{N(Q, \gamma \rho r, v) \ge E - \eta_0/2 \} \cap B_r(p) \subset B_{\rho r}(L^{k-1}).
$$ 
\end{rmk}

\subsection{Good trees}
Let $x \in B_1(0)$ and $B_{r_A}(x)$ be a good ball for $A \ge 0$.  We will detail the inductive construction of a good tree based at $B_{r_A}(x)$.  The induction will build a successively refined covering $B_{r_A}(x) \cap S^k_{\epsilon, \eta R}(v)$.  We will terminate the process and have a cover which consists of a collection of bad balls with packing estimates and a collection of stop balls whose radii are comparable to $R$.  We shall use the notation $\mathcal{G}_i$ to denote the collection of centers of good balls of scale $r_i$, $\mathcal{B}_i$ shall denote the collection of centers of bad balls of scale $r_i$.

Because $B_{r_A}(x)$ is a good ball, at scale $i = A$, we set $\mathcal{G}_A = x$.  We let $\mathcal{B}_A = \emptyset$.

Now the inductive step.  Suppose that we have constructed our collections of good and bad balls down to scale $j-1 \ge A$.  Let $\{z \}_{J_i}$ be a maximal $\frac{2}{5}r_j$-net in 
$$
B_{r_A}(x) \cap S^k_{\epsilon, \eta R}(v) \cap B_{r_{j-1}}(\mathcal{G}_{j-1})\setminus \cup_{i=A}^{j-1}B_{r_i}(\mathcal{B}_i).
$$
We then sort these points into $\mathcal{G}_j$ and $\mathcal{B}_j$ depending on whether $B_{r_j}(z)$ is a good ball or a bad ball.  If $r_j > R$, we proceed inductively.  If $r_j \le R$, then we stop the procedure.  In this case, we let $\mathcal{S} = \mathcal{G}_j \cup \mathcal{B}_j$ and we call this the collection of ``stop" balls.  

The covering at which we arrive at the end of this process shall be called the ``good tree at $B_{r_A}(x)$."  We shall follow \cite{EdelenEngelstein17} and denote this $\mathcal{T}_{\mathcal{G}} = \mathcal{T}_{\mathcal{G}}(B_{r_A}(x))$.
We shall call the collection of ``bad" ball centers, $\cup_{i}\mathcal{B}_i$, the ``leaves of the tree" and denote this collection by $\mathcal{F}(\mathcal{T}_{\mathcal{G}}).$  We shall denote the collection of ``stop" ball centers by $\mathcal{S}(\mathcal{T}_{\mathcal{G}}) = \mathcal{S}$.

For $b \in \mathcal{F}(\mathcal{T}_{\mathcal{G}})$ we let $r_b = r_i$ for $i$ such that $b \in \mathcal{B}_i$.  Similarly, is $s \in \mathcal{S}(\mathcal{T}_{\mathcal{G}})$, we let $r_s = r_j$ for the terminal $j$.

\begin{thm}  \label{good trees}
A good tree, $\mathcal{T}_{\mathcal{G}}(B_{r_A}(x))$, enjoys the following properties:
\begin{itemize}
\item[(A)] Tree-leaf packing:
$$
\sum_{b \in \mathcal{F}(\mathcal{T}_{\mathcal{G}})} r_b^k \le C_2(n, \Lambda, \alpha, M_0, \epsilon) r^k_A.
$$ 
\item[(B)] Stop ball packing:
$$
\sum_{s \in \mathcal{S}(\mathcal{T}_{\mathcal{G}})} r_s^k \le C_2(n, \Lambda, \alpha, M_0, \epsilon) r^k_A.
$$ 
\item[(C)] Covering control:
$$
\mathcal{S}^k_{\epsilon, \eta R}(v) \cap B_{r_A}(x) \subset \bigcup_{s \in \mathcal{S}(\mathcal{T}_{\mathcal{G}})} B_{r_s}(s) \cup \bigcup_{b \in \mathcal{F}(\mathcal{T}_{\mathcal{G}})} B_{r_b}(b).
$$
\item[(D)] Size control: for any $s \in \mathcal{S}(\mathcal{T}_{\mathcal{G}})$, $\rho R \le r_s \le R$.
\end{itemize}
\end{thm}

\begin{proof}
First, observe that by construction
$$
\{B_{\frac{r_b}{5}}(b) : b \in \mathcal{F}(\mathcal{T}_{\mathcal{G}}) \} \cup \{B_{\frac{r_s}{5}}(s) : s \in \mathcal{S}(\mathcal{T}_{\mathcal{G}}) \} 
$$
is pairwise disjoint and centered in the set $\mathcal{S}^k_{\epsilon, \eta R}(v).$
Next, all bad balls and stop balls are centered in a good ball of the previous scale.  By our definition of good balls, then, we have for all $i$
$$
N(b, \gamma r_i, v) = N(b, \gamma \rho r_{i-1}, v) \ge E - \eta' \quad \forall b \in \mathcal{B}_i
$$
and 
$$
N(s, \gamma r_s, v) \ge E - \eta' \quad \forall s \in \mathcal{S}(\mathcal{T}_{\mathcal{G}}) .
$$
Since by monotonicity we have that $\sup_{p \in B_{r_A}(x)} N(Q, 2r_A, v) \le E + \eta'$, we can apply Lemma \ref{packing lem} to $B_{r_A}(x)$ and get the packing estimates, (A), (B).

Covering control follows from our choice of a maximal $\frac{2}{5}r_i$-net at each scale $i$.  If $i$ is the first scale at which a point, $x \in \mathcal{S}^k_{\epsilon, \eta R}(v)$, was not contained in our inductively refined cover, it would violate the maximality assumption.

The last condition, (D), follows because we stop only if $j$ is the first scale for which $r_j \le R$.  Since we decrease by a factor of $\rho$ at each scale, (D) follows.
\end{proof}

\subsection{Bad trees}
Let $B_{r_A}(x)$ be a bad ball.  Note that for every bad ball, there is a $(k-1)$-dimensional affine plane, $L^{k-1}$, associated to it which satisfies the properties elaborated in Corollary \ref{key dichotomy}.  Our construction of bad trees will differ in several respects from our construction of good trees.  The idea is still to define an inductively-refined cover at decreasing scales of $B_{r_A}(x) \cap \mathcal{S}^k_{\epsilon, \eta R}(v)$.  We shall again sort balls at each step into ``good," ``bad," and ``stop" balls. But these balls will play slightly different roles and the ``stop" balls will have different radii.

We shall reuse the notation $\mathcal{G}_i$ to denote the collection of centers of good balls of scale $r_i$, $\mathcal{B}_i$ to denote the collection of centers of bad balls of scale $r_i$, and $\mathcal{S}_i$ to denote the collection of centers of stop balls of scale $r_i$.

At scale $i = A$, we set $\mathcal{B}_A = x$, since $B_{r_A}(x)$ is a bad ball, and set $\mathcal{S}_A= \mathcal{G}_A = \emptyset$.  Suppose, now that we have constructed good, bad, and stop balls for scale $i-1 \ge A$.  If $r_i > R$, then define $\mathcal{S}_i$ to be a maximal $\frac{2}{5}\eta r_{i-1}$-net in 
$$
B_{r_A}(x) \cap \mathcal{S}^k_{\epsilon, \eta R}(v) \cap \cup_{b \in \mathcal{B}_{i-1}} B_{r_{i-1}}(b) \setminus B_{2\rho r_{i-1}}(L^{k-1}_b).
$$
Note that $\eta << \rho$, so $\eta r_{i-1} < r_i$.
We then let $\{z\}$ be a maximal $\frac{2}{5}r_{i}$-net in 
$$
B_{r_A}(x) \cap \mathcal{S}^k_{\epsilon, \eta R}(v) \cap \cup_{b \in \mathcal{B}_{i-1}} B_{r_{i-1}}(b) \cap B_{2\rho r_{i-1}}(L^{k-1}_b).
$$
We then sort $\{z\}$ into the disjoint union $\mathcal{G}_i \cup \mathcal{B}_i$ depending on whether $B_{r_i}(z)$ is a good ball or a bad ball.

If $r_i \le R$, then we terminate the process by defining $\mathcal{G}_i = \mathcal{B}_i = \emptyset$ and letting $\mathcal{S}_i$ be a maximal $\frac{2}{5}\eta r_{i-1}$-net in 
$$
B_{r_A}(x) \cap \mathcal{S}^k_{\epsilon, \eta R}(v) \cap B_{r_i}(\mathcal{B}_{i-1}).
$$

The covering at which we arrive at the end of this process shall be called the ``bad tree at $B_{r_A}(x)$."  We shall follow \cite{EdelenEngelstein17} and denote this $\mathcal{T}_{\mathcal{B}} = \mathcal{T}_{\mathcal{B}}(B_{r_A}(x))$.
We shall call the collection of ``good" ball centers, $\cup_{i}\mathcal{G}_i$, the ``leaves of the tree" and denote this collection by $\mathcal{F}(\mathcal{T}_{\mathcal{B}}).$  We shall denote the collection of ``stop" ball centers by $\mathcal{S}(\mathcal{T}_{\mathcal{B}}) = \cup_i \mathcal{S}_i$.

As before, we shall use the convention that for $g \in \mathcal{F}(\mathcal{T}_{\mathcal{B}})$ we let $r_g = r_i$ for $i$ such that $g \in \mathcal{G}_i$.  However, note that now, if $s \in  \mathcal{S}_i \subset \mathcal{S}(\mathcal{T}_{\mathcal{B}})$, we let $r_s =\eta r_{i-1}$.

\begin{thm} \label{bad trees}
A bad tree, $\mathcal{T}_{\mathcal{B}}(B_{r_A}(x))$, enjoys the following properties:
\begin{itemize}
\item[(A)] Tree-leaf packing:
$$
\sum_{g \in \mathcal{F}(\mathcal{T}_{\mathcal{B}})} r_g^k \le 2c_2(n)\rho r^k_A.
$$ 
\item[(B)] Stop ball packing:
$$
\sum_{s \in \mathcal{S}(\mathcal{T}_{\mathcal{B}})} r_s^k \le c(n, \eta) r^k_A.
$$ 
\item[(C)] Covering control:
$$
\mathcal{S}^k_{\epsilon, \eta R}(v) \cap B_{r_A}(x) \subset \bigcup_{s \in \mathcal{S}(\mathcal{T}_{\mathcal{B}})} B_{r_s}(s) \cup \bigcup_{g \in \mathcal{F}(\mathcal{T}_{\mathcal{B}})} B_{r_g}(g).
$$
\item[(D)] Size control: for any $s \in \mathcal{S}(\mathcal{T}_{\mathcal{B}})$, at least one of the following holds:
$$\eta R \le r_s \le R \quad \text{  or  }\quad \sup_{p \in B_{2r_s}(s)} N(Q, 2r_s, v) \le E- \eta/2.$$
\end{itemize}
\end{thm}

\begin{proof}
Conclusion (C) follows identically as in Theorem \ref{good trees}.  Next, we consider the packing etimates. Let $r_i > R$.  Then, by construction, for any $b \in \mathcal{B}_{i-1}$, we have that 
$$
\mathcal{G}_i \cup \mathcal{B}_i \cup B_{r_{i-1}}(b) \subset B_{2\rho r_{i-1}}(L_b^{k-1}).
$$

Thus, since the points $\mathcal{G}_i \cup \mathcal{B}_i$ are $\frac{2}{5}r_i$ disjoint, we calculate
$$
|\mathcal{G}_i \cup \mathcal{B}_i \cup B_{r_{i-1}}(b)| \le \omega_{k-1}\omega_{n-k+1}(3\rho)^{n-k+1} \frac{1}{\omega_n (\rho/5)^n} \le c_2(n) \rho^{1-k}.
$$
We can push this estimate up the scales as follows
\begin{eqnarray*}
|\mathcal{G}_i \cup \mathcal{B}_i \cup B_{r_{i-1}}(b)| r_i^k &\le &c_2(n) \rho^{1}|\mathcal{B}_{i-1}| r_{i-1}^k\\
& \le & c_2(n) \rho^{1}|\mathcal{B}_{i-1} \cup \mathcal{G}_{i-1}| r_{i-1}^k\\
& \vdots & \\
&\le & (c_2 \rho)^{i-A}r_A^k.
\end{eqnarray*}
Summing over all $i \ge A$, then, we have that 
\begin{equation}\nonumber
\sum_{i = A+1}^{\infty} |\mathcal{B}_{i-1} \cup \mathcal{G}_{i-1}|r_i^k \le \sum_{i = A+1}^{\infty} (c_2 \rho)^{i-A}r_A^k.
\end{equation}

Since we chose $c_2\rho \le 1/2$, we have that the sum converges and $\sum_{i = A+1}^{\infty} |\mathcal{B}_{i-1} \cup \mathcal{G}_{i-1}|r_i^k \le  2c_2 \rho r_A^k $. This proves (A).

To see (B), we observe that for any given scale $i \ge A+1$, the collection of stop balls, $\{B_{\eta r_{i-1}}(s) \}_{s \in \mathcal{S}_i}$, form a Vitali collection centered in $B_{r_{i-1}}(\mathcal{B}_{i-1})$.  Thus, we have that 
$$
|\{\mathcal{S}_i \}| \le \frac{10^n}{\eta^n} |\{\mathcal{B}_{i-1}\}|.
$$

Since by construction there are no stop balls at the initial scale, $A$, we compute that 
$$
\sum_{i=A+1}^{\infty}|\{\mathcal{S}_i \}|(\eta r_{i-1})^k \le 10^k \eta^{k-n} \sum_{i=A}^{\infty} |\{\mathcal{B}_i \}|r^k_i \le c(n, \eta)r^k_A.
$$
This is (B).

We now argue (D).  For $s \in \mathcal{S}_i$ where $r_i > R$, by construction $s \in B_{r_{i-1}}(b) \setminus B_{2\rho r_{i-1}}(L^{k-1})
$ for some $b \in \mathcal{B}_{i-1}$.  By Corollary \ref{key dichotomy}, the construction, and our choice of $\eta \le \frac{\rho}{2}$, we have that 
\begin{equation}\nonumber
\sup_{p \in B_{2r_s}(s)}N(Q, 2r_s, v) 
\le \sup_{p \in B_{2\eta r_{i-1}}(s)}N(Q, 2\eta r_{i-1}, v) \le E - \eta/2. 
\end{equation}
 On the other hand, if $r_i \le R,$ then $r_{i-1} > R$.  Thus 
\begin{equation}\nonumber
R \ge \rho r_{i-1} \ge \eta r_{i-1} = r_s \ge \eta R.
\end{equation}
This proves (D).
\end{proof}

\section{The Covering}\label{S:covering}

Assuming that $||\ln(h) ||_{\alpha} \le \frac{1}{2C_1+1}\eta$, for $0< \eta \le \eta_0(n, \Lambda, \alpha, E +1, \epsilon, \eta', \gamma_0, \rho)$ as in Section \ref{S:trees}, we now wish to build the covering of $\mathcal{S}^k_{\epsilon, \eta R} \cap B_1(0)$ using the tree constructions, above.  Note that $B_1(0)$ is either a good ball or a bad ball.  Therefore, we can construct a tree with $B_1(0)$ as the root.  Then in each of the leaves, we construct either good trees or bad trees, depending upon the type of the leaves.  Since in each construction, we decrease the size of the leaves by a factor of $\rho<1/10$, we can continue alternating tree types until the process terminates in finite time.  

Explicitly, we let $\mathcal{F}_0 = \{0\}.$ and let $B_1(0)$ be the only leaf.  We set $\mathcal{S}_0 = \emptyset$.  Now, assume that we have defined the leaves and stop balls up to stage $i-1$.  
Since by hypothesis, the leaves in $\mathcal{F}_i$ are all good balls or bad balls, if they are good, we define for each $f \in \mathcal{F}_{i-1}$ the good tree $\mathcal{T}_{\mathcal{G}}(B_{r_f}(f)).$ We then set
\begin{equation*}
\mathcal{F}_i = \bigcup_{f \in \mathcal{F}_{i-1}} \mathcal{F}(\mathcal{T}_{\mathcal{G}}(B_{r_f}(f)))
\end{equation*}
and 
\begin{equation*}
\mathcal{S}_i = \mathcal{S}_{i-1} \cup \bigcup_{f \in \mathcal{F}_{i-1}} \mathcal{S}(\mathcal{T}_{\mathcal{G}}(B_{r_f}(f))).
\end{equation*}
Since all the leaves of good trees are bad balls, all the leaves of $\mathcal{F}_i$ are bad.

If, on the other hand, leaves of $\mathcal{F}_{i-1}$ are bad, then for each $f \in \mathcal{F}_{i-1},$ we construct a bad tree, $\mathcal{T}_{\mathcal{B}}(B_{r_f}(f))$.  In this case, we set 
\begin{equation*}
\mathcal{F}_i = \bigcup_{f \in \mathcal{F}_{i-1}} \mathcal{F}(\mathcal{T}_{\mathcal{B}}(B_{r_f}(f)))
\end{equation*}
and 
\begin{equation*}
\mathcal{S}_i = \mathcal{S}_{i-1} \cup \bigcup_{f \in \mathcal{F}_{i-1}} \mathcal{S}(\mathcal{T}_{\mathcal{B}}(B_{r_f}(f))).
\end{equation*}
Since all the leaves of bad trees are good balls, all the leaves of $\mathcal{F}_i$ are good.

This construction gives the following estimates.

\begin{lem}\label{tree properties}
For the construction described above, there is an $N \in \mathbb{N}$ such that $\mathcal{F}_N = \emptyset$ with the following properties:
\begin{itemize}
\item[(A)] Leaf packing:
$$
\sum_{i = 0}^{N-1} \sum_{f \in \mathcal{F}_i} r_f^k \le c(n).
$$ 
\item[(B)] Stop ball packing:
$$
\sum_{s \in \mathcal{S}_N} r_s^k \le c(n, \Lambda, \alpha, M_0, \epsilon).
$$ 
\item[(C)] Covering control:
$$
\mathcal{S}^k_{\epsilon, \eta R}(v) \cap B_1(0) \subset \bigcup_{s \in \mathcal{S}_N} B_{r_s}(s).
$$
\item[(D)] Size control: for any $s \in \mathcal{S}_N$, at least one of the following holds:
$$\eta R \le r_s \le R \quad \text{  or  }\quad \sup_{p \in B_{2r_s}(s)} N(Q, 2r_s, v) \le E- \eta/2.$$
\end{itemize}
\end{lem}

\begin{proof}
By construction, each of the leaves of a good or bad tree satisfy $r_f \le r_i$.  Thus, there is an $i$ sufficiently large so that $r_i <R$.  Thus, $N$ is finite.

To see (A), we use the previous theorems.  That is, if the leaves, $\mathcal{F}_i$, are good, then they are the leaves of bad trees rooted in $\mathcal{F}_{i-1}$.  Thus, we calculate by Theorem \ref{bad trees}
\begin{equation*}
\sum_{f \in \mathcal{F}_i} r_f^k \le 2c_2(n)\rho \sum_{f' \in \mathcal{F}_{i-1}} r_{f'}^k.
\end{equation*}

On the other hand, if the leaves, $\mathcal{F}_i$, are bad, then they are the leaves of good trees rooted in $\mathcal{F}_{i-1}$.  Thus, we calculate by Theorem \ref{good trees}
\begin{equation*}
\sum_{f \in \mathcal{F}_i} r_f^k \le C_2(n, \Lambda, \alpha, M_0, \epsilon) \sum_{f' \in \mathcal{F}_{i-1}} r_{f'}^k.
\end{equation*} 

Concatenating the estimates, since we alternate between good and bad leaves, we have
\begin{equation*}
\sum_{f \in \mathcal{F}_i} r_f^k \le c(n)(2C_2(n, \Lambda, \alpha, M_0, \epsilon)c_2(n)\rho)^{i/2}.
\end{equation*} 
By our choice of $\rho$, then, $\sum_{f \in \mathcal{F}_i} r_f^k \le c(n)2^{-i/2}$.  The estimate (A) follows immediately.  

We now turn our attention to (B).  Each stop ball, $s \in \mathcal{S}_N$, is a stop ball coming from a good or a bad tree rooted in one of the leaves of a bad tree or good tree.  We have the estimates from Theorems \ref{good trees} and \ref{bad trees}, which give bounds packing both leaves and stop balls.  Combining these, we get 
\begin{align*}
\sum_{s\in \mathcal{S}_N} r_s^k & = \sum_{i = 0}^N \sum_{s \in \mathcal{S}_i} r_s^k\\
& \le \sum_{i = 0}^{N-1} \sum_{f \in \mathcal{F}_i}c(n, \eta) r_f^k\\
& \le C(n, \eta).
\end{align*} 
Recalling the dependencies of $\eta$ gives the desired result.

(C) follows inductively from the analogous covering control in Theorems \ref{good trees} and Theorem \ref{bad trees} applied to each tree constructed.  (D) is immediate from these theorems, as well.  
\end{proof}

\begin{cor} \label{covering cor} 
Fix $0<\epsilon.$  Let $v \in \mathcal{A}(\Lambda, \alpha, M_0)$ satisfying $\sup_{p \in B_2(0)}N(Q, 2, v)\le E$.  Fix $0 < \epsilon.$ There is an $\eta_0(n, \Lambda, \alpha, M_0, \epsilon, E)>0$ such that if $0 < \eta \le \eta_0$ and $||\ln(h) ||_{\alpha} \le \frac{\eta}{2C_1+1}$ then given any $0< R \le 1$ there is a collection of balls, $\{B_{r_x}(x) \}_{x \in \mathcal{U}}$ with centers $x \in S^k_{\epsilon, \eta R}(v) \cap B_1(0)$.  Further, $R \le r_x \le \frac{1}{10}$ and the collection has the following properties:
\begin{itemize}
\item[(A)] Packing:
$$
\sum_{x \in \mathcal{U}} r_x^k \le c(n, \Lambda, \alpha, M_0, E, \epsilon).
$$ 
\item[(B)] Covering control:
$$
\mathcal{S}^k_{\epsilon, \eta R}(v) \cap B_1(0) \subset \bigcup_{x \in \mathcal{U}} B_{r_x}(x).
$$
\item[(C)] Energy drop:  For every $x \in \mathcal{U}$, either 
$$r_x = R  \qquad \text{  or  } \quad \sup_{p \in B_{2r_s}(s)} N(Q, 2r_s, v) \le E- \eta_0/2.$$
\end{itemize}
\end{cor}

This follows immediately from the Lemma \ref{tree properties} with $\eta \le \eta_1$, $\mathcal{S}_N = \mathcal{U}$, and setting $r_x = max\{R, r_s \}$.

\subsection{Proof of Theorem \ref{main theorem}}

\begin{lem}\label{kappa scale lem}
Let $v \in \mathcal{A}(\Lambda, \alpha, M_0)$ with $||\ln(h) ||_{\alpha} \le \Gamma$.  There exists a scale $\kappa(n, \Lambda, \alpha, M_0, \Gamma, \epsilon)>0$ such for all balls, $B_r(y),$ with $0< r< \kappa$ and $y\in B_{1/4}(0)$, the function $\tilde v(x) = v(rx+y)$ on $B_1(0)$ satisfies the following properties.
\begin{align*}
\sup_{p \in B_{1}(0)} N(Q, 2, \tilde v) & \le C(n, \Lambda, \alpha, M_0, \Gamma)\\
||\ln (\tilde h)||_{C^{0,\alpha}(B_1(0))} & \le \frac{\eta_0}{2C_1 + 1},
\end{align*}
where $\eta_0 = \eta_0(n, \Lambda, \alpha, C(n, \Lambda, \alpha, M_0, \Gamma) +1, \eta', \epsilon, \gamma_0, \rho) = \eta_0(n, \Lambda, \alpha, M_0, \Gamma, \epsilon)$ as in Corollary \ref{key dichotomy} and $C(n, \Lambda, \alpha, M_0, \Gamma)$ is as in Lemma \ref{N bound lem}.
\end{lem}

\begin{proof}  We know by Lemma \ref{N bound lem} that for any ball, $B_r(y)$, with $0< r< \frac{1}{2}$ and $y\in B_{1/4}(0)$ 
$$
\sup_{p \in B_{r}(y)} N(Q, 2r, v) \le C(n, \Lambda, \alpha, M_0, \Gamma).
$$

Next we consider the effect of rescaling on H\"older continuous functions.
\begin{eqnarray*}
|\tilde v(x)-\tilde v(z)| &= & |v(rx + y)-v(rz + y)|\\
& \le & \Gamma |rx-ry|^{\alpha}\\
& = & \Gamma r^{\alpha}|x-y|^{\alpha}.
\end{eqnarray*}
Since $r^{\alpha} \rightarrow 0$ as $r \rightarrow 0$, there exists an $\kappa(n, \Lambda, \alpha, M_0, \Gamma, \epsilon)>0$ such that $\Gamma\kappa^{\alpha}< \frac{\eta_0}{2C_1 + 1}$. 
\end{proof}

\begin{thm}\label{U_i theorem}
Let $v \in \mathcal{A}(\Lambda, \alpha, M_0)$ with $||\ln(h) ||_{\alpha} \le \Gamma$.  For all $\epsilon>0$ there exists an $\eta_0(n, \Lambda, \alpha, M_0, \Gamma, \epsilon)>0$ such that for all $0< R< 1$ and $k= 1, 2, ..., n-1$ we can find a collection of balls, $\{B_R(x_i)\}_i$ with the following properties:
\begin{enumerate}
\item $\mathcal{S}^k_{\epsilon, \eta_0 R}(v) \cap B_{1/4}(0) \subset \cup_i B_R(x_i).$
\item$|\{x_i\}_i| \le c(n, \Lambda, \alpha, M_0, \Gamma, \epsilon) R^{-k}$.
\end{enumerate}
\end{thm}

\begin{proof}  Cover $\mathcal{S}^k_{\epsilon, \eta r}(v) \cap B_{1/4}(0)$ by balls $B_{\kappa}(y_j)$, with $y_j \in B_{1/4}(0)$ such that $B_{1/4}(0) \subset \bigcup_jB_{\kappa}(y_j)$ for $0<\kappa(n, \Lambda, \alpha, M_0, \Gamma, \epsilon)$ the constant in Lemma \ref{kappa scale lem}.  Note that we need at most $c(n, \Lambda, \alpha, M_0, \Gamma, \epsilon)$ such balls.  

We now wish to apply Corollary \ref{covering cor} to the rescaled functions $\tilde v_i(x) = v(\kappa x + y_i)$ in $B_1(0).$  However, a careful reader may object that $\tilde v_i \not \in \mathcal{A}(n, \Lambda, \alpha),$  since it is possible that $\{ 0\} \not \in \{\tilde v = 0\}$.  However, by Lemma \ref{N bound lem}, we have that $N(Q, r, \tilde v) \le C$ for all $0<r<2$ and all $p \in B_{1}(0)$.  This, and the local uniformly Lipschitz bound from Lemma \ref{Lipschitz} gives us the necessary compactness properties to push through all the previous results for $\tilde v_i$ without changing the constants. The geometry of the sets we are considering, \textit{qua} geometry, is invariant under such rescaling. Furthermore, in balls, $B_{\kappa}(y_i)$, for which $\tilde v_i$ is harmonic, the lemmata of this paper simplify and the desired results are already contained in \cite{NaberValtorta17-1}. 

We now construct the desired covering in $B_1(0)$ for each $\tilde v_i.$ Ensuring that $c(n, \Lambda, \alpha, M_0, \Gamma, \epsilon)$ is sufficiently large, we may reduce to arguing for $r < \eta$.  We now use Corollary \ref{covering cor} to build a covering $\mathcal{U}_1$.  If every $r_x = R$, then the packing and covering estimates give the claim directly, since
\begin{equation*}
R^{k-n} \text{Vol}\left(B_R(\mathcal{S}^k_{\epsilon, \eta_0 R}(\tilde v_i) \cap B_1(0))\right) \le \omega_n R^{k-n} \sum_{\mathcal{U}_1}(2R)^n =  \omega_n 2^n \sum_{\mathcal{U}_1}r_x^k \le c(n, \Lambda,  \alpha, M_0, \Gamma, \epsilon)
\end{equation*}

If there exists an $r_x \not = R$, we use Corollary \ref{covering cor}, to build a finite sequence of refined covers, $\mathcal{U}_1, \mathcal{U}_2, \mathcal{U}_3 , ... $ such that for each for each $i$, the covering satisfies the following properties:

\begin{itemize}
\item[$(A_i)$] Packing:
$$
\sum_{x \in \mathcal{U}_i} r_x^k \le c(n, \Lambda, \alpha, M_0, \Gamma, \epsilon)( 1 + \sum_{x \in \mathcal{U}_{i-1}} r_x^k  ).
$$ 
\item[$(B_i)$] Covering control:
$$
\mathcal{S}^k_{\epsilon, \eta_0 R}(\tilde v_i) \cap B_1(0) \subset \bigcup_{x \in \mathcal{U}_i} B_{r_x}(x).
$$
\item[$(C_i)$] Energy drop:  For every $x \in \mathcal{U}_i$, either 
$$r_x = R  \qquad \text{  or  } \quad \sup_{p \in B_{2r_s}(s)} N(Q, 2r_s, \tilde v_i) \le C(n, \Lambda, \alpha, M_0, \Gamma)- i(\frac{\eta_0}{2}).$$
\item[$(D_i)$] Radius control:
$$
\sup_{x \in \mathcal{U}_i} r_x \le 10^{-i}.
$$
\end{itemize}
 
If we can construct such a sequence of covers, then we claim that this process will terminate in finite time, \textit{independent of} $R$.  Recall that blow-ups of $\tilde v_i$ are homogeneous harmonic polynomials.  Therefore
$$N(Q, 0, \tilde v_i) = \lim_{r \rightarrow \infty}N(Q, r, \tilde v_i) \ge 1$$
for all $Q \in \partial \Omega^{\pm}$.  By Remark \ref{N- remark} we have that for all $0<r \le 1$ 
$$N(Q, r, \tilde v_i) \ge 1-C(n, \Lambda, \alpha, M_0, \Gamma, \epsilon),$$
for all $p \in B_1(0)$. Therefore, we know that for $i$ sufficiently large such that
\begin{align*}
i > (C(n, \Lambda, \alpha, M_0, \Gamma, \epsilon) + C(n, \Lambda, \alpha, M_0, \Gamma, \epsilon) - 1)\frac{2}{\eta_0}
\end{align*}
it must be the case that $r_x = R$ for all $x \in \mathcal{U}_i$.  In this case, we will have the claim with a bound of the form
\begin{equation*}
R^{k-n} \text{Vol}(B_R(\mathcal{S}^k_{\epsilon, \eta_0 R}(\tilde v_i) \cap B_1(0))) \le c(n, \Lambda, \alpha, M_0, \Gamma, \epsilon )^{C(n, \Lambda, \alpha, M_0, \Gamma, \epsilon)}.
\end{equation*}

Thus, we reduce to inductively constructing the required covers.  Suppose we have already constructed $\mathcal{U}_{i-1}$ as desired.  For each $x \in \mathcal{U}_{i-1}$ with $r_x > R$, we apply Corollary \ref{covering cor} at scale $B_{r_x}(x)$ to obtain a new collection of balls, $\mathcal{U}_{i, x}$.  From the assumption that $r_x \le 1/10$ and the way that H\"older norms scale, it is clear that $\tilde v_i$ satisfies the hypotheses of Corollary \ref{covering cor} in $B_{r_x}(x)$ with the same constants.
To check packing control, we have that 
\begin{equation*}
\sum_{y \in \mathcal{U}_{i, x}} r_y^k \le c(n, \Lambda, \alpha, M_0, \Gamma, \epsilon) r_x^k .
\end{equation*}
Covering control follows immediately from the statement of Corollary \ref{covering cor}.  Similarly, 
from hypothesis $(C_{i-1}),$ we have that $\sup_{p \in B_{2r_x}(x)} N(Q, 2r_x, \tilde v_i) \le C(n, \Lambda, \alpha, M_0, \Gamma, \epsilon)- (i-1)\frac{\eta_0}{2}.$  Thus, the statement of Corollary \ref{covering cor} at scale $B_{r_x}(x)$ gives that $\sup_{p \in B_{2r_y}(y)} N(Q, 2r_y, \tilde v_i) \le C(n, \Lambda, \alpha, M_0, \Gamma, \epsilon)- i(\frac{\eta_0}{2})$ for all $y \in \mathcal{U}_{i, x}$ with $r_y > R$.
Radius control follows immediately from the fact that $\sup_{y \in \mathcal{U}_{i, x}} r_y \le r_x /10 \le 10^{-i}.$

Thus, if we let 
\begin{equation*}
\mathcal{U}_i =  \{x \in \mathcal{U}_{i-1} | r_x = R \} \cup \bigcup_{\substack{x \in \mathcal{U}_{i-1}\\
r_x > R}} \mathcal{U}_{i, x}
\end{equation*}
then $\mathcal{U}_i$ satisfy the inductive claim.  

To obtain the cover which proves the theorem, then, we simple scale each covering of $\mathcal{S}^k_{\epsilon, \frac{\eta_0 R}{\kappa}}(\tilde v_i) \cap B_1(0)$ to a covering of $\mathcal{S}^k_{\epsilon, \eta_0 R}(v) \cap B_{\kappa}(y_i)$ and sum over the $c(n, \Lambda, \alpha, M_0, \Gamma, \epsilon)$ such balls which cover $\mathcal{S}^k_{\epsilon, \eta_0 R}(v) \cap B_{\frac{1}{4}}(0).$  This completes the proof.
\end{proof}

\subsubsection{Proof of Theorem \ref{main theorem}}

\begin{proof}
By Theorem \ref{U_i theorem}, we have
\begin{align*}
\text{Vol}(B_R(\mathcal{S}^k_{\epsilon, \eta_0 R}(v) \cap B_{1/4}(0))) & \le C(n, \Lambda, \alpha, M_0, \Gamma, \epsilon)R^{n-k}.
\end{align*}

Thus, let $r_0 = \eta_0$ and $r = \eta_0 R' $ for $0< R' \le 1$.  For any $r \le R \le R'$, by containment, we have 
\begin{align*}
B_{R}(\mathcal{S}^k_{\epsilon, r}(v) \cap B_{1/4}(0)) \subset  B_{R'}(\mathcal{S}^k_{\epsilon, r}(v) \cap B_{1/4}(0)) \subset \bigcup_i B_{2R'}(x_i),
\end{align*}
where $\{x_i \}$ are the centers of the balls in the covering constructed in Theorem \ref{U_i theorem}.  Therefore, the estimates in Theorem \ref{U_i theorem} give
\begin{align*}
\text{Vol}(B_R(\mathcal{S}^k_{\epsilon, r}(v) \cap B_{1/4}(0))) & \le C(n, \Lambda, \alpha, M_0, \Gamma, \epsilon)2^n(R')^{n-k}\\
& \le C(n, \Lambda, \alpha, M_0, \Gamma, \epsilon)2^n\left(\frac{R}{\eta_0}\right)^{n-k}\\
& \le C(n, \Lambda, \alpha, M_0, \Gamma, \epsilon)R^{n-k}
\end{align*}
by increasing our constant $C(n, \Lambda, \alpha, M_0, \Gamma, \epsilon)$.

For any $R' \le R$, by containment, we have 
\begin{align*}
B_{R}\left(\mathcal{S}^k_{\epsilon, r}(v) \cap B_{1/4}(0)\right) \subset \bigcup_i B_{2R}(x_i)
\end{align*}
where $\{x_i \}$ are the centers of the balls in the covering constructed in Theorem \ref{U_i theorem}. In this case
\begin{align*}
\text{Vol}(B_R(\mathcal{S}^k_{\epsilon, r}(v) \cap B_{1/4}(0))) & \le C(n, \Lambda, \alpha, M_0, \Gamma, \epsilon)2^n(R)^{n-k}\\
& \le C(n, \Lambda, \alpha, M_0, \Gamma, \epsilon)R^{n-k}
\end{align*}
by increasing our constant $C(n, \Lambda, \alpha, M_0, \Gamma, \epsilon)$.  This concludes the proof of Theorem \ref{main theorem}.
\end{proof}

\section{Proof of Corollary \ref{cor 3.3}}
 
 In this section, we prove that $\emph{sing}(\partial \Omega^{\pm}) \subset \mathcal{S}^{k-3}_{\epsilon}(v)$ for $\epsilon$ small enough.

\begin{lem}\label{epsilon regularity}
Let $v \in \mathcal{A}(\Lambda, \alpha, M_0)$ with $||\ln (h) ||_{\alpha} \le \Gamma$.  Then there exists an $0< \epsilon = \epsilon(M_0, \alpha, \Gamma)$ such that $\emph{sing}(\partial \Omega^{\pm}) \cap B_{1/4}(0) \subset \mathcal{S}^{n-3}_{\epsilon}(v).$  
\end{lem}

\begin{proof}
We must argue that there is an $0< \epsilon$ such that for all $Q \in \emph{sing}(\partial \Omega^{\pm}) \cap B_1(0)$ and all radii $0< r$
\begin{align*}
\int_{B_1(0)} |T_{Q, r}v - P|^2 dV \ge \epsilon
\end{align*}
for all $(n-2)$-symmetric functions $P$.

If $P$ is $(n-2)$-symmetric, $P$ only depends upon $2$ variables.  By Complex Analysis all homogeneous harmonic polynomials in $2$ dimensions are of the form $q(z) = c (x+ i y)^k$.  By Theorem \ref{flat implies smooth}, we need only consider $k \ge 2$.  Hence, the zero-set, $\Sigma_q$, of any $\text{Re}(q)$ is the union of an even number of infinite rays equidistributed in angle.  If we label the connected components of $\RR^2 \setminus \Sigma_q$,  $\{U_i\}$, we see that by the Maximum Principle, the sign of $q$ must change from one $U_i$ to another, contiguous $U_j$.

Thus, the zero set of $P$ is $\Sigma_P = \Sigma_q \times \RR^{n-2}$ for some homogeneous harmonic polynomial, $\text{Re}(q): \RR^2 \rightarrow \RR$ of degree $\ge 2$.  We label the connected components of $\RR^n \setminus \Sigma_q \times \RR^{n-2}$ as $\{W_i\}$.

Now, we claim that there is a constant, $0< c(M_0, \Gamma, \alpha) \le 1,$ such that one of the following estimates must hold
\begin{align*}
(1) \qquad \qquad & \mathcal{H}^n\left(T_{Q, r}\Omega^{-} \cap \bigcup_i \{W_i : P>0 \text{  on  } W_i \} \cap B_1(0)\right)\ge  c \\
(2) \qquad \qquad & \mathcal{H}^n\left(T_{Q, r}\Omega^{+} \cap \bigcup_i \{W_i : P<0 \text{  on  } W_i \} \cap B_1(0)\right)\ge c.
\end{align*}

Note that by Theorem \ref{TBE combo}(2), we need only consider $P$ with degree $\le d(M_0)< \infty.$  Reducing to $\RR^2$, since the rays of $\Sigma_q$ are equidistributed, for $q$ of degree $k$ the connected components occupy a sector of  aperture $\frac{\pi}{k}.$  Thus, if $B_{\frac{1}{M_0}}(A^{\pm}_1(0)) \subset T_{Q, r}\Omega^{\pm},$ is the ball guaranteed by the corkscrew condition, then for $c = \frac{1}{4M_0^n}$, there exists an integer $k(M_0)$ such that 
\begin{align*}
\mathcal{H}^{n}(B_{\frac{1}{M_0}}(A^{\pm}_1(0)) \cap \{ P \cdot  T_{Q, r}v < 0 \})  \ge c
\end{align*}
for all $P$ with degree $\ge k(M_0)$.

For $P$ with degree $\le k(M_0)$, We argue by contradiction.  Suppose that no such constant exists.  Then, there would be a sequence of functions, $v_j \in \mathcal{A}(\Lambda, \alpha, M_0)$ with points, $Q_j \in B_{1/4}(0)$ and radii, $0< r_j \le 1/2$ and zero sets, $\Sigma_{P_j}$ for $P_j$ satisfying $2 \le \text{degree}(P_j) \le k(M_0)$ such that the scaled and translated mutual boundaries, $T_{Q_j, r_j}\partial \Omega^{\pm}_j,$ satisfy the following property
\begin{align*}
\mathcal{H}^n(T_{Q_j, r_j}\Omega_j^{-} \cap \bigcup_i \{W_{i, j} : P_j>0 \text{  on  } W_{i, j} \} \cap B_1(0)) \rightarrow 0\\
\mathcal{H}^n(T_{Q_j, r_j}\Omega_j^{+} \cap \bigcup_i \{W_{i, j} : P_j<0 \text{  on  } W_{i, j} \} \cap B_1(0)) \rightarrow 0.
\end{align*}
By Lemma \ref{compactness-ish I} there exists a subsequence for which $T_{Q_j, r_j}\partial \Omega^{\pm}_j$ converge locally in the Hausdorff metric to a limit set, $A\subset \RR^n$.  By Theorem \ref{NTA limit set}, $A$ must be the mutual boundary of a pair of two-sided NTA domains, $\Omega^{\pm}_{\infty}$ with constant $2M_0$.   Furthermore, up to scaling and rotation, the space of homogeneous harmonic functions of $2$ variables in $\RR^n$ with $2 \le \text{degree}(P) \le k(M_0)$ is finite-dimensional.  Since the space of rotations is compact, we may find a subsequence, $\Sigma_{P_j}$, which converge to $\Sigma_{P_{\infty}}$, for some $(n-2)$-symmetric $P_{\infty}$ locally in the Hausdorff metric.  This implies that

\begin{align}\label{A equals P}
\mathcal{H}^n(\Omega^{-}_{\infty} \cap \bigcup_i \{W_{i, \infty} : P_{\infty}>0 \text{  on  } W_{i, \infty} \} \cap B_1(0)) = 0\\
\mathcal{H}^n(\Omega^{+}_{\infty} \cap \bigcup_i \{W_{i, \infty} : P_{\infty}<0 \text{  on  } W_{i, \infty} \} \cap B_1(0)) = 0.
\end{align}
Indeed, if there were $p \in \bigcup_i \{W_{i, \infty} : P_{\infty}>0 \text{  on  } W_{i, \infty} \} \cap B_1(0)$ such that $p \in \Omega^{-}_{\infty}$, since $W_{i, \infty}$ and $\Omega^{-}$ are open, there would  exist a ball $B_{\delta}(p) \subset \Omega^{-} \cap W_{i, \infty}$.  Therefore, since $\Sigma_{P_j} \rightarrow \Sigma_{P_{\infty}}$ and $T_{Q_j, r_j}\partial \Omega^{\pm}_j \rightarrow A$ locally in the Hausdorff metric, for all $j$ sufficiently large, $B_{\frac{1}{2}\delta}(p) \subset W_{i, j} \cap T_{Q_j, r_j}\partial \Omega^{-}_j.$  This is a contradiction.  The other equation follows identically.

Now, we claim that $A \cap B_1(0) = \Sigma_{P_{\infty}} \cap B_1(0).$  Suppose not, then there exists a point, $p \in \Sigma_{P_{\infty}} $ with $p \not \in A$ or there exists a point, $Q \in A$ such that $Q \not \in \Sigma_{P_\infty}.$  In the former case, suppose the $\text{dist}(Q, A) > \delta.$  Then, $B_{\delta}(p)$ must intersect at least $2$ contiguous connected components, $W_{i, \infty}, W_{j, \infty}$.  Since they are contiguous, the sign of $P_{\infty}$ must be positive on one and negative on the other.  This contradicts Equation \ref{A equals P}.  Similarly, if there exists a point, $Q \in A$ such that $Q \not \in \Sigma_{P_\infty}$ then there exists a ball $B_{\delta}(Q)$ which intersects both $\Omega^{\pm}_{\infty}$ but which is contained in a single $W_{i, \infty}$.  This also contradicts Equation \ref{A equals P}.

However, if $P_{\infty}$ is $(n-2)$-symmetric with degree $\ge 2$, then $\Sigma_{P_{\infty}}$ does not divide $\RR^n$ into two connected components.  This contradicts our assumption that $A= \Sigma_{P_{\infty}}$ was the mutual boundary of a pair of two-sided NTA domains with constant $2M_0.$  Therefore, such a constant, $0< c= c(M_0, \Gamma, \alpha)$ must exist.

Without loss of generality, we assume (1) holds.  By Lemma \ref{porous set bounds} we may find a radius, $0< r = r(M_0, \Gamma, \alpha)$ such that $\mathcal{H}^n(B_r(T_{Q, r}\partial \Omega^{\pm})) < \frac{1}{20}c(\alpha, M_0, \Gamma)$.  Now, consider 
$$p \in \bigcup \{W_i : P>0 \text{  on  } W_i \} \cap B_1(0) \setminus B_r(T_{Q, r}\partial \Omega^{\pm}).$$
By Lemma \ref{growth control}, $|T_{Q, r}v(p)| \ge c'$ for a constant, $c' = c'(M_0, \Gamma, \alpha)$.  Thus
\begin{align*}
\int_{B_1(0)} |T_{Q, r}v - P|^2 dV & \ge \int_{B_1(0) \cap T_{Q, r}\partial \Omega^{-} \cap \bigcup_i \{W_i : P>0 \text{  on  } W_i\}  } |T_{Q, r}v - P|^2 dV \\
& \ge \frac{19}{20}c(\alpha, M_0, \Gamma)c'(\alpha, M_0, \Gamma)^2.
\end{align*}
If (2) holds, an identical argument with signs switched proves the claim.
\end{proof}

\begin{rmk}\label{linear approx rmk}
The argument above can be modified to show that there is an $0<\epsilon'$ such that if $Q \in \partial \Omega$ but $Q \not \in \mathcal{S}^{n-3}_{\epsilon', r_0}$, then $Q \not \in \mathcal{S}^{n-2}_{\epsilon', r_0}.$  Indeed, if $Q \not \in \mathcal{S}^{n-3}_{\epsilon', r_0}$, then there exists a radius, $r_0 \le r$, and an $(n-2)$-symmetric function, $P$, such that $||T_{Q, r}v - P||^2_{L^2(B_1(0))} \le \epsilon'$.  However, by taking $\epsilon' < \epsilon(\alpha, M_0, \Gamma)$ in Lemma \ref{epsilon regularity}, we see that $P$ must be $(n-1)$-symmetric.  
\end{rmk}

\section{Appendix A}

The purpose of this section is to justify Lemma \ref{Minkowski bounds on limit set}.  We use the language of porous sets.  For a non-empty set $E \subset \RR^n$, $x \in E$, and radius $0<r$, we write
\begin{align}
P(E, x, r) = \sup\{0, h: h>0, B_h(y) \subset B_r(x) \setminus E \text{ for some } y \in B_r(x) \}.
\end{align}

For $\alpha >0$, we say that $E$ is \textit{$\alpha$-porous} if 
\begin{align}
\liminf_{r \rightarrow 0}\frac{P(E, x, r)}{r} > \alpha
\end{align}
for all $x \in E$.  

We shall say that $E$ is \textit{$\alpha$-porous down to scale $r_0$} if  
\begin{align}
\frac{P(E, x, r)}{r} > \alpha
\end{align}
for all $x \in E$ and for all $r_0 \le r$.

\begin{rmk} \label{NTA porous}
By definition, for $\Omega^{\pm} \in \mathcal{D}(n, \alpha, M_0)$, the boundary, $\partial \Omega^{\pm},$ is $\frac{1}{M_0}$-porous.  Similarly, $B_{r}(\partial \Omega^{\pm})$ is $\frac{1}{2M_0}$-porous down to scale $r_0 = 2rM_0.$
\end{rmk}

\begin{lem}\label{porous set bounds}
Let $E \subset \RR^n$ be a non-empty, bounded set, $E \subset [0, 1]^{n}$ with $0 \in E$.  If $E$ is $\alpha$-porous down to scale $r_0 \ll 1$, then there is a $k = k(\alpha)$, $k' = k'(n)$, and $N \le \frac{-1}{k + k'}\log_2(r_0)$
\begin{align*}
\text{Vol}(E) \le \left(1 - \frac{1}{2^{k + k'(n)}}\right)^N.
\end{align*}
Moreover, there exists an $0< \epsilon = \epsilon (\alpha, n)$ and a constant $c(n, \alpha)$ such that
\begin{align*}
\mathcal{M}^{n-\epsilon}_{r_0}(E) \le \left(1-c \right)^N
\end{align*}
\end{lem}

\begin{proof}
Let $\{Q^i_j\}_j$ be the collection of dyadic sub-cubes, $Q^i_j \subset [0, 1]^n$ with $\ell(Q^i_j) = 2^{-i}$.  Let $k \in \mathbb{N}$ be the smallest number such that $2^{-k} \le \alpha$.  Note that for any $y \in [0, 1]^n$ with $B_{\frac{\alpha}{2}}(y) \subset [0, 1]^n$, there exists a dyadic cube, $Q^{k+k'(n)}_j \subset B_{\frac{\alpha}{4}}(y)$, where is the smallest integer such that $k'(n) \ge 2+ \frac{1}{2} \log_2(n).$  Let $\frac{1}{2}Q^{i}_j$ denote an axis-parallel cube with the same center as $Q^i_j$, but side length half that of $Q^i_j$.

Now, we apply the standard argument.  Tile $[0, 1]^n$ by $Q^{k + k'(n)}_j.$ By our porosity assumption, there exists a $Q^{k + k'(n)}_{j'}$ which does not intersect $E.$  Thus 
\begin{align*}
\text{Vol}(E) & \le \sum_{j \not = j'} \text{Vol}(Q^{k + k'(n)}_j)\\
& \le (2^{(k + k'(n))n} - 1)2^{(-k-k'(n))n}\\
& \le \left(1 - \frac{1}{2^{k + k'(n)}}\right).
\end{align*}

Now, within each $Q^{k + k'(n)}_j$ which intersects $E$, either $E$ intersects $\frac{1}{2}Q^{k + k'(n)}_j$, or it doesn't.   If $E \cap \frac{1}{2}Q^{k + k'(n)}_j = \emptyset$, then we tile $Q^{k + k'(n)}_j$ by cubes, $\{Q^{2(k + k'(n))}_{\ell}\}_\ell$ and overestimate
\begin{align*}
\text{Vol}(E \cap Q^{k + k'(n)}_j ) & \le \sum_{\ell: Q^{2(k + k'(n))}_\ell \cap (E \cap Q^{k + k'(n)}_j) \not \emptyset} \text{Vol}(Q^{2(k + k'(n))}_\ell)\\
& \le (2^{2(k + k'(n))n} - 1)2^{-2(k+k'(n))n} \text{Vol}(Q^{k + k'(n)}_j)\\
& \le \left(1 - \frac{1}{2^{k + k'(n)}}\right) \text{Vol}(Q^{k + k'(n)}_j).
\end{align*}

If $E \cap \frac{1}{2}Q^{k + k'(n)}_j \not= \emptyset$, then there exists a ball, $B_{2^{-k-k'(n) -1}}(x) \subset Q^{k + k'(n)}_j$, centered on $x \in E$.  By our porosity assumption and choice of $k'(n)$, we can still tile $Q^{k + k'(n)}_j$ by $Q^{2(k + k'(n))}_\ell$ and be guaranteed that at least one such sub-cube does not intersect $E \cap Q^{k + k'(n)}_j$.  Thus, we overestimate in the same manner as above.  

We can continue, inductively, only stopping at the first $N$ such that $2^{-(N+1)(k + k'(n))} < r_0$.  This gives the desired bound
\begin{align*}
\text{Vol}(E) \le \left(1 - \frac{1}{2^{k + k'(n)}}\right)^N.
\end{align*}

Taking a bit more care, we can actually improve these estimates.  Let $0< \epsilon = \epsilon (\alpha, n)$ be such that
\begin{align*}
\left(1 - \frac{1}{2^{k + k'(n)}}\right) < \left(2^{\epsilon(k + k'(n))} - \frac{1}{2^{(k + k'(n))(n-\epsilon)}}\right) < 1.
\end{align*}
Then, we bound $\mathcal{M}^{n-\epsilon}_{r_0}(E)$ as follows
\begin{align*}
\mathcal{M}^{n-\epsilon}_{r_0}(E) & =\inf \{\sum_{i}r^{n-\epsilon} : x_i \in E, ~ r_0 \le r, E \subset \cup_i B_r(x_i) \}\\
& \le \sum_j \ell(Q^{N(k+ k'(n))})^{n -\epsilon}\\
& \le \left(2^{\epsilon(k + k'(n))} - \frac{1}{2^{(k + k'(n))(n-\epsilon)}}\right)^N.
\end{align*}
\end{proof}

As immediate corollaries, we have the following statements.

\begin{cor}\label{minkowski dimension of porous sets}
If $E \subset \RR^n$ is $\alpha$-porous, then there exists an $0< \epsilon = \epsilon(\alpha, n)$ such that $\overline{dim_{\mathcal{M}}}(E) \le n-\epsilon$.
\end{cor}

\begin{proof}
Recall that $\overline{dim_{\mathcal{M}}}(E) = \inf\{s: \mathcal{M}^{*, s}(E) = 0\}$ and that $\mathcal{M}^{*, s}(E) = \limsup_{r_0 \rightarrow 0} \mathcal{M}^{n-\epsilon}_{r_0}(E).$

By taking $0< \epsilon$ to be as in Lemma \ref{porous set bounds}, we have that
\begin{align*}
\mathcal{M}^{n-\epsilon}_{r_0}(E)  & \le \left(2^{\epsilon(k + k'(n))} - \frac{1}{2^{(k + k'(n))(n-\epsilon)}}\right)^N\\
& \le (1-c)^N,
\end{align*}  
where $c = c(\alpha, n, \epsilon)$ and $N= N(\alpha, n, r_0)$, as in the previous lemma.  Thus, letting $r_0 \rightarrow 0$, $N \rightarrow \infty$ and we have that $\mathcal{M}^{n-\epsilon}(E) = 0.$
\end{proof}

Recalling Remark \ref{NTA porous}, Corollary \ref{minkowski dimension of porous sets} gives Lemma \ref{Minkowski bounds on limit set}.  

\begin{cor}
Let $\Sigma \subset \RR^n$ be the mutual boundary of a pair of unbounded two-sided NTA domains with NTA constant $1< M_0$. Then, there is an $0< \epsilon = \epsilon(M_0, n)$ such that $\overline{dim_{\mathcal{M}}}(E) \le n-\epsilon$.
\end{cor}

\bibliography{references}
\bibliographystyle{alpha}

\end{document}